\documentclass[11pt]{amsart}
\usepackage{amstext,amsfonts,amssymb,amscd,amsbsy,amsmath,verbatim, mathrsfs, fullpage, graphicx,amsthm,xcolor}
\usepackage[alphabetic,abbrev,lite]{amsrefs} 
\usepackage{tikz}
\usetikzlibrary{matrix,arrows,decorations.pathmorphing,cd}

\usepackage{hyperref}
\usepackage[capitalise]{cleveref}

\newcommand{\shrinkmargins}[1]{
  \addtolength{\textheight}{#1\topmargin}
  \addtolength{\textheight}{#1\topmargin}
  \addtolength{\textwidth}{#1\oddsidemargin}
  \addtolength{\textwidth}{#1\evensidemargin}
  \addtolength{\topmargin}{-#1\topmargin}
  \addtolength{\oddsidemargin}{-#1\oddsidemargin}
  \addtolength{\evensidemargin}{-#1\evensidemargin}
  }

\shrinkmargins{.7}

\DeclareMathOperator{\Hom}{Hom}

\DeclareMathOperator{\Spec}{Spec}

\DeclareMathOperator{\Gal}{Gal}
\DeclareMathOperator{\Frob}{Frob}

\DeclareMathOperator{\Res}{Res}

\DeclareMathOperator{\Conf}{Conf}

\DeclareMathOperator{\Br}{Br}

\DeclareMathOperator{\disc}{disc}

\newcommand{\field}[1]{\mathbb{#1}}

\newcommand{\Q}{\field{Q}}

\newcommand{\Z}{\field{Z}}
\newcommand{\A}{\field{A}}
\newcommand{\F}{\field{F}}

\newcommand{\C}{\field{C}}
\renewcommand{\P}{\field{P}}

\newcommand{\KK} {\mathcal{K}}
\newcommand{\ra}{\rightarrow}

\newcommand{\Gm}{\mathbb{G}_m}

\newcommand{\tensor} {\otimes}

\newcommand{\beq}{\begin{displaymath}}
\newcommand{\eeq}{\end{displaymath}}
\newcommand{\beqn}{\begin{equation}}
\newcommand{\eeqn}{\end{equation}}
\newcommand{\jordan}[1]{{\color{violet} ($\spadesuit$ Jordan: #1)}}

\theoremstyle{plain}
\newtheorem{thm}{Theorem}[subsection]
\newtheorem{prop}[thm]{Proposition}
\newtheorem{cor}[thm]{Corollary}
\newtheorem{lem}[thm]{Lemma}

\theoremstyle{definition}
\newtheorem{defn}[thm]{Definition}

\newtheorem{exmp}[thm]{Example}

\theoremstyle{remark}
\newtheorem{rem}[thm]{Remark}

\numberwithin{equation}{section}

\title{Averages of Arithmetic Functions over Conductors of Function Fields}

\author{Jordan Ellenberg}

\address{Department of Mathematics, University of Wisconsin, Madison, WI 53706}
\email{ellenber@math.wisc.edu}

\author{Mark Shusterman} 
\thanks{The second author is The Dr. A. Edward Friedmann Career Development Chair in Mathematics}

\address{Mark Shusterman, Faculty of Mathematics and Computer Science, Weizmann
Institute of Science, 234 Herzl Street, Rehovot 76100, Israel.}
\email{mark.shusterman@weizmann.ac.il}

\begin{document}

\begin{abstract}

For a finite group $G$ and a sufficiently large (but fixed) prime power $q$ coprime to $G$ we obtain asymptotics for the number of regular Galois extensions $L/ \F_q(t)$, with $\Gal(L/\F_q(t)) \cong G$, ramified at a single place of $\F_q(t)$, thus making progress on a positive characteristic analog of the Boston--Markin conjecture.
We also obtain similar results for other arithmetic functions of the product of places of $\F_q(t)$ ramified in $L$, and for more general one-variable function fields over $\F_q$ in place of $\F_q(t)$. 

Some of our proofs make crucial use of a series of recent breakthroughs by Landesman--Levy, as well as a new `vanishing of stable homology in a given direction' result for representations of braid groups arising from braided vector spaces. 
Other inputs include a study of (rings of coinvariants of) braided vector spaces associated to racks with $2$-cocycles, a connection between convolution of arithmetic functions and direct sums of braided vector spaces, and a Goursat lemma for racks.

\end{abstract}

\maketitle

{\em The authors dedicate this paper to the memory of Nigel Boston (1961-2024), whose outlook and ideas are foundational to the modern study of statistics of field extensions, and who greatly influenced both of us.}

\section{Introduction}

\subsection{Homological Stability}

It is by now well understood that purely topological theorems about homological stability for moduli spaces can (in many cases) be used to prove theorems in arithmetic statistics over global function fields. 
For example, a series of works starting with \cite{evw}, and culminating in \cites{llSplit, ll, llNew}, was concerned mostly with computing the homology of Hurwitz spaces (and their variants). In some recent works, such as \cites{etw, el, BDPW, MomentsOfLfunctions}, it has become clear that there is utility in proving homological stability theorems where instead of studying the homology of the space itself, one studies the homology of (non-constant) local systems on that space. The moduli spaces in question are often $K(\pi,1)$, so this reduces to computing the homology of the discrete group $\pi$ with some of its representations as coefficients.

For instance, $\Conf^n(\C)$ - the configuration space of $n$ unordered distinct points in the plane is a $K(B_n,1)$ where $B_n$ is the Artin braid group on $n$ strands given by the presentation
\[
B_n = \langle \sigma_1, \dots, \sigma_{n-1} : \sigma_i \sigma_j = \sigma_j \sigma_i \text{ for } i > j+1, \ \sigma_i \sigma_{i+1} \sigma_i = \sigma_{i+1} \sigma_i \sigma_{i+1} \text{ for } i < n-1 \rangle.
\]

The representations of $B_n$ whose homological stability we study (with applications to arithmetic statistics in mind) are of the form $V^{\otimes n}$ for a braided vector space $V$ - see \cref{def:braidedobject} for the notion of a braided vector space, and \cref{BraidedToNaturoid} for the action of $B_n$ on $V^{\otimes n}$.


A homological stability theorem in this setting is the assertion that, for each nonnegative integer $i$, the group $H_i(B_n, V^{\tensor n})$ is independent of $n$ once $n$ is large enough relative to $i$ - for arithmetic applications one typically wants `large enough' to mean `larger than a constant multiple of $i$'.
This constant is called the slope of stability. An example is \cite{el}*{Theorem 4.2.6} which applies under suitable assumptions on $H_0(B_n, V^{\otimes n})$.
As another example, \cite{etw} produces highly non-trivial upper bounds on the dimension of $H_i(B_n, V^{\tensor n})$.  

Many of the braided vector spaces that play a role in arithmetic arise from racks - see \cref{RackDef} and \cref{RacksToBraidedVS}.
For these, \cite{llNew} proves homological stability and determines the stable value (in many cases).
In this paper we obtain homological stability, indeed vanishing, for certain braided vector spaces; the braided vector spaces to which this result are quite general and particular need not arise from racks.  That said, the particular examples to which we apply the results in the latter part of the paper are related to racks (though often in a twisted way,  using the notion of a rack with 2-cocycle.)

Our most general vanishing of homology result is \cref{th:0controlledvanishing}, a special case of which is the following.

\begin{thm} \label{SamplerVanishingStableHomology}

Suppose that there exists a nonnegative integer $d$ such that $H_0(B_n, V^{\otimes n}) = 0$ for every $n > d$. Then
$
H_i(B_n, V^{\tensor n}) = 0
$
for $n > (d+2)i + d$.

\end{thm}

A significant feature of \cref{th:0controlledvanishing}, not reflected in \cref{SamplerVanishingStableHomology}, is that it allows one to obtain homological stability `in a given direction' - a form of stability considered also in \cite{llNew}.
This is accomplished here by keeping track of a grading on $V$.
Other notable features are that that \cref{th:0controlledvanishing} applies more generally to surface braid groups, and that the bounds it provides are not only effective but explicit.

In order to apply \cref{SamplerVanishingStableHomology}, 
we study the ring of coinvariants of $V$ - see \cref{DefAlgCoinvs}.
Much is known about this ring in case $V$ arises from a rack - starting from works of Conway--Parker and Fried--V\"olklein, all the way to more recent works such as \cites{evw, el, shus}.
We extend this theory to braided vector spaces arising from $2$-cocycles on racks in \cref{CoinvsAlg}. 

We wonder whether there is an analog of \cref{SamplerVanishingStableHomology} for automorphism groups of free groups, or for mapping class groups of connected closed orientable surfaces.

\subsection{Arithmetic Applications}


Let $G$ be a finite group,
let $p$ be a prime number not dividing $|G|$, and let $q$ be a power of $p$.
Every Galois extension of $\F_q(t)$ with Galois group $G$ is thus tamely ramified.

Let $R \subseteq G$ be a nonempty conjugacy-closed generating set consisting of nontrivial elements of $G$ such that $r^q \in R$ for every $r \in R$.
Let $R_1, \dots, R_k$ be the conjugacy classes of $G$ whose (disjoint) union is $R$.

Fix a separable tame closure $\overline{\F_q(t)}$ of $\F_q(t)$, a place of it lying over each (finite) place of $\F_q(t)$, and a (topological) generator of the corresponding (tame) inertia subgroup of $\Gal(\overline{\F_q(t)}/\F_q(t))$.

For nonnegative integers $n_1, \dots, n_k$ denote by $\mathcal E_q^R(G;n_1, \dots, n_k)$ the family of regular Galois extensions $L$ of $\F_q(t)$ inside $\overline{\F_q(t)}$, equipped with an isomorphism $\varphi \colon \Gal(L/\F_q(t)) \to G$ such that the following conditions are met.

\begin{itemize}
    
\item The infinite place of $\F_q(t)$ splits completely in $L$.

\item The chosen generators of inertia subgroups in $\Gal(\overline{\F_q(t)}/\F_q(t))$ of (finite) places ramified in $L$ are mapped to $R$ under the composition of group homomorphisms
\[
\begin{tikzcd}
	{\operatorname{Gal}(\overline{\mathbb{F}_q(t)}/\mathbb{F}_q(t))} & {\operatorname{Gal}(L/\mathbb{F}_q(t))} & G.
	\arrow[from=1-1, to=1-2]
	\arrow["\varphi", from=1-2, to=1-3]
\end{tikzcd}
\]

\item For each $1 \leq j \leq k$ the sum of the degrees of all the (finite) places of $\F_q(t)$ whose inertia generators map to $R_j$ is $n_j$.

\end{itemize}

\begin{rem}

The extensions in $\mathcal E_q^R(G;n_1, \dots, n_k)$ correspond to Galois $G$-covers of $\P^1$ over $\F_q$ by smooth projective geometrically connected curves branched at $n = n_1 + \dots + n_k$ geometric points in $\mathbb P^1$ with $n_i$ geometric branch points in $\mathbb P^1$ of type $R_i$ for every $1 \leq i \leq k$, and $|G|$ many $\F_q$-points over $\infty \in \mathbb P^1$. 
    
\end{rem}

\begin{rem}

Raising to $q$th power permutes $\{R_1, \ldots, R_k\}$; if $R_i$ and $R_j$ are in the same cycle of this permutation for some $1 \leq i < j \leq k$, then $\mathcal E_q^R(G;n_1,\ldots, n_k)$ can only be nonempty if $n_i = n_j$.
    
\end{rem}

As an application of the `homological stabilization in a given direction' results of \cite{llNew}, that work obtains, up to a count of connected components, an asymptotic for $\# \mathcal E_q^R(G;n_1, \dots, n_k)$ as $n = n_1 + \dots + n_k \to \infty$ with $q$ sufficiently large depending on $|G|$, thus going way beyond what is known for the analogous problem over $\Q$ in place of $\F_q(t)$.
In this paper we obtain enhancements and variants of this asymptotic counting, by studying the average over $L \in \mathcal E_q^R(G;n_1, \dots, n_k)$ of arithmetic functions of the conductor of $L$.
The guiding principal is that as $n \to \infty$ and $L$ ranges over $\mathcal E_q^R(G;n_1, \dots, n_k)$, the conductors $f_L$ should behave like a random multiset of degree $n$ squarefree polynomials in $\F_q[t]$.

The conductor of $L$, which we denote by $f_L$, is the radical of the discriminant of $L$, or equivalently the (monic squarefree) polynomial in $\F_q[t]$ obtained by multiplying all the monic irreducible polynomials in $\F_q[t]$ that ramify in $L$.
It also possible to think of $f_L$ as (defining) the branch locus of the branched cover of $\P^1$ over $\F_q$ by the smooth projective geometrically connected curve corresponding to $L$. 

\subsubsection{Irreducible Polynomials}

Perhaps the most natural function to average over the extensions $L/\F_q(t)$ in $\mathcal E_q^R(G;n_1, \dots, n_k)$ is the indicator function $\mathbf{1}_{\textup{irr}}$ of $f_L$ being irreducible, as suggested by Boston and Markin. \cite{BM} also observed that, for $f_L$ to be irreducible, all of the $n_j$ but one have to be $0$.
For the purpose of averaging $\mathbf{1}_{\textup{irr}}$ we may therefore assume that $R$ is a (single) conjugacy class of $G$.
In particular $G$ is generated by a single conjugacy class, or equivalently by \cite{nsw}*{Theorem 10.2.6} - the abelianization of $G$ is cyclic.

In analytic number theory, given an arithmetic function with an asymptotic estimate for its sum, one is often interested in asymptotics for the sum of this function over the prime numbers (or monic irreducible polynomials). In our case the value of the function on a (monic squarefree) polynomial $g \in \F_q[t]$ is the number of $L \in \mathcal E_q^R(G;n_1, \dots, n_k)$ with $f_L = g$.

Regardless of the basic method being employed to estimate the sum of an arithmetic function over the primes - be it a (sophisticated) sieve, an identity such as that of Vaughn, or a convolution identity involving the M\"obius function, an estimate for the sum of the function over the multiples of a given integer (or polynomial) which is allowed to grow with the range of summation is usually available. 
In our case, such an estimate is not provided by the results of Landesman--Levy in their present form, yet we manage to obtain the desired asymptotic.

\begin{thm} \label{MainArithApp}

For $q$ sufficiently large depending on $|G|$, as $n \to \infty$ we have
\begin{equation*} 
\#\{L \in \mathcal E_q^R(G;n) : f_L \text{ is irreducible}\} \sim \frac{q}{(q-1)n} \# \mathcal E_q^R(G;n)    
\end{equation*}
with a power saving error term.

\end{thm}

In particular, the odds that $f_L$ is irreducible are asymptotically the ones of a random squarefree polynomial to be irreducible.

\cref{MainArithApp} makes progress toward conjectures from \cite{BM}, as restated in the first part of \cite{shus}*{Conjecture 1.9}.
The latter work also proves a large finite field version of \cref{MainArithApp}.

A version of the problem not involving counting, only existence of $G$-extensions, and allowing ramification at a small number of primes, has been considered over global fields for instance in \cites{BSEF, BSS, DW, Ho, JR, KS, KNS, MU, No, Pl}.
An upper bound for the version of our problem over $\Q$ with $G = S_4$ has been obtained in \cite{BG}*{Theorem 7.11}, and a lower bound for $G = S_3$ (respectively, $G = S_4$) with at most $3$ (respectively, $8$) ramified primes has been obtained \cite{TT}.

The key inputs in our proof of \cref{MainArithApp} are \cref{SamplerVanishingStableHomology}, and \cite{llNew}*{Theorem 1.3.5} - a representation stability result which relies on, among many other things, a large monodromy theorem from \cite{shus}.

\cite{shus}*{Conjecture 1.4} makes a prediction for the counts over $\Q$ analogous to that from \cref{MainArithApp}, based on established heuristics, yet it does not commit to an explicit constant in the asymptotics, only predicting the rate of growth.
Having \cref{MainArithApp}, one can be more confident about making a conjecture over $\Q$ with the most natural constant in the asymptotics - the value of the Riemann zeta function at $2$ (and no correction factor). 

\cref{MainArithApp} (and other results we obtain) have implications towards refinements of the Cohen--Lenstra heuristics where one cares about the factorization of the discriminant of the function field whose class group is studied.  

There are some applications of the ideas around \cref{MainArithApp} that we leave for future works to pursue.
One such application, following \cite{els}, is to show that the asymptotic proportion among quadratic Dirichlet $L$-functions over $\F_q[t]$ with irreducible conductor of those having a zero at the central point $s = 1/2$ (or at any other given point) is $0$.
Another potential application, following \cites{el,llNew}, would be to the distribution of Selmer groups of prime quadratic twists of an abelian variety over a global function field, and as a standard consequence, a result toward the minimalist conjecture for the rank of the Mordell--Weil group in this family.

\subsubsection{An Aside on Random Profinite Groups}

\cite{BE} suggests, among other things, the following random model of profinite groups.
For each positive integer $n$ consider the average over all the degree $n$ irreducible polynomials $P \in \F_q[t]$ of the Dirac measure at the (maximal pro-prime-to-$q$ quotient of the) \'etale fundamental group of $\operatorname{Spec} \F_q[t,P^{-1}]$ - the affine line punctured at the roots of $P$. \cite{BE} postulates that as $n \to \infty$ these (averaged) probability measures converge (weakly) to a probability measure on a suitable space of profinite groups.
The assumption that $P$ is irreducible (or at least, has a bounded number of irreducible factors) is crucial for if dispensed with, a limiting probability measure as $n \to \infty$ will not exist.

Arithmetic topology furnishes us with an analogous random model of (profinite) $3$-manifold groups with boundary - one selects a random knot in the $3$-sphere and takes the profinite completion of the fundamental group of the complement in the $3$-sphere of the interior of a tubular neighborhood of the knot. The analogy becomes particularly close if the knot is chosen as the closure of a random braid in $B_n$ that maps to an $n$-cycle in $S_n$, and then the limit as $n \to \infty$ is taken.

\cite{SaWo} determines such a limiting probability distribution for the Dunfield--Thurston model of a random closed $3$-manifold, and \cite{SaWoo} studies measures on profinite groups (and other categories) much more generally, emphasizing the key role of the $G$-moment in studying such distributions.
The $G$-moment is the expectation of the number of surjections from a random profinite group in our model onto (the given finite group) $G$.
\cref{MainArithApp} makes progress towards the computation of the $G$-moment for the random model of profinite groups proposed in \cite{BE}.

\subsubsection{Other factorization functions}

Here a factorization function is a function on squarefree monic polynomials $f \in \F_q[t]$ determined by the (multiset of) degrees of the irreducible factors of $f$.
For example, we denote by $\omega(f)$ the number of (distinct monic) irreducible factors of $f$, and define the M\"obius function by $\mu(f) = (-1)^{\omega(f)}$.

\begin{thm} \label{MobMainThm}

For $q$ large enough depending on $|G|$, as $n = n_1 + \dots + n_k \to \infty$ we have cancellation with a power saving in
\[
\sum_{L \in \mathcal E_q^R(G; n_1, \dots, n_k)} \mu(f_L).
\]

\end{thm}

This makes progress on the second part of \cite{shus}*{Conjecture 1.9}.
It is notable that neither our methods, nor other methods such as those of Landesman--Levy, seem to provide an asymptotic with power saving for $\# \mathcal E_q^R(G; n_1, \dots, n_k)$ when $k > 1$. The difficulty is in understanding the asymptotic number of connected components of Hurwitz spaces when one of the $n_1, \dots, n_k$ is small. 
Progress on this problem, and thus also on the Malle conjecture, has recently been made in \cite{Sant}.

Our methods allow us to treat also other factorization functions such as the generalized divisor functions 
\begin{equation*} 
d_m(f) = \# \{(f_1, \dots, f_m) \in \F_q[t]^m : f_1, \dots, f_m \textup{ are monic, } \ f_1 \cdots f_m = f\} =m^{\omega(f)} 
\end{equation*}
of a monic squarefree $f \in \F_q[t]$.

In \cref{HigherGenus} these results are obtained for function fields of arbitrary smooth projective curves over $\F_q$, not only for $\P^1$.

\subsubsection{Arithmetic functions beyond factorization}

The following result, valid in arbitrary positive characteristic, extends the $p \neq 2$ case of \cref{MobMainThm}.

\begin{thm} \label{CharDiscCancThm}

For $q$ large enough depending on $|G|$, and a nontrivial character $\chi \colon \F_q^\times \to \C^\times$, as $n = n_1 + \dots + n_k \to \infty$ we have cancellation with a power saving in
\[
\sum_{L \in \mathcal E_q^R(G; n_1, \dots, n_k)} \chi(\disc(f_L))
\]
where $\disc(f)$ is the (nonzero) discriminant of a (squarefree) polynomial $f \in \F_q[t]$.

\end{thm}

Indeed, if $q$ is not a power of $2$ and $\chi$ is quadratic, we have $\chi(\disc(f)) = (-1)^{\deg f} \mu(f)$, so in this case \cref{CharDiscCancThm} reduces to \cref{MobMainThm}.
When applied to all characters, \cref{CharDiscCancThm} implies that as $L$ ranges over $\mathcal E_q^R(G; n_1, \dots, n_k)$, the discriminant of $f_L$ asymptotically equidistributes in $\F_q^\times$.
\cref{CharDiscCancThm} is proven in \cref{SecMulCharDiscGenZero}, and improved power savings for some particular choices of $G$ and $R$ are obtained using \cref{pr:dexamples}.

Suppose now that $q$ is odd.
For a monic irreducible $P \in \F_q[t]$ and $f \in \F_q[t]$ not divisible by $P$ the Legendre symbol $(\frac{f}{P})$ is defined to be $1$ if $f$ is a square modulo $P$ and $-1$ otherwise. For a monic squarefree $g \in \F_q[t]$ coprime to $f$ the Jacobi symbol $(\frac{f}{g})$ is defined to be the product of the Legendre symbols of $f$ over the monic irreducible factors of $g$.

\begin{thm} \label{LegendreThmConductors}

For $q$ large enough depending on $|G|$, as $n \to \infty$ we have 
\[
\sum_{L \in \mathcal E_q^R(G; n)} \ \sum_{gh = f_L} \left( \frac{g}{h} \right) \sim 2 \# \mathcal E_q^R(G;n)
\]
where $g,h \in \F_q[t]$ are monic (squarefree, and coprime), with a power saving error term.

\end{thm}

The arithmetic function in \cref{LegendreThmConductors} has been considered in \cite{anh}, appears in the study of multiple Dirichlet series, and serves as an illustration of the kind of arithmetic functions our methods can handle. We intend to consider additional arithmetic function in future works.

The power savings in \cref{MobMainThm} and \cref{CharDiscCancThm} are explicit, and we know exactly how large $q$ has to be for our proof to work.
On the other hand, the power savings in \cref{MainArithApp} and \cref{LegendreThmConductors} are not explicit, and neither is the `large enough' condition on $q$.
The reason is that \cref{SamplerVanishingStableHomology} provides an explicit slope while the homological stability results of Landesman--Levy, which are crucial for our proofs of \cref{MainArithApp} and \cref{LegendreThmConductors}, do not provide such information in their present form, even though with some additional work it can perhaps be obtained.

\subsection{Topology Meets Arithmetic}

The space $\Conf^n$ descends to a finite-type scheme over $\operatorname{Spec} \Z$, whose $\F_q$-points are the monic squarefree degree $n$ polynomials in $\F_q[t]$,
and for various braided vector spaces $V$, the representation $V^{\otimes n}$ of $B_n$ descends to an \'etale sheaf on (an open subset of) $\operatorname{Spec} \Z$.
The trace functions of some of these sheaves are the ones whose sums we estimate in our theorems. To do this, we apply the Grothendieck--Lefschetz trace formula and Deligne's upper bound on the eigenvalues of Frobenius acting on the (compactly supported, \'etale)  cohomology of our sheaves. 
We then need to prove vanishing of some cohomology groups, upper bound the dimensions of those cohomology groups that do not vanish, or compare to some other cohomology groups on which the trace of Frobenius has already been computed.

For example, in order to prove \cref{MobMainThm} we study the one-dimensional vector space braided by negation, whose trace function is $f \mapsto (-1)^{\deg f}\mu(f)$, and apply \cref{SamplerVanishingStableHomology} to it(s tensor powers). 
To prove \cref{MainArithApp} we introduce in \cref{PlainDirectSumEx} a (plain) direct sum operation on braided vector spaces, and consider the direct sum of the trivial (one-dimensional) braided vector space with the one used to study the M\"obius function.  
We show in \cref{ConvolutionCorrDirectSum} that the (plain) direct sum of braided vector spaces corresponds, under the function-sheaf dictionary, to Dirichlet convolution of arithmetic functions on polynomials in $\F_q[t]$.
We then conclude in \cref{WedgeArithmetized} that the value of the trace function of our $2$-dimensional braided vector space on (a monic squarefree polynomial) $f \in \Conf^n(\F_q)$ is $0$ if $f$ has an irreducible factor in $\F_q[t]$ of even degree, and $d_2(f)$ otherwise.

From the (function theoretic) perspective of analytic number theory, it is not at all clear why control of the behavior of this function on a (multi)set would tell us anything about irreducible polynomials (in that set). But once nonabelian harmonic analysis is employed, this function and the indicator function of monic irreducible polynomials turn out to involve the same representations of the braid group, so controlling the homology groups that govern this function gives information about irreducible polynomials as well.
It would be interesting to understand more generally when two arithmetic functions are `twinned' in this way.

The braiding of the aforementioned vector spaces, and of other ones whose trace functions are factorization functions, are involutions. Equivalently, the corresponding representations of $B_n$ are inflated from $S_n$ via the surjective group homomorphism $B_n \to S_n$ that sends $\sigma_i$ to the transposition $(i \ \ i+1)$ for $1 \leq i \leq n-1$.
The irreducible representations of $S_n$ that appear in the study of irreducible polynomials (on whose roots Frobenius acts by an $n$-cycle) are the wedge powers of the standard representation, so we study the homology of Hurwitz spaces twisted by these representations.
It would also be interesting to determine the homology of Hurwitz spaces twisted by other irreducible representations of $S_n$, and \cite{hmw} makes significant progress on this problem.

The arithmetic function in \cref{CharDiscCancThm} arises from a vector space whose braiding is given by scaling by a root of unity whose order is that of $\chi$.
The arithmetic function in \cref{LegendreThmConductors} can be viewed as a generalized convolution, and it arises from a generalized notion of a (not necessarily plain) direct sum of braided vector spaces that we study in \cref{Addability}.

All the braided vector spaces we consider arise from (cyclotomic) $2$-cocycles on racks, and so do other braided vector spaces of interest in arithmetic, such as those in \cite{SW}. It would therefore be interesting to see a relevant example of a braided vector space, say over a finite field, that does not arise from a $2$-cocycle on a rack.

Given a braided vector space $V$ over a finite field, for every $n$ we can choose a number field, and an open subset of the spectrum of its ring of integers to which $V^{\otimes n}$ descends. It is natural to wonder whether this choice can be made uniform in $n$, or even whether $V^{\otimes n}$ descends to an open subscheme of $\Spec \Z$ - this is what happens if $V$ comes from a rack that embeds into a group. 
Notable in this context is \cite{tomerkei}, whose results we would expect to imply a positive answer to this question for the braided vector spaces coming from a certain special class of finite racks called {\em keis}.

The paper is structured as follows.  In section~\ref{s:homologicalvanishing}, we prove Theorem~\ref{th:0controlledvanishing}, which does not require the machinery of the rest of the paper.  In section~\ref{s:sheaves}, we lay out the basic arithmetic geometry necessary to move between theorems about group cohomology and arithmetic counting problems.  In section~\ref{s:coarsebounds}, we recount the bounds on Betti numbers we will need in order to get coarse (but good enough) control of homology outside the stable range.  Section~\ref{s:racks}, concerning racks, $2$-cocycles on racks, and operations on racks and the vector spaces they span, is the longest part of the paper.  With future applications in mind, we have tried to lay out the basics of the theory in a unified way.  We particularly point out \cref{CoinvsAlg} that recasts some results of \cites{el, evw} in the generality of $2$-cocycles on racks and the context of the results obtained in this paper, \cref{RackifyingCocycle} that embeds braided vector spaces associated to $2$-cocycles into braided vector spaces associated to racks, and \cref{GoursatRacks} which is a version of Goursat's lemma for racks.  With this setup in place, the final three sections are devoted to proving arithmetic counting theorems for extension of function fields, including \cref{LegendreThmConductors}.

\section{Homological Vanishing}

\label{s:homologicalvanishing}

In this section, we prove 
the main homological stability theorem needed for our applications. 
We begin with a modicum of definitions and examples regarding braided objects in categories.

Let $\mathcal C$ be a monoidal category - a category with tensor products, and a unit $I$ that satisfy natural compatibilities.
For the applications in this paper, the category $\mathcal C$ will (typically) be the category of (graded) vector spaces over a field $\kappa$.
Another relevant example is the category of sets with direct (Cartesian) products as tensor products, and singletons as unit elements.


\begin{defn} \label{def:braidedobject}

Let $V$ be an object of $\mathcal C$, and let $T \colon V \otimes V \to V \otimes V$ be an isomorphism. 
We say that $T$ is a braiding of $V$, or that $V$ is braided (in case $T$ is implicit),  if
\[
(\mathrm{id}_V \otimes T) \circ (T \otimes \mathrm{id}_V) \circ (\mathrm{id}_V \otimes T) = (T \otimes \mathrm{id}_V) \circ (\mathrm{id}_V \otimes T) \circ (T \otimes \mathrm{id}_V)
\]
as morphisms from $V \otimes V \otimes V$ to $V \otimes V \otimes V$.
For objects $V,W$ of $\mathcal C$ with braidings $T_V, T_W$, a morphism $f \colon V \to W$ is a morphism of braided objects if $T_W \circ (f \otimes f) = (f \otimes f) \circ T_V$ in $\operatorname{Mor}(V \otimes V, W \otimes W)$. 

\end{defn}

With this definition, braided objects in $\mathcal C$ form a category.
Note that the braiding is part of the data of a braided object (namely, braided objects are objects with extra structure).

\begin{exmp} \label{TrivialBraidingUnit}

On the unit object $I$ of $\mathcal C$ we have the (trivial) braiding $I \otimes I \to I \otimes I$ obtained by identifying $I \otimes I$ with $I$ and using the identity morphism of $I$.  
    
\end{exmp}

\begin{defn} \label{PermBraidObj}
    
Let $V$ be an object of $\mathcal C$ with braiding $T_V \colon V \otimes V \to V \otimes V$. We say that $V$ is a permutational braided object if $T_V \circ T_V = \mathrm{id}_{V \otimes V}$.
    
\end{defn}

\begin{exmp} \label{ScalingBraidedVS}

Braided (finite-dimensional) vector spaces are the braided objects in the category of (finite-dimensional) vector spaces. 
For instance, given a field $\kappa$ and $\zeta \in \kappa^\times$ we have the one-dimensional vector space $\kappa_\zeta$ over $\kappa$ braided by $T_{\kappa_\zeta} = \zeta \cdot \mathrm{id}_{\kappa_\zeta \otimes \kappa_\zeta}$.
This braided vector space is permutational if and only if $\zeta \in \{1, -1\}$.
We will sometimes pick a nonzero (basis) vector $v_\zeta \in \kappa_\zeta$.

\end{exmp}

\begin{exmp} \label{IntroducingKwedge}

Over a field $\kappa$ of characteristic different from $2$ we have a $2$-dimensional permutational braided vector space $\kappa_\wedge$ with basis $v_1, v_{-1}$ whose braiding is given by
\[
v_1 \otimes v_1 \mapsto v_1 \otimes  v_1, \qquad
 v_1 \otimes v_{-1} \mapsto v_{-1} \otimes v_1, \qquad
v_{-1} \otimes v_1 \mapsto v_1 \otimes v_{-1}, \qquad
v_{-1} \otimes v_{-1} \mapsto - v_{-1} \otimes v_{-1},
\]
or equivalently, it is represented by the matrix
\[
\begin{pmatrix}
&1 &0 &0 &0 \\
&0 &0 &1 &0 \\
&0 &1 &0 &0 \\
&0 &0 &0 &-1
\end{pmatrix}
\]
in the basis $v_1 \otimes  v_1, v_1 \otimes v_{-1}, v_{-1} \otimes v_1, v_{-1} \otimes v_{-1}$ for the vector space $\kappa_\wedge \otimes \kappa_\wedge$ over $\kappa$.
    
\end{exmp}

\begin{exmp} \label{LegendreBraidVS}

Over a field $\kappa$ of characteristic different from $2$ we have a $2$-dimensional braided vector space $\kappa_{\pm}$ with basis $v, \underline v$ whose braiding is given by
\[
v \otimes v \mapsto v \otimes  v, \qquad
 v \otimes \underline v \mapsto \underline v \otimes v, \qquad
\underline v \otimes v \mapsto - v \otimes \underline v, \qquad
\underline v \otimes \underline  v \mapsto \underline  v \otimes \underline v, \qquad
\]
or equivalently, it is represented by the matrix
\[
\begin{pmatrix}
&1 &0 &0 &0 \\
&0 &0 &-1 &0 \\
&0 &1 &0 &0 \\
&0 &0 &0 &1
\end{pmatrix}
\]
in the basis $v \otimes  v, v \otimes \underline v, \underline v \otimes v, \underline v \otimes \underline  v$ for the vector space $\kappa_{\pm} \otimes \kappa_{\pm}$ over $\kappa$.
    
\end{exmp}

\begin{defn} \label{BraidedToNaturoid}

To a braided object $V$ in $\mathcal C$ we functorially associate the sequence $\{V^{\otimes n}\}_{n \geq 0}$ of objects in $\mathcal C$ with an action of the braid group $B_n$, namely a group homomorphism $B_n \to \operatorname{Aut}(V^{\otimes n})$ given by 
\begin{equation*} 
\sigma_i \mapsto \mathrm{id}_{V^{\otimes (i-1)}} \otimes T_V \otimes \mathrm{id}_{V^{\otimes (n-i-1)}}, \qquad 1 \leq i \leq n-1.
\end{equation*}

\end{defn}

Note that $V^{\otimes 0} = I$, and that the groups $B_0 = B_1$ are trivial.

For nonnegative integers $m,n$, and $g \in B_m, \ h \in B_n$, we then get a commutative diagram
\begin{equation} \label{NaturoidCommDiag}
\begin{tikzcd}
V^{\otimes m} \otimes V^{\otimes n} \arrow[rr,"g \otimes h"] \arrow [d] & & V^{\otimes m} \otimes V^{\otimes n} \arrow[d] \\
V^{\otimes (m+n)} \arrow[rr,"gh"] & &  V^{\otimes (m+n)}
\end{tikzcd}
\end{equation}
in $\mathcal C$ where the vertical arrows are (the isomorphisms) coming from the symmetric monoidal structure, and $gh$ stands for juxtaposition of braids.

\begin{exmp} \label{SignInflation}

For $1 \leq i \leq n-1$ the element $\sigma_i \in B_n$ acts on the one-dimensional vector space $\kappa_\zeta^{\otimes n}$ via multiplication by $\zeta$.
In case $\zeta = -1 \neq 1$ this is the inflation of the sign representation from $S_n$ to $B_n$. 
    
\end{exmp}

A braided object $V$ in $\mathcal C$ is permutational in the sense of \cref{PermBraidObj} if and only if the action of $B_n$ on $V^{\otimes n}$ factors via the natural group homomorphism $B_n \to S_n$.

For an action of a group $G$ on a vector space $W$ over a field $\kappa$ we denote by
\[
H_0(G,W) = W_G = W/\operatorname{Span}_\kappa \ \{gw - w : g \in G, \ w \in W\}
\]
the coinvariants - the largest quotient of $W$ (in the category of vector spaces over $\kappa$) on which $G$ acts trivially.

\begin{defn} \label{DefAlgCoinvs}

Let $\kappa$ be a field.
To a braided vector space $V$ over $\kappa$ we functorially associate its graded $\kappa$-algebra of coinvariants
\[
C(V) = \bigoplus_{n=0}^\infty H_0(B_n, V^{\otimes n})
\]
with multiplication induced by the natural isomorphisms $V^{\otimes m} \otimes V^{\otimes n} \to V^{\otimes (m+n)}$, and well-defined by the commutativity of the diagram in \cref{NaturoidCommDiag}.

\end{defn}

This graded $\kappa$-algebra is the quotient of the (possibly noncommutative)  tensor (associative unital) graded $\kappa$-algebra 
$
\bigoplus_{n=0}^\infty V^{\otimes n}
$
modulo the (two-sided) ideal
\[
\langle gv - v : g \in B_n, \ v \in V^{\otimes n}, \ n \geq 0\rangle = \langle T(v_1 \otimes v_2) - v_1 \otimes v_2 : v_1, v_2 \in V \rangle.
\]

The projection from $C(V)$ to its degree $0$ part allows us to think of $C(V)$ as an augmented $\kappa$-algebra, so that the kernel $I$ of this projection (generated as a non-unital $\kappa$-algebra by $V = V^{\otimes 1}$) is the augmentation ideal of $C(V)$.

    

The configuration space of $n$ unordered distinct points on a genus $g$ surface with one boundary component and $f$ punctures is a $K(\pi,1)$ whose fundamental group is called the $n$-strand braid group on a surface of type $(g,f)$; we denote this group $B^n_{g,f}$.  
We think of the $B^n_{g,f}$ as a family of groups with $g,f$ fixed and $n$ varying.
We have $B_n = B^n_{0,0}$. 
By \cite{el}*{Notation 3.1.1} we have an injective group homomorphism
\begin{equation} \label{DirectProductSurfacBraid}
B_i \times B^{n-i}_{g,f} \to B^n_{g,f}
\end{equation}
for every $0\leq i \leq n$.
In particular, we view $B_n$ as a subgroup of $B^n_{g,f}$, corresponding to the $n$ points on the surface moving within a small disc.

\begin{defn}[\cite{el}*{Definition 3.1.6}]

Let $g,f$ be nonnegative integers.
Let $\kappa$ be a field, let $W$ be a vector space over $\kappa$, and let $V$ be a braided vector space over $\kappa$.
We say that $(V,W)$ is a coefficient system for the family of groups $B^n_{g,f}$ if for every $n \geq 0$ we have a $\kappa$-linear action of $B^n_{g,f}$ on $W \otimes V^{\otimes n}$ such that the diagram 
\[
\begin{tikzcd}
	{(B_i \times B^{n-i}_{g,f}) \times (V^{\otimes i} \otimes W \otimes V^{\otimes (n-i)})} && {V^{\otimes i} \otimes W \otimes V^{\otimes (n-i)}} \\
	\\
	{B^n_{g,f} \times (W \otimes V^{\otimes n})} && {W \otimes V^{\otimes n}}
	\arrow[from=1-1, to=1-3]
	\arrow[from=1-1, to=3-1]
	\arrow[from=1-3, to=3-3]
	\arrow[from=3-1, to=3-3]
\end{tikzcd}
\]
with horizontal arrows coming from the actions of the braid groups, and the vertical arrows coming from the monoidal structure and \cref{DirectProductSurfacBraid}, commutes.

A morphism from a coefficient system $(W_1,V_1)$ to a coefficient system $(W_2, V_2)$ is a pair of $\kappa$-linear maps $W_1 \to W_2$, $V_1 \to V_2$ such that for every $n \geq 0$ the induced $\kappa$-linear map $W_1 \otimes V_1^{\otimes n} \to W_2 \otimes V_2^{\otimes n}$ is a homomorphism of representations of $B^n_{g,f}$.
    
\end{defn}



Coefficient systems form a category, and $(V,W) \mapsto V$ gives rise to a functor from this category to the category of braided vector spaces.




In some situations, it will be useful to consider braided vector spaces $V$ which themselves carry a nonnegative grading that is compatible with the braiding; these are braided objects in the monoidal category of graded vector spaces.  When $V$ is graded, $C(V)$ is defined in the same way, but the grading (by nonnegative integers) is now given by
$$
C(V)_m = \bigoplus_{n = 0} ^\infty H_0(B_n, (V^{\tensor n})_m).
$$
When $V = V_1$ is homogeneous of degree $1$, this is simply $C(V)_m = H_0(B_m, V^{\tensor m})$.  When $(V,W)$ is a coefficient system and $V$ is a graded vector space, we place a grading on $W \tensor V^{\tensor n}$ by concentrating $W$ in degree $0$.







If $M$ is a graded vector space, we denote by $\deg M$ the largest $d$ such that $M_d$ is nonzero.  If there is no such $d$, we write $\deg M = \infty$.

\begin{thm} \label{th:0controlledvanishing}

Let $g,f$ be nonnegative integers, and let $(V,W)$ be a coefficient system for $B_{g,f}^n$ over a field $\kappa$.  
Suppose that $d = \deg C(V)$ is finite.
Then for nonnegative integers $m,n,p$ we have
$$
H_p(B_{g,f}^n, W \tensor (V^{\tensor n})_m) = 0
$$
whenever $p < (m-d)/(d+2\deg V)$.
In particular, if $V = V_{\leq 1}$, then
$$
H_p(B_{g,f}^n, W \tensor V^{\tensor n}) = 0
$$
whenever $p < (n-d)/(d+2)$.

\end{thm}

\begin{proof}
    The basic argument is derived from the spectral sequence used in \cite{evw} to compute the cohomology of Hurwitz spaces, which was generalized to base curves of arbitrary genus in \cite{el}.
    We begin by recalling some notation from \cites{el, evw}, slightly modified in order to account for the more general grading we allow on $V$. We denote by $M_p$ the graded vector space whose $m$th graded piece is
    $$
    \bigoplus_{n=0}^{\infty} H_p(B_{g,f}^n, W \tensor (V^{\tensor n})_m).
    $$
    Then $M_p$ is a graded $C(V)$-module.  (And when $(g,f) = (0,0)$, $M_0$ is actually equal to $C(V)$.)

    For any $C(V)$-module $M$, we define as in \cite{el}*{Definition 3.2.1}, \cite{evw}*{\S 4.1} a complex of vector spaces $\KK(M)$ whose term in position $q$ is given by $\KK(M)_q = M \tensor V^{\tensor q}$.  (Here, the notation differs slightly from that of \cite{el} and \cite{evw}, where we wrote $M[q] \tensor V^{\tensor q}$ with $M[q]$ denoting a shift by $q$ and $V$ placed in degree $0$.  This is of course the same thing as $M \tensor V^{\tensor q}$ where $V$ is placed in degree $1$.  But in the general case where we want to allow $V$ to have support in multiple gradings, this notation is more convenient.) 
    

    By \cite{el}*{Proposition 3.2.4}, there is a spectral sequence whose $E^1$ page is given by
 $$
 E^1_{pq} = \KK(M_p)_q.
 $$
In \cite{el}, the braided vector space $V$ is concentrated in grade $1$, so let us explain the difference, which is purely notational.  The $n$th graded piece of the spectral sequence in \cite{el} arises from the action of $B_{g,f}^n$ on a simplicial set $\A(g,f,n)$ which arises as a combinatorialization of the arc complex.  Given any representation $F$ of $B_{g,f}^n$ over $\kappa$, that action affords a spectral sequence whose $E^1_{pq}$ term is
\begin{equation*} \label{eq:elss}
H_p(B_{g,f}^n, k[\A(g,f,n)]_q \tensor F).
\end{equation*}  

In \cite{el}, we apply this when $(V,W)$ is a coefficient system and $F = W \tensor (V^{\tensor n}).$  In the present setting, $W \tensor (V^{\tensor n})$ breaks up as a sum of the graded pieces $W \tensor (V^{\tensor n})_m$, and each graded piece yields its own spectral sequence.  The direct sum of these spectral sequences over all nonnegative integers $n,m$ thus yields a bigraded spectral sequence $E$. 

As per the last paragraph of the proof of \cite{el}*{Proposition 3.2.4}, we know that $(E^{\infty}_{pq})_{n,m} = 0$ whenever $n > p+q+1.$  In this argument, we want to focus on the grading indexed by $m$ and ignore the one indexed by $n$.  So we observe that $W \tensor (V^{\tensor n})_m = 0$ once $m > n \deg V$, whence the $(n,m)$ bigraded piece of $E$ is zero once $m > n \deg V$.  We conclude that $(E^{\infty}_{pq})_{n,m} = 0$ whenever $m > (p+q+1)\deg V$; in other words, $\deg E^{\infty}_{pq} \leq (p+q+1)\deg V$.

 We now have the tools in place for an induction on $p$.  The statement to be proved can also be written as
 $$
 \deg M_p \leq (d+2\deg V)p + d.
 $$
 We first consider the base case $p=0$.  The coinvariants $(M_0)_n = H_0(B_{g,f}^n, W \tensor V^{\tensor n})$ lie in the image of $C(V)_n W = C(V)_n (M_0)_0$, so they vanish once $n > \deg C(V)$; in other words, $\deg M_0 \leq d$ as claimed.

 Now let $p > 0$ and suppose that $\deg M_{p'} \leq (d+2 \deg V)p' + d$ for all $p' < p.$

 The convergence of the spectral sequence to $0$ in high degrees implies that
 \beq
 \deg E^2_{p0} \leq \max((p+1)\deg V, \max_{q > 0} \deg E^1_{p-q,q+1})
 \eeq
 since, for each $q>0$, we know that the successive quotient $E^{q+1}_{p0}/E^{q+2}_{p0}$ is surjected upon by the image of a differential coming from $E^{q+1}_{p-q,q+1}$, which is a subquotient of $E^1_{p-q,q+1}$.  On the other hand,
 $$
 \deg E^1_{p-q,q+1} = \deg \KK(M_{p-q})_{q+1}
 = \deg M_{p-q} \tensor V^{\tensor q+1} \leq (d+2\deg V)(p-q) + d + (q + 1) \deg V
 $$
 where the last step follows from the induction hypothesis applied to $M_{p-q}$.  But
 $$(d+2 \deg V)(p-q)+d+(q+1)\deg V = (d+2 \deg V)p+(1-q)(d+\deg V)
 $$
 which is at most $(d+2 \deg V)p$ when $q > 0$.  For $p>0$ we also have $(p+1)\deg V \leq 2p\deg V \leq (d+2 \deg V)p$.  So
 $$
 \deg E^2_{p0} \leq \max((p+1)\deg V,\max_{q > 0} \deg E^1_{p-q,q+1}) \leq (d+2 \deg V)p.
 $$

 Now $E^2_{p0}$ is the cokernel of the differential from $E^1_{p1}$ to $E^1_{p0}$, or in concrete terms, the cokernel of the map
 $$
 V \tensor M_p \ra M_p
 $$
which we may also write as $M_p / I M_p$, where $I$ is the augmentation ideal of $C(V)$.  In other words, $M_p$ is generated as a $C(V)$-module in degree at most $(d+2\deg V)p$, and since $C(V)$ itself has degree $d$, we conclude that $\deg M_p \leq (d+2\deg V)p + d$, as claimed.  This completes the proof. 
\end{proof}

\section{Spaces, Sheaves, Trace Functions}
\label{s:sheaves}
\subsection{Configuration Spaces}

We denote by $\Conf^n = \Conf^n \mathbb A^1$ the configuration space of $n$ distinct unordered points on the affine line over $\Z$.
This is the affine open subscheme of the moduli space of monic degree $n$ polynomials where the discriminant is invertible.
This discriminant then gives rise to a morphism of schemes 
\begin{equation} \label{DiscriminantMorphismConSpace}
\delta \colon \Conf^n \to \Gm.
\end{equation}
The analytification of the base change of $\delta$ to $\C$ induces, on the level of fundamental groups, the surjective group homomorphism $B_n \to \Z$ that maps the standard generators $\sigma_i \in B_n$ to $1 \in \Z$ for every $1 \leq i \leq n-1$.


Let $n_1, \dots, n_k$ be nonnegative integers summing up to $n$.
Consider the open subscheme of $\Conf^{n_1} \times \cdots \times \Conf^{n_k}$ given by 
\[
\Conf^{n_1, \dots, n_k} = \left\{(f_1, \dots, f_k) \in \prod_{i=1}^k \Conf^{n_i} : \operatorname{Res}(f_i,f_j) \textup{ is invertible for } 1 \leq i < j \leq k  \right\}.
\]
Its complex points can be viewed as configurations in $\C$ of $n$ distinct points having $n_i$ points of color $i$, with two configurations considered identical if their sets of points of color $i$ coincide for $1 \leq i \leq k$.
We have a finite \'etale map 
\begin{equation} \label{ProductFiniteEtaleMapConf}
\tau \colon \Conf^{n_1, \dots, n_k} \to \Conf^n, \qquad  \tau(f_1, \dots, f_k) = f_1 \cdots f_k.
\end{equation}
The analytification of the base change of this map to $\C$ induces an inclusion of 
\[
B_{n_1, \dots, n_k} = \pi_1(\Conf^{n_1, \dots, n_k}(\C))
\]
into $B_n$ with image having index $\binom{n}{n_1, \dots, n_k}$.
The (analytification of the base change to $\C$ of the) open immersion \begin{equation} \label{ConfOpenImmersoionProduct}
c \colon \Conf^{n_1, \dots, n_k} \to \Conf^{n_1} \times \cdots \times \Conf^{n_k}    
\end{equation}
induces a surjective group homomorphism $B_{n_1, \dots, n_k} \to B_{n_1} \times \cdots \times B_{n_k}$ for which the natural inclusion $B_{n_1} \times \cdots \times  B_{n_k} \to B_{n_1, \dots, n_k}$, given by juxtaposition of braids, is a section.
We have
\[
\Conf^{n_1, \dots, n_k}(\F_q) = \left\{(f_1, \dots, f_k) \in \prod_{i=1}^k \Conf^{n_i}(\F_q) : \gcd(f_i,f_j)=1 \textup{ for } 1 \leq i < j \leq k  \right\}.
\]

In case $k=n$ and $n_1 = \dots = n_k = 1$ we write $\operatorname{PConf}^n$ for $\Conf^{n_1, \dots, n_k}$, and call it ordered configuration space.
The map $\tau \colon \operatorname{PConf}^n \to \operatorname{Conf}^n$ reconstructing a squarefree polynomial from an ordering of its roots is a Galois $S_n$-cover.
On the level of complex points this map forgets the ordering of $n$ distinct points in $\C$, and $PB_n = \pi_1(\operatorname{PConf}^n(\C))$ is the kernel of the group homomorphism $B_n \to S_n$.
One can thus also describe $B_{n_1, \dots, n_k}$ as the inverse image of $S_{n_1} \times \dots \times S_{n_k}$ under the group homomorphism $B_n \to S_n$.


For $f \in \Conf^n(\F_q)$ we denote by $\sigma_f$ the (conjugacy class of the) permutation in $S_n$ by which $\operatorname{Frob}_q$ acts on the $n$ roots of $f$ in $\overline{\F_q}$.  
The continuous group homomorphism $\pi_1^{\textup{\'et}}(\Conf^n) \to S_n$ corresponding to the cover $\operatorname{PConf}^n \to \operatorname{Conf}^n$ maps $\operatorname{Frob}_f$ to $\sigma_f$. 
The cycle structure of $\sigma_f$ is the factorization type of $f$.

Let $\ell$ be an auxiliary prime number not dividing $q$, and fix an isomorphism of fields $\overline {\Q_\ell} \cong \C$.
In the sequel we may tacitly use this isomorphism with no further mention.

Let $\rho$ be a finite-dimensional representation of $S_n$ over $\overline{\Q_\ell}$. We view $\rho$ as a constructible lisse \'etale sheaf of vector spaces over $\overline{\Q_\ell}$ on $\Conf^n$.
This sheaf is punctually pure of weight zero, as can be seen from its trace function
\begin{equation*} 
\operatorname{tr}(\operatorname{Frob}_f, \rho_{\bar f}) = \operatorname{tr}(\rho(\sigma_f))   
\end{equation*}
where $\bar f$ is a geometric point of $\Conf^n$ over $f$.

\begin{exmp} \label{TrivialRepSnTraceFunction}

Denote by $\operatorname{triv}_n$ the trivial one-dimensional representation of $S_n$. The trace function of $\operatorname{triv}_n$ is the constant function $1$ on $\Conf^n(\F_q)$.
    
\end{exmp}

\begin{exmp} \label{MobiusFunctionRepThry}

The trace function of $\operatorname{sign}_n$ is $(-1)^n$ times the M\"obius function by \cite{Sawin}*{Lemma 3.5}.
The M\"obius function is multiplicative - if  for $f \in \Conf^n(\F_q)$ we have $f = gh$ with $g \in \Conf^k(\F_q)$ and $h \in \Conf^{n-k}(\F_q)$ then $\mu(f) = \mu(g)\mu(h)$.

\end{exmp}

\begin{exmp} \label{DivisorFunctionRepThry}

Induction of a representation from a finite index subgroup of a fundamental group corresponds to pushforward of a locally constant sheaf by a finite \'etale map.
By the function-sheaf dictionary, the corresponding operation on functions is summation over the fiber, so
\[
\operatorname{tr}(\operatorname{Frob}_f, \operatorname{Ind}^{S_n}_{S_{n_1} \times \dots \times S_{n_k}} \overline{\Q_\ell}) = \# \tau^{-1}(f)
\]
where $\tau$ is the multiplication map from \cref{ProductFiniteEtaleMapConf}.
We conclude that 
\[
\operatorname{tr}\left(\operatorname{Frob}_f, \bigoplus_{\substack{n_1, \dots, n_k \geq 0 \\ n_1 + \dots+n_k = n}} \operatorname{Ind}^{S_n}_{S_{n_1} \times \dots \times S_{n_k}} \overline{\Q_\ell} \right) = d_k(f).
\]

\end{exmp}

\begin{exmp} \label{VonMangoldtUsingWedges}

By \cite{Sawin}*{Lemma 3.6} we have
\begin{equation*}
\mathbf 1_{\textup{irr}}(f) =  \frac{1}{n} \sum_{k=0}^{n-1} (-1)^k \cdot \operatorname{tr}(\operatorname{Frob}_f, \wedge^k \operatorname{std}_n)
\end{equation*}
where $\operatorname{std}_n$ is the standard $(n-1)$-dimensional representation of $S_n$.
    
\end{exmp}

For $0 \leq k \leq n$ we have the resultant morphism
\begin{equation} \label{ResultantMorphismConfSpace}
\Res \colon \Conf^{k,n-k} \to \Gm.
\end{equation}
The group $B_{k,n-k}$ is generated by $\{\sigma_1, \dots, \sigma_{k-1}, \sigma_k^2, \sigma_{k+1}, \dots, \sigma_{n-1}\}$ subject to the relations in \cite{Manf}*{Proposition 1, Theorem 3}, and (the analytification of the base change to $\C$ of) $\Res$ induces a group homomorphism $B_{k,n-k} \to \Z$ that maps $\sigma_k^2$ to $1$ and the other generators to $0$.

At times, we will denote the map $\Conf^{k,n-k} \to \Conf^n$ from \cref{ProductFiniteEtaleMapConf} by $\tau^{(k)}$, and the map $\Conf^{k,n-k} \to \Conf^k \times \Conf^{n-k}$ from \cref{ConfOpenImmersoionProduct} by $c_k$.


\subsection{Hurwitz Spaces}


There exists a scheme $\mathsf{Hur}_{G,R}^{n_1, \dots, n_k}$ over $\mathbb Z_{(p)}[\zeta_{q-1}]$ equipped with a finite \'etale morphism  
\begin{equation} \label{HurwitzFiniteCoverConfCol}
\pi \colon \mathsf{Hur}_{G,R}^{n_1, \dots, n_k} \to \textup{Conf}^{n_1, \dots, n_k}
\end{equation}
for which we have the identification
\begin{equation} \label{HurwitzPointsGCovers}
\mathsf{Hur}_{G,R}^{n_1, \dots, n_k}(\F_q) = \mathcal E_q^R(G;n_1, \dots, n_k)
\end{equation}
such that $L \in \mathcal E_q^R(G;n_1, \dots, n_k)$ is mapped under $\pi$ to $(f_1, \dots, f_k)$ where $f_j$ is the product of the monic irreducible polynomials ramified in $L$ with the corresponding inertia generator mapped to $R_j$ for every $1 \leq j \leq k$.
We thus have
\begin{equation} \label{MapPiComputesProdRamPrimes}
(\tau \circ \pi)(L)= f_L
\end{equation}
where the notation is that of \cref{ProductFiniteEtaleMapConf}.

The sheaf $\tau_* \pi_* \overline{\mathbb Q_\ell}$ is lisse, punctually pure of weight $0$, and
\begin{equation} \label{PushforwardFromHurwitzSpaceTrace}
\operatorname{tr}(\operatorname{Frob}_{\bar g}, \tau_* \pi_* \overline{\mathbb Q_\ell}) = \#\{L \in \mathcal E_q^R(G;n_1, \dots, n_k) : f_L = g\}, \qquad g \in \Conf^n(\F_q).
\end{equation}
The analytification of the base change of $\tau_* \pi_* \overline{\mathbb Q_\ell}$ to $\C$ corresponds to a direct summand of the representation $\overline{\Q_\ell}R^{\otimes n}$ of $B_n$ arising from \cref{RacksToBraidedVS}, \cref{BraidedToNaturoid}, and \cref{MultiPartitionExmpAction}, spanned by those $n$-tuples that generate $G$, multiply up to $1$ in $G$, and have $n_i$ entries from $R_i$ for every $1 \leq i \leq k$.
The standard generator $\sigma_i \in B_n$ with $1 \leq i \leq n-1$ maps an $n$-tuple $(r_1, \dots, r_n)$ of elements from $R$ to the $n$-tuple
\[
(r_1, \dots, r_{i-1}, r_{i+1}, r_{i+1}^{-1}r_ir_{i+1}, r_{i+2}, \dots, r_n) \in R^n.
\]

\subsection{Character Sheaves} \label{CharSheavesSection}

Let $\chi \colon \F_q^\times \ra \C^\times$ be a nontrivial multiplicative character.
Denote the order of $\chi$ by $\mathfrak o$, a divisor of $q-1$, and note that $\F_q$ contains a primitive $\mathfrak o$th root of unity, so the morphism $\Gm \to \Gm$ of schemes over $\F_q$ raising to power $\mathfrak o$ is a Galois cover with Galois group $\{\alpha \in \Gm(\F_q) = \F_q^{\times} : \alpha^{\mathfrak o} = 1\}$ acting by translation.
We view this group as the quotient of $\F_q^{\times}$ obtained by mapping $x \in \F_q^\times$ to $x^{\frac{q-1}{\mathfrak o}}$, and note that $\chi$ factors via this quotient.
As a result, $\chi$ gives rise to a lisse \'etale $\overline{\mathbb Q_\ell}$-sheaf $\mathcal L_\chi$ of rank one on $\Gm$ over $\F_q$ whose trace function is given by
\begin{equation} \label{BasicTraceFunctionMultChar}
\operatorname{tr}(\operatorname{Frob}_{\bar \alpha}, \mathcal L_\chi) = \chi(\alpha), \quad \alpha \in \Gm(\F_q).
\end{equation}
The sheaf $\mathcal L_\chi$ is thus punctually pure of weight $0$.

Let $\zeta \in \C$ be a root of unity of order $\mathfrak o$, choose a ring homomorphism $\mathbb Z[\zeta] \to \F_q$, and note that the morphism $\Gm \to \Gm$ of schemes over $\mathbb Z[\zeta]$ raising to power $\mathfrak o$ is Galois, and base change by $\mathbb Z[\zeta] \to \F_q$ induces an isomorphism from its Galois group onto the aforementioned Galois group over $\F_q$.
Therefore, there exists a rank one lisse sheaf punctually pure of weight $0$ on $\Gm$ over $\mathbb Z[\zeta]$ whose pullback under $\operatorname{Spec} \F_q\to \operatorname{Spec} \Z[\zeta]$ is $\mathcal L_{\chi}$. Abusing notation, we denote this sheaf by $\mathcal L_\chi$ as well.
The analytification of the base change of $\mathcal L_\chi$ to $\C$ corresponds to the representation of $\pi_1(\Gm(\C)) = \pi_1(\C^\times) \cong \Z$ that maps $1 \in \Z$ to a primitive $\mathfrak o$th root of unity.

Recalling \cref{DiscriminantMorphismConSpace}, we see that the sheaf $\delta^{-1} \mathcal L_\chi$ is lisse of rank $1$, punctually pure of weight $0$, and its trace function is given by
\begin{equation} \label{TraceFunctionCharacterSheafDiscriminant}
\operatorname{tr}(\operatorname{Frob}_{\bar f}, \delta^{-1}\mathcal L_\chi) = \chi(\disc(f)), \qquad f \in \Conf^n(\F_q).
\end{equation}
The analytification of the base change of $\delta^{-1} \mathcal L_\chi$ to $\C$ corresponds to the representation of $B_n$ that maps $\sigma_i \in B_n$, for every $1 \leq i \leq n$, to a given primitive $\mathfrak o$th root of unity.

Similarly, recalling \cref{ResultantMorphismConfSpace}, the sheaf $\operatorname{Res}^{-1} \mathcal L_\chi$ is lisse of rank $1$ on $\Conf^{k,n-k}$, punctually pure of weight $0$, and its trace function is given by
\begin{equation} \label{TraceFunctionCharacterSheafResultant}
\operatorname{tr}(\operatorname{Frob}_{\overline{(f,g)}}, \operatorname{Res}^{-1}\mathcal L_\chi) = \chi(\Res(f,g)), \qquad (f,g) \in \Conf^{k,n-k}(\F_q).
\end{equation}

Pushing forward lisse \'etale sheaves by a finite \'etale map corresponds to induction of representations from the open subgroup of the \'etale fundamental group corresponding to the finite \'etale map.
By the function-sheaf dictionary, pushforward by a proper map corresponds to summation of the trace function over the fibers of the map.
The sheaf $\tau^{(k)}_* \Res^{-1}\mathcal L_\chi$ is thus lisse on $\Conf^n$ punctually pure of weight $0$ and its trace function is given by
\begin{equation} \label{TraceFunctionCharacterSheafResultantSumOne}
\operatorname{tr}(\operatorname{Frob}_{\bar f}, \tau_*^{(k)} \operatorname{Res}^{-1}\mathcal L_\chi) = \sum_{\substack{(g,h) \in \Conf^{k,n-k}(\F_q) \\ gh = f}} \chi(\Res(g,h)), \qquad f \in \Conf^{n}(\F_q).
\end{equation}
As a result, we have
\begin{equation} \label{TraceFunctionCharacterSheafResultantSumTwo}
\operatorname{tr} \left(\operatorname{Frob}_{\bar f}, \bigoplus_{k=0}^n \tau_*^{(k)} \operatorname{Res}^{-1}\mathcal L_\chi \right) = \sum_{gh = f} \chi(\Res(g,h)), \qquad f \in \Conf^{n}(\F_q).
\end{equation}
where $g,h$ are monic (squarefree, and coprime).

\begin{lem} \label{ChiResultantLegendreSymbol}

In case $\mathfrak 0 = 2$ we have
$
\chi(\Res(f,g)) = \left( \frac{g}{f} \right).
$
    
\end{lem}

\begin{proof}

Both sides are multiplicative in $f$ so we can assume $f$ is irreducible.
The irreducibility of $f$ means that if $\xi \in \overline{\F_q}$ is a root of $f$, then the $k$ (distinct) roots of $f$ in $\overline{\F_q}$ are $\xi^{q^i}$ for $0 \leq i \leq k-1$.
Using the fact that Frobenius is an automorphism, and the formula for the sum of a geometric progression, we get
\begin{equation*} 
\mathrm{Res}(f,g) = \prod_{i=0}^{k-1}g(\xi^{q^i}) = \prod_{i=0}^{k-1}g(\xi)^{q^{i}} = g(\xi)^{\frac{q^{k}-1}{q-1}}.
\end{equation*}
Since $\chi$ is quadratic, we arrive at the congruence 
\begin{equation*}
\chi (\mathrm{Res}(f,g)) \equiv \mathrm{Res}(f,g)^{\frac{q-1}{2}} \equiv g(\xi)^{\frac{q^{k}-1}{2}} \equiv \left(\frac{g}{f} \right)
\end{equation*}
modulo $p$, which suffices for our claim.    
\end{proof}

\section{Coarse Bounds on (Unstable) Cohomology}

\label{s:coarsebounds}
\subsection{Comparison of Cohomology}


Let $U$ be a smooth curve of genus $g$ over an open subscheme $S$ of the spectrum of the ring of integers of a number field $K$.
Suppose that the base change of $U$ to $\C$ has $f+1$ missing points for some nonnegative integer $f$. We then have
\[
\pi_1(\Conf^n U(\C)) = B^{n}_{g,f}.
\]

\begin{thm} \label{ComparisonCohomCor}

Let $\kappa$ be a finite field, let $n$ be a positive integer, and let $\mathcal F$ be a locally constant constructible sheaf of vector spaces over $\kappa$ on $\operatorname{Conf}^n U$.
Denote by $\mathcal F_{\C}$ the base change of $\mathcal F$ to $\C$, and by $\mathcal F_{\C}^{\textup{an}}$ its analytification, viewed as a representation of $B^n_{g,f}$.
Then for every prime number $p$ not invertible in $S$ and different from the characteristic of $\kappa$ we have an isomorphism 
\[
H^i_c(\operatorname{Conf}^n U_{\overline{\F_p}}, \mathcal F) \cong 
H_{2n-i}(B^n_{g,f}, \mathcal (\mathcal F_\C^{\textup{an}})^\vee)^\vee
\]
of vector spaces over $\kappa$.


\end{thm}

\begin{proof}

As in the proof of \cite{el}*{Theorem 9.2.4}, one can use \cite{el}*{Corollary B.1.4} to get a normal crossings compactification of $\Conf^n U$, and then conclude using \cite{evw}*{Proposition 7.7}, comparison of \'etale and singular cohomology over $\C$, and Poincar\'e duality.

    
\end{proof}

\subsection{General Betti Bounds}

The following is an immediate consequence of \cite{Callegaro}*{Theorem 2.10, Theorem 2.11, Remark 2.16}.

\begin{lem} \label{Salvetti}

Let $n$ be a positive integer, let $\kappa$ be a field, and let $M$ be a representation of $B_n$ over $\kappa$. Then for every integer $j$ we have
\[
\dim_\kappa H_j(B_n,M) \leq \binom{n-1 }{j} \cdot \dim_\kappa M.
\]

\end{lem} 

More generally, the following is a consequence of the proof of \cite{el}*{Lemmas 4.3.1 and 4.3.2}.

\begin{lem} \label{AaronJordanExplicated}

Let $n$ be a positive integer, let $f,g$ be nonnegative integers, let $\kappa$ be a field, and let $M$ be a representation of $B^n_{g,f}$ over $\kappa$. Then for every integer $j$ we have
\[
\dim_\kappa H_j(B^n_{g,f},M) \leq \binom{2g+f+n}{2g+f+j} \cdot \dim_\kappa M.
\]

\end{lem}

\section{Braided Objects}

\label{s:racks}

In this section, we set up some key facts about braided objects, and about ways of generating new braided objects from old ones, which will be critical for the examples we want to treat in this paper.

Recall that a monoidal category $\mathcal C$ is symmetric if it has a (functorial) commutativity constraint, namely an isomorphism
\begin{equation} \label{SwapSymmetricMonoidal}
s_{V,W} \colon V \otimes W \to W \otimes V
\end{equation}
for objects $V,W$ of $\mathcal C$ satisfying $s_{W,V} \circ s_{V,W} = \mathrm{id}_{V \otimes W}$ and some other natural properties. 

\begin{exmp}
    
The category of (graded) vector spaces over a field $\kappa$ is symmetric monoidal, and so is the category of sets.

\end{exmp}

\subsection{Racks}

\begin{defn} \label{RackDef}
A rack is a set $R$ with a binary operation $x^y$ for $x,y \in R$ such that for every $y \in R$ the function $x \mapsto x^y$ is a bijection on $R$, and for all $x,y,z \in R$ we have $(x^y)^z = (x^z)^{y^z}$.
A function $f \colon R \to S$ between racks is a morphism if for every $x,y \in R$ we have $f(x^y) = f(x)^{f(y)}$.
\end{defn}

Racks form a category that admits direct products.
Indeed for racks $R,S$ we endow the set $R \times S$ with a rack structure by setting 
$
(x,y)^{(z,w)} = (x^z,y^w)
$
for $x,z \in R$ and $y,w \in S$. 
Direct products make racks a symmetric monoidal category with unit given by the one-element rack $\mathcal T_1$.
The empty rack is an initial object in the category of racks.

If a subset $R'$ of a rack $R$ is a rack with respect to the binary operation on $R$ we say that $R'$ is a subrack of $R$. In this case the inclusion of $R'$ into $R$ is a homomorphism of racks. 
A finite subset $R'$ of a rack $R$ is a subrack if and only if for every $r,s \in R'$ we have $r^s \in R'$.

If $R$ is a rack and $R_i$ for $i \in I$ are subracks of $R$, then $\cap_{i \in I} R_i$ is also a subrack of $R$.

\begin{defn} \label{GeneratingRack}

Let $R$ be a rack.
Given an indexing set $I$, and elements $r_i \in R$ for every $i \in I$, we denote by
\[
\langle r_i \rangle_{i \in I} = \bigcap_{\substack{R' \text{ a subrack of } R \\ r_i \in R' \text{ for every } i \in I }} R'
\]
the subrack of $R$ generated by all the $r_i$ for $i \in I$.

\end{defn}

We say that the elements $r_i \in R$ generate $R$ if $\langle r_i \rangle_{i \in I} = R$ or equivalently if there is no proper subrack of $R$ containing $r_i$ for every $i \in I$.
For a nonnegative integer $n$ we introduce the notation
\begin{equation} \label{GeneratingNtuplesRack}
R^n_{\times} = \{(r_1, \dots, r_n) \in R^n : \langle r_1, \dots, r_n \rangle = R\}
\end{equation}
for those $n$-tuples in $R$ that generate it.

\begin{defn} \label{DefQuandle}
    
A rack $R$ is said to be a quandle if for every $r \in R$ we have $r^r = r$. 

\end{defn}

We get a full monoidal subcategory of quandles in the category of racks.

\begin{exmp} \label{ForgetFulGroupsRacks}

Let $R$ be a conjugacy-closed subset of a group $G$, and let $x^y = y^{-1} x y$ be the conjugation of $x$ by $y$ in $G$.
Then $R$ is a rack, and even a quandle.
In particular, taking $R = G$, we get a forgetful functor from the category of groups to the category of racks.
    
\end{exmp}

\begin{exmp} \label{TrivialRack}

Endowing a set $X$ with the trivial rack structure given by $x^y = x$ for all $x,y \in X$ is a fully faithful symmetric monoidal functor from the catgeory of sets to the category of quandles. We denote the resulting trivial rack on a set of cardinality $\nu$ by $\mathcal T_\nu$.
Note that $\mathcal T_1$ is a final object in the category of racks.
We will use the notation $\mathcal T_2 = \{\varphi, \psi\}$.

\end{exmp}

\begin{defn} \label{DisjoinUnionRacks}

Let $R$ and $S$ be racks. 
We endow the disjoint union $R \coprod S$ with a rack structure by 
\[
x^y = 
\begin{cases}
x^y &x,y \in R \\
x^y &x,y \in S \\
x &x \in R, \ y \in S \\
x &x \in S, \ y \in R.
\end{cases}
\]

\end{defn}

This defines a coproduct in the category of racks (and in the subcategory of quandles). Racks form a symmetric monoidal category with respect to the disjoint union with unit given by the empty rack.

The rack $\mathcal T_\nu$ from \cref{TrivialRack} is the disjoint union of $\nu$ copies of $\mathcal T_1$.

\begin{exmp} \label{BraidedSetRack}

Braided sets are the braided objects in the category of sets.
There is a fully faithful functor from the category of racks to the category of braided sets endowing (the underlying set of) a rack $R$ with the braiding $T \colon R \times R \to R \times R$, given by $T(x,y) = (y,x^y)$. 
As a result, \cref{BraidedToNaturoid} gives us an action of $B_n$ on $R^n$.
Note that this action preserves $R^n_\times$ as a set.

    
\end{exmp}

\begin{exmp} \label{PermutationalRack}

Let $R$ be a set, and let $\varphi \colon R \to R$ be a bijection.
Then $x^y = \varphi(x)$ endows $R$ with a rack structure.
This rack is a quandle if and only if $R$ is a singleton.
The construction extends to a functor from the category of sets equipped with a permutation to the category of racks.

In the sequel we will be interested in the case $R = \Z/m\Z$ for $m \in \Z$, and $\varphi(r) = r + 1$ for every $r \in R$.


    
\end{exmp}

\begin{exmp} \label{JoyceQuandle}

Let $\mathcal S_{\pm} = \{J_1, J_2,J_3\}$ be the Joyce quandle, defined by
\[
J_1^{J_1} = J_1^{J_2} = J_1, \quad J_2^{J_1} = J_2^{J_2} = J_2, \quad J_3^{J_1} = J_3^{J_2} = J_3^{J_3} = J_3, \quad J_1^{J_3} = J_2, \quad J_2^{J_3} = J_1. 
\]

\end{exmp}

\begin{defn} \label{PartitionRackDef}

A partition of a rack $R$ is a choice of two disjoint subsets $S,T \subseteq R$ whose union is $R$ such that $s^r \in S$ and $t^r \in T$ for every $s \in S$, $t \in T$, $r \in R$.
We say that the partition is nontrivial if $S$ and $T$ are nonempty.
    
\end{defn}

Note that the sets $S$ and $T$ are racks themselves.
Also, every two racks $R,S$ form a partition of $R \coprod S$.

If $S,T \subseteq R$ is a partition of a rack, and $Y$ is a rack, then $Y\times S, Y \times T \subseteq Y \times R$ is a partition as well.

\begin{exmp} \label{MultiPartitionExmpAction}

We can similarly speak of a partition of a rack $R$ into (disjoint) subsets $S_1, \dots, S_k$ such that for every $1 \leq i \leq k$, $s \in S_i$, and $r \in R$ we have $s^r \in S_i$.

For nonnegative integers $n_1, \dots, n_k$ we denote by $n$ their sum and put
\[
R(n_1, \dots, n_k) = \left\{(r_1, \dots, r_n) \in R^n : \#\{1 \leq j \leq n : r_j \in S_i\} = n_i, \text{ for every } 1 \leq i \leq k \right\}.
\]
This set is invariant (as a set) under the action of $B_n$ on $R^n$ from \cref{BraidedSetRack}, and we have a disjoint union
\begin{equation} \label{SetTheoreticTacitBinomial}
R^n = \bigcup_{n_1 + \dots + n_k = n} R(n_1, \dots, n_k).   
\end{equation}

\end{exmp}

\begin{defn}

We say that a subset $S$ of a rack $R$ is an ideal (of $R$) if for every $r \in R$ the function that takes $s \in S$ to $s^r$ is a permutation of $S$.  
    
\end{defn}

A finite subset $S$ of a rack $R$ is an ideal (of $R$) if and only if for every $s \in S$ and $r \in R$ we have $s^r \in S$.

In a partition $S,T \subseteq R$ the subsets $S$ and $T$ are ideals.
Conversely, for an ideal $S$ of a rack $R$ the subset $T = R \setminus S$ is an ideal of $R$ as well, and $S,T \subseteq R$ is a partition of $R$.

For a rack $R$, an ideal $I$ of $R$, and a subrack $S$ of $R$, the intersection of $S$ and $I$ is an ideal of $S$.

For ideals $I,J$ of a rack $R$, the intersection $I \cap J$ is also an ideal of $R$.

\begin{defn} \label{ConjugacyClosedSubsetGroup}

We say that a subset $R$ of a group $G$ is conjugacy-closed if it is a (disjoint) union of conjugacy classes of $G$, or equivalently, if for every $g \in G$ and $r \in R$ we have $r^g \in R$. 
    
\end{defn}

\begin{exmp} \label{IdealsInGroups}

Let $G$ be a group, and let $R \subseteq G$.
When $G$ is viewed as a rack, $R$ is an ideal of $G$ if and only if $R$ is a conjugacy-closed subset of $G$.
    
\end{exmp}


    




\subsubsection{Cocycles}

\begin{defn} \label{RacksWithCocyclesDef}

Let $\kappa$ be a field, and $R$ be a rack.
A function $c \colon R \times R \to \kappa^\times$ is said to be a $2$-cocycle (of $R$ over $\kappa$) if $c(r,s) \cdot c(r^s,t) = c(r,t) \cdot c(r^t,s^t)$ for every $r,s,t \in R$. We say that $c$ is cyclotomic if $c(x,y)$ is a root of unity for every $x,y \in R$.

Let $(R,c)$, $(S,d)$ be racks equipped with $2$-cocycles over $\kappa$, and let $f \colon R \to S$ be a morphism of racks. We say that $f$ is cocyclic if $c(x,y) = d(f(x),f(y))$ for every $x,y \in R$.

\end{defn}

Racks equipped with $2$-cocycles over $\kappa$ and cocyclic morphisms form a symmetric monoidal category.
Indeed if $c$ (respectively, $d$) is a $2$-cocycle of $R$ (respectively, $S$) over $\kappa$ then 
\[
e \colon (R \times S) \times (R \times S) \to k^\times, \qquad 
e((x,y),(z,w)) = c(x,z) \cdot d(y,w), \qquad x,z \in R, \ y,w \in S
\]
is a $2$-cocycle of $R \times S$ over $\kappa$.





\begin{exmp} \label{CocycleWedgePMexmp}

Let $R$ be a rack, let $\kappa$ be a field, and let $c \colon R \times R \to \kappa^\times$ be a function satisfying 
\[
c(x^z,y) = c(x,y), \qquad c(x,y^z) = (x,y), \qquad x,y,z \in R.
\]
Then $c$ is a $2$-cocycle.
In particular for $\lambda \in k^\times$, the constant function $c(r,s) = \lambda$ for $r,s \in R$ is a $2$-cocycle.
As another such example we have the cyclotomic $2$-cocycles $c_\wedge$ and $c_{\pm}$ over a field $\kappa$ of characteristic different from $2$, on the rack $R \times \mathcal T_2$ given, in the notation of \cref{TrivialRack}, by 
\[
c_\wedge\left((x,y),(z,w)\right) = 
\begin{cases}
-1 & y = w = \psi \\ 
1 &\text{else}
\end{cases},
\qquad
c_{\pm} \left((x,y),(z,w)\right) = 
\begin{cases}
-1 & y = \psi, \ w = \varphi \\ 
1 &\text{else}
\end{cases}
\]
for $x,z \in R$ and $y,w \in \mathcal T_2$.
In case $R = \mathcal T_1$ we view $c_\wedge$ and $c_{\pm}$ as $2$-cocycles on $\mathcal T_1 \times \mathcal T_2 = \mathcal T_2$.
    
\end{exmp}

\begin{exmp} \label{RacksToBraidedVS}

Let $\kappa$ be a field.
The (symmetric monoidal) association to a set $S$ of the vector space $\kappa S$ of its (formal) $\kappa$-linear combinations extends to a faithful functor from braided sets to braided vector spaces over $\kappa$.
Composing this with the functor from \cref{BraidedSetRack} we get a faithful functor from the category of racks to the category of braided vector spaces over $\kappa$, associating to a rack $R$ the vector space $\kappa R$ (sometimes denoted also $\kappa[R]$) braided by
$
T(x \otimes y) = y \otimes x^y
$
for $x,y \in R$.

More generally, we have a faithful functor $(R,c) \mapsto \kappa R(c)$ from the category of racks with $2$-cocycles over $\kappa$ to the category of braided vector spaces over $\kappa$ where $\kappa R(c)$ is the vector space $\kappa R$ braided by 
\begin{equation} \label{CocycleToBVs}
T(x \otimes y) = c(x,y) \cdot y \otimes x^y, \qquad x,y \in R.
\end{equation}
If $c$ is the constant function $1$ we have $\kappa R(c) = \kappa R$.

For instance, in the notation of \cref{IntroducingKwedge} and \cref{LegendreBraidVS}, for a field $\kappa$ of characteristic different from $2$ we have 
\begin{equation} \label{WedgePMCocycles}
\kappa_\wedge \cong \kappa \mathcal T_2(c_{\wedge}), \qquad \kappa \mathcal T_2(c_{\pm}) \cong \kappa_{\pm}
\end{equation} 
as braided vector spaces over $\kappa$.
If $\zeta \in \kappa^\times$ and $c$ is the constant $2$-cocycle on $\mathcal T_1$ with value $\zeta$ then, in the notation of \cref{ScalingBraidedVS}, we have $\kappa\mathcal T_1(c) \cong \kappa_\zeta$ as braided vector spaces over $\kappa$.

\end{exmp}

\begin{rem} \label{LinearizationNaturoid}

Let $\kappa$ be a field and $R$ be a rack.
The representation $\kappa R^n$ of $B_n$ over $\kappa$ associated to the action of $B_n$ on $R^n$ from \cref{BraidedSetRack} is naturally isomorphic to the representation of $B_n$ on $\kappa R^{\otimes n}$ arising from \cref{BraidedToNaturoid} for the braided vector space $\kappa R$ from \cref{RacksToBraidedVS}.

\end{rem}

\begin{exmp} \label{ProductRackPartition}

Let $R$ be a rack, let $\kappa$ be a field, and let $c$ be a $2$-cocycle on $R \times \mathcal T_2$ valued in $\kappa^\times$. We will use the notation
\[
R_\varphi = R \times \{\varphi\}, \qquad 
R_\psi = R \times \{\psi\}, \qquad
c_\varphi = c|_{R_\varphi \times R_{\varphi}}, \qquad
c_\psi = c|_{R_\psi \times R_{\psi}},
\]
The racks $R_\varphi$ and $R_\psi$ are isomorphic to $R$ and form a partition of $R \times \mathcal T_2$ in the sense of \cref{PartitionRackDef}.
    
\end{exmp}

\begin{defn} \label{CocycleArisesFromPartition}

Let $R$ be a rack with a partition $S,T \subseteq R$, and let $\kappa$ be a field.
We say that a function $c \colon R \times R \to \kappa^\times$ arises from the partition $S,T \subseteq R$ if the restriction of $c$ to each of the four subsets $S \times S$, $S \times T$, $T \times S$, and $T \times T$ of $R \times R$ is a constant function.

\end{defn}

A function that arises from a partition is necessarily a $2$-cocycle.

\begin{exmp}

The two $2$-cocycles $c_\wedge$ and $c_{\pm}$ on $R \times \mathcal T_2$, introduced in \cref{CocycleWedgePMexmp}, arise from the partition in \cref{ProductRackPartition}.

\end{exmp}

\begin{prop} \label{BnActionBVsCocycleNabla}

Let $R$ be a rack, let $\kappa$ be a field, let $c \colon R \times R \to \kappa^\times$ be a $2$-cocycle, let $(x_1, \dots, x_n) \in R^n$, and let $g \in B_n$.
Then there exists a unique element $\nabla(x_1, \dots, x_n;g) \in \kappa^\times$ for which
\[
(x_1 \otimes \dots \otimes x_n)^g = \nabla(x_1, \dots, x_n;g) \cdot y_1 \otimes \dots \otimes y_n, \qquad (y_1, \dots, y_n) = (x_1, \dots, x_n)^g \in R^n,
\]
where the action of $g$ on $x_1 \otimes \dots \otimes x_n$ in $\kappa R(c)^{\otimes n}$ comes from \cref{RacksToBraidedVS} and \cref{BraidedToNaturoid}, while the action of $B^n$ on $R^n$ comes from \cref{BraidedSetRack}.
    
\end{prop}

\begin{proof}

Induct on the length of $g$ as a word in the symmetric generating set 
\begin{equation} \label{SymmGenSetForBn}
\mathcal S = \{\sigma_i^\epsilon : 1 \leq i \leq n-1, \quad  \epsilon \in \{\pm 1\}\}
\end{equation}
for $B_n$, with the base case of length $1$ being \cref{CocycleToBVs}.
\end{proof}

\begin{exmp} \label{TwoExamplesForNabla}

For $1 \leq i \leq n-1$ we have 
\begin{equation*} 
\nabla(x_1, \dots, x_n;\sigma_i) = c(x_i, x_{i+1}), \qquad
\nabla(x_1, \dots, x_n;\sigma_i^2) = c(x_i, x_{i+1}) \cdot c(x_{i+1}, x_i^{x_{i+1}}).
\end{equation*}


\end{exmp}

In the $\kappa$-algebra of coinvariants of $\kappa R(c)$, introduced in \cref{DefAlgCoinvs}, we have
\begin{equation} \label{CoinvariantsUsingNabla} 
x_1 \cdots x_n = \nabla(x_1, \dots, x_n;g) \cdot y_1 \cdots y_n
\end{equation}
where the notation here is that of \cref{BnActionBVsCocycleNabla}, namely in the action of $B_n$ on $R^n$ from \cref{BraidedSetRack} we have $(x_1, \dots, x_n)^g = (y_1, \dots, y_n)$.

Let $V$ be a braided vector space over a field $\kappa$.
The representation of $B_n$ on $V^{\otimes n}$ from \cref{BraidedToNaturoid} gives rise to an action of $B_n$ on $\mathbb P V^{\otimes n}$ (the associated projective representation).
We view this as an action of a group on a set.

\begin{cor} \label{ProjectiveBraidRep}

The action of $B_n$ on $\mathbb P \kappa R(c))^{\otimes n}$ preserves (setwise) the (lines spanned by the) pure tensors of elements of $R$.
This action of $B^n$ (on a set) is isomorphic to the action of $B^n$ on $R^n$ described in \cref{BraidedSetRack}.

\end{cor}

\begin{proof}

Follows at once from \cref{BnActionBVsCocycleNabla}.
\end{proof}

\begin{prop} \label{RackifyingCocycle}

Let $\kappa$ be a field, let $A \leq \kappa^\times$ be a subgroup, let $R$ be a rack, and let $c \colon R \times R \to A$ be a $2$-cocycle.
Then
\begin{equation} \label{DefExtensionRackCocycle}
(r,a)^{(s,b)} = (r^s, c(r,s) \cdot a), \qquad r,s \in R, \quad a,b \in A
\end{equation}
endows $R \times A$ with a rack structure which we denote by $R_c$.
This association $(R,c) \mapsto R_c$ extends to a functor from the category of racks equipped with $A$-valued $2$-cocycles to the category of racks.
Also, the rack $R_c$ is a quandle if and only if $R$ is a quandle and $c(r,r) = 1$ for every $r \in R$.

Moreover, if $A$ is finite (so $c$ is cyclotomic) we have an injective morphism of braided vector spaces from $\kappa R(c)$ to $\kappa R_c$ mapping every $r \in R$ to
\begin{equation} \label{EmbBVScocycleRack}
\sum_{a \in A} a^{-1} \cdot (r,a) \in \kappa R_c.
\end{equation}

Suppose at last that $f \colon \kappa R(c) \to V$ is a morphism of braided vector spaces over $\kappa$.
Then there exists a morphism of braided vector spaces $\kappa R_c \to V$ that fits into the commutative diagram
\begin{equation} \label{UniversalityCocycleRackBVS}
\begin{tikzcd}
	{\kappa R_c} \\
	{\kappa R(c)} & V.
	\arrow[from=1-1, to=2-2]
	\arrow[from=2-1, to=1-1]
	\arrow["f"', from=2-1, to=2-2]
\end{tikzcd}
\end{equation}
    
\end{prop}

\begin{proof}

The fact that \cref{DefExtensionRackCocycle} gives a rack, and that the construction is functorial, is immediate from \cref{RackDef} and \cref{RacksWithCocyclesDef}.
The rack $R_c$ is a quandle if and only if for every $r \in R$ and $a \in A$ the element
\[
(r,a)^{(r,a)} = (r^r, c(r,r) \cdot a)
\]
equals $(r,a)$.
The equality on the first coordinate is equivalent to $R$ being a quandle, and the equality on the second to $c(r,r) = 1$.

Suppose from now on that $A$ is finite.
It is clear that \cref{EmbBVScocycleRack} describes an injective $\kappa$-linear map, so it remains to check that this map intertwines the braidings in accordance to \cref{def:braidedobject}.
This check can be performed on $R \times R$ viewed as a basis for $\kappa R(c) \otimes \kappa R(c)$.

For $(r,s) \in R \times R$, applying first the braiding and then the map from \cref{EmbBVScocycleRack} we get
\[
c(r,s) \cdot \left( \sum_{a \in A} a^{-1} \cdot (s,a) \right) \otimes \left( \sum_{b \in A} b^{-1} \cdot (r^s,b) \right) =
c(r,s) \cdot \sum_{(a,b) \in A \times A} a^{-1}b^{-1} \cdot (s,a) \otimes (r^s,b),
\]
whereas applying first the map from \cref{EmbBVScocycleRack} and then the braiding gives
\[
\sum_{(\alpha, \beta) \in A \times A} \alpha^{-1} \beta^{-1} \cdot (s,\beta) \otimes (r^s, c(r,s) \cdot \alpha) = c(r,s) \cdot \sum_{(\gamma, \beta) \in A \times A} \gamma^{-1} \beta^{-1} \cdot (s, \beta) \otimes (r^s,\gamma), 
\]
so interchanging $a$ with $\beta$, and $b$ with $\gamma$, we arrive at the required equality.    

At last we claim that the $\kappa$-linear map from $\kappa R_c$ to $V$ that sends $(r,a) \in R_c$ to $|A|^{-1}af(r) \in V$ (where we think of $r$ as an element of $\kappa R(c)$) is a morphism of braided vector spaces over $\kappa$ that makes the diagram in \cref{UniversalityCocycleRackBVS} commute.
The commutativity is immediate from \cref{EmbBVScocycleRack}, so it remains to check for $(r,a), (s,b) \in R_c$ that braiding $(r,a) \otimes (s,b) \in \kappa R_c \otimes \kappa R_c$ and applying the tensor product with itself of the linear map just constructed gives the same vector in $V \otimes V$ as applying this tensor product to $(r,a) \otimes (s,b) \in \kappa R_c \otimes \kappa R_c$ and then braiding (in $V \otimes V$).
Indeed, that same vector in $V \otimes V$ is $|A|^{-2}\cdot a \cdot b \cdot c(r,s) \cdot f(s) \otimes f(r^s)$ as can be seen from \cref{CocycleToBVs}, \cref{DefExtensionRackCocycle}, and our assumption that $f$ is a morphism of braided vector spaces over $\kappa$.
\end{proof}

\subsubsection{Associated Groups}

\begin{defn} \label{DefRackGroup}

To a rack $R$ we functorially associate the group given by the presentation
\[
\Gamma_R = \langle R : s^{-1}rs = r^s \ \text{for all} \ (r,s) \in R^2 \rangle.
\]
This group acts from the right on $R$, and is sometimes called the structure group of $R$.

\end{defn}

This functor $R \mapsto \Gamma_R$ is left adjoint to the forgetful functor from the category of groups to the category of racks described in \cref{ForgetFulGroupsRacks}.

\begin{exmp} \label{StructureGroupRackInGroup}

Let $R$ be a conjugacy-closed subset of a group $G$.
The aforementioned adjunction provides us with a group homomorphism $\Gamma_R \to G$ via which the inclusion of $R$ into $G$ factors.
The group $G$ acts on $R$ from the right by conjugation, and the action of $\Gamma_R$ on $R$ factors via the group homomorphism $\Gamma_R \to G$.

\end{exmp}

\begin{exmp} \label{StructureGroupEmptyRack}

The structure group of the empty rack is trivial.
    
\end{exmp}

\begin{exmp} \label{TrivRackStructGroup}

The structure group of $\mathcal T_1$ is (isomorphic to) $\Z$.
    
\end{exmp}

\begin{exmp} \label{DirectProdStructGroupsRack}

Let $R,S$ be racks. Then the structure group of $R \coprod S$ is the direct product of the structure group of $R$ and the structure group of $S$.

\end{exmp}

\begin{defn}

Let $R$ be a rack.
We term the orbits of the action of $\Gamma_R$ on $R$ the connected components of $R$.
We say that $R$ is connected if it has exactly one connected component, or equivalently, if the action of $\Gamma_R$ on $R$ is transitive.
    
\end{defn}

The connected components of $R$ form the finest partition of $R$ in the sense of \cref{MultiPartitionExmpAction}.
The empty rack is not connected.

\begin{exmp} \label{SingleConjugacyClassConnected}

A conjugacy-closed subset $R$ of a group $G$ is a (single) conjugacy class of $G$ if it is connected (as a rack).
A generating conjugacy class $R$ of a group $G$ is connected (as a rack).
    
\end{exmp}

\begin{exmp} \label{PermutationalRackConnected}

The rack $\Z/m\Z$ from \cref{PermutationalRack} is connected.
    
\end{exmp}

\begin{exmp} \label{RackifyingCpm}

The connected components of the Joyce quandle $\mathcal S_{\pm}$ from \cref{JoyceQuandle} are $\{J_1, J_2\}$ and $\{J_3\}$.
Next we check that the structure group of $\mathcal S_{\pm}$ is abelian.
Indeed conjugation in $\mathcal S_{\pm}$ by $J_1$ and $J_2$ is the identity map, and $J_3$ is invariant under conjugation by any element of $\mathcal S_{\pm}$. It follows that the images in the structure group of $\mathcal S_{\pm}$ of any two elements of $\mathcal S_{\pm}$ commute, so our check is complete.
A similar argument shows that the structure group of the quandle $(\mathcal T_2)_{c_{\pm}}$ constructed in \cref{RackifyingCocycle} is abelian as well.

For a field $\kappa$ of characteristic different from $2$ we have an injective homomorphism $\kappa_{\pm} \to \kappa \mathcal S_{\pm}$ of braided vector spaces defined by
\[
v \mapsto J_3, \quad \underline{v} \mapsto J_1 - J_2
\]
on the basis for $\kappa_{\pm}$ from \cref{LegendreBraidVS}.








\end{exmp}

The abelianization $\Gamma_R^{\textup{ab}}$ of the structure group $\Gamma_R$ of a rack $R$ is canonically isomorphic to the free abelian group on the set of connected components of $R$.

\begin{exmp} \label{RackificationOneDimensional}

Let $n$ be a positive integer, let $\kappa$ be a field of characteristic not dividing $n$, and let $\zeta \in \kappa$ be a primitive $n$th root of unity.
Endow $\Z/n\Z$ with the rack structure from \cref{PermutationalRack}.
Then the $\kappa$-linear map from $\kappa_\zeta$ to $\kappa[\Z/n\Z]$ sending the fixed nonzero $v_\zeta \in \kappa_\zeta$ from \cref{ScalingBraidedVS} to 
\[
\sum_{r \in \Z/n\Z} \zeta^{-r} \cdot r \in \kappa[\Z/n\Z]
\]
is an injective morphism of braided vector spaces in view of \cref{RackifyingCocycle} applied to $R = \mathcal T_1$, $A = \{\alpha \in \kappa : \alpha^n = 1\}$, and $c \colon R \times R \to A$ the $2$-cocycle that maps the unique pair in $R \times R$ to $\zeta$.
In particular, the rack $\mathbb Z/n \mathbb Z$ above becomes isomorphic to the rack $(\mathcal T_1)_c$ from \cref{RackifyingCocycle} once we identify $A$ with $\Z/n\Z$ by sending a residue class $r$ of integers modulo $n$ to $\zeta^r \in A$.

Next we check that the structure group of this rack is abelian (in fact, cyclic).
For that matter, it would be enough to show that the images of all elements of $\Z/n\Z$ in the structure group of this rack coincide.
Indeed for $a \in \Z/n\Z$ we have
\[
a = a^a = a+1
\]
in the structure group of $\Z/n\Z$, so the required coincidence follows from the fact that $1$ is a generator of the additive group of residue classes of integers modulo $n$. 
 
\end{exmp}

The inclusion of a subrack $S$ into a rack $R$ induces a group homomorphism $\Gamma_S \to \Gamma_R$, and thus a right action of $\Gamma_S$ on $R$. 

\begin{lem} \label{SequalsRtransitive}

Suppose that $S \neq \emptyset$ and that the action of $\Gamma_S$ on $R$ is transitive.
Then $S = R$.
    
\end{lem}

\begin{proof}

Let $r \in R$. 
Our task is to show that $r \in S$.
Our assumption that $S$ is nonempty allows us to pick $s \in S$.
The transitivity of the action of $\Gamma_S$ on $R$ gives us $\gamma \in \Gamma_S$ with $s^\gamma = r$.
Since $S$ is a subrack of $R$ we have $s^\gamma \in S$, hence $r \in S$ as required.
\end{proof}

\begin{defn}

For a rack $R$ denote by $\operatorname{Inn}(R)$ the subgroup of $\operatorname{Aut}(R)$ generated by the automorphisms $x \mapsto x^y$, with $y$ ranging over $R$, and $x \in R$. 
We call $\operatorname{Inn}(R)$ the inner automorphism group of $R$.

\end{defn}

The map sending $y \in R$ to the automorphism $x \mapsto x^y$ in $\operatorname{Inn}(R)$ is a homomorphism of racks, so the adjunction mentioned after \cref{DefRackGroup} tells us that this map factors via the homomorphism of racks $R \to \Gamma_R$.
Therefore, the action of the structure group $\Gamma_R$ on $R$ factors via $\operatorname{Inn}(R)$.
Consequently, the connected components of $R$ are the orbits of the action of $\operatorname{Inn}(R)$ on $R$.
The kernel of the group homomorphism $\Gamma_R \to \operatorname{Inn}(R)$ is the collection of all elements in $\Gamma_R$ that act trivially on $R$.

\begin{exmp}

The group of inner automorphisms of a trivial rack is trivial.
    
\end{exmp}

\begin{exmp}

The group of inner automorphisms of the rack $\mathbb Z/n \mathbb Z$ from \cref{PermutationalRack} is isomorphic to the group $\mathbb Z/n\mathbb Z$.
    
\end{exmp}

\begin{exmp} \label{InnerAutJoyceQuandle}

The group of inner automorphisms of $\mathcal{S}_{\pm}$ is isomorphic to $\mathbb Z/2\mathbb Z$.
    
\end{exmp}

\begin{exmp}

For racks $R$, $S$ we have
$
\operatorname{Inn}(R \coprod S) \cong \operatorname{Inn}(R) \times \operatorname{Inn}(S).
$
    
\end{exmp}

\begin{exmp}

The group of inner automorphisms of $\mathcal{S}_{\wedge}$ is isomorphic to $\mathbb Z/2\mathbb Z$.
    
\end{exmp}

\begin{exmp}

Let $R$ be a conjugacy-closed generating set of a group $G$.
Then the inner automorphism group of $R$ is isomorphic to the quotient of $G$ by its center.

\end{exmp}

\begin{lem} \label{GenerateRackCriterion}

Let $X$ be a subset of a rack $R$.
Then $X$ generates $R$ if and only if $X$ contains at least one element from each connected component of $R$ and the image of $X$ in $\operatorname{Inn}(R)$ is a generating set.
    
\end{lem}

\begin{proof}

Suppose first that $X$ generates $R$.
Since unions of connected components of $R$ are subracks of $R$, it follows that no connected component of $R$ is disjoint from $X$.
Because the image of $R$ in $\operatorname{Inn}(R)$ is a generating set, and $X$ generates $R$, it follows that the image of $X$ in $\operatorname{Inn}(R)$ is a generating set as well.
The proof of one implication is thus complete.

For the other direction, suppose now that $X$ meets each connected component of $R$ and that its image generates $\operatorname{Inn}(R)$.
Let $r \in R$, and denote by $S$ the subrack of $R$ generated by $X$.
Our task is to show that $r \in S$.
By assumption, there exists $x \in X$ that lies in the connected component of $r$.
Our assumption that $X$ generates $\operatorname{Inn(R)}$ implies then that there exist $y_1, \dots, y_n \in X$ and $\epsilon_1, \dots, \epsilon_n \in \{\pm 1\}$ such that
\[
x^w = r, \quad w = \prod_{i=1}^n y_i^{\epsilon_i} \in \operatorname{Inn}(R).
\]
Since $X \subseteq S$ we get that $x,y_1,\dots,y_n \in S$.
Because $S$ is a subrack of $R$, it follows that $r \in S$.
\end{proof}

The following is a special case of \cite{shus}*{Definition 4.28}.

\begin{defn} \label{SynchDefRacks}

We say that two racks $R$ and $S$ are synchronized if for every connected component $C$ of $R$ and every connected component $D$ of $S$ the subset $C \times D$ of the rack $R \times S$ is a connected component, namely the action of the structure group of $R \times S$ on $C \times D$ is transitive.
    
\end{defn}

\begin{lem} \label{SynchQuandle}

Suppose that $R$ is a quandle.
Then $R$ and $S$ are synchronized.

\end{lem}

\begin{proof}

For $r_1, r_2$ lying in the same connected component of $R$, and $s_1, s_2$ lying in the same connected component of $S$, we should show that $(r_1, s_1)$ and $(r_2, s_2)$ lie in the same connected component of $R \times S$.
Our assumption that $r_1, r_2$ lie in the same connected component implies that there exists $s$ in the connected component of $s_1$ such that $(r_1,s_1)$ and $(r_2, s)$ lie in the same connected component of $R \times S$.
Our assumption that $s_1, s_2$ lie in the same connected component implies that $s$ and $s_2$ lie in the same connected component.
This means that there exist $x_1, \dots, x_n \in S$ and $\epsilon_1, \dots, \epsilon_n \in \{\pm 1\}$ for which
\[
s^{w} = s_2, \quad w = \prod_{i=1}^n x_i^{\epsilon_i}.
\]
Since $R$ is a quandle, we have $r_2^{r_2} = r_2$, so
\[
(r_2,s)^u = (r_2, s_2), \quad u = \prod_{i=1}^n (r_2, x_i^{\epsilon_i}).
\]
It follows that $(r_1, s_2)$ and $(r_2, s_2)$ indeed lie in the same connected component.    
\end{proof}

\begin{lem} \label{InnerProduct}

Let $R$ and $S$ be racks.
Then the homomorphism of racks $R \times S \to \operatorname{Inn}(R)$ factors via $\operatorname{Inn}(R \times S)$. In other words, we have a commutative diagram
\[\begin{tikzcd}
	{R \times S} & {\operatorname{Inn}(R \times S)} \\
	R & {\operatorname{Inn}(R)}
	\arrow[from=1-1, to=1-2]
	\arrow[from=1-1, to=2-1]
	\arrow[from=1-2, to=2-2]
	\arrow[from=2-1, to=2-2]
\end{tikzcd} 
\]
where the right vertical arrow is a homomorphism of groups.

\end{lem}

\begin{proof}

The homomorphism of racks $R \times S \to \operatorname{Inn}(R)$ factors, in view of the adjunction mentioned after \cref{DefRackGroup}, via a group homomorphism $\Gamma_{R \times S} \to \operatorname{Inn}(R)$.
Our task is to show that this group homomorphism factors via the group homomorphism $\Gamma_{R \times S} \to \operatorname{Inn}(R \times S)$.
For that matter we take an element
\[
g = \prod_{i=1}^n (r_i,s_i)^{\epsilon_i} \in \Gamma_{R \times S}, \quad (r_i,s_i) \in R \times S, \quad \epsilon_i \in \{\pm 1\},
\]
that acts trivially on $R \times S$ with the purpose of checking that $g$ lies in the kernel of the group homomorphism $\Gamma_{R \times S} \to \operatorname{Inn}(R)$.
That homomorphism maps $g$ to
$
\prod_{i=1}^n r_i^{\epsilon_i} \in \operatorname{Inn}(R)
$
which acts as the identity on $R$ because $g$ acts as the identity on $R \times S$.
This concludes our verification that $g$ lies in the kernel of the homomorphism $\Gamma_{R \times S} \to \operatorname{Inn}(R)$.
\end{proof}

\cref{InnerProduct} provides us with an injective group homomorphism $\operatorname{Inn}(R \times S) \to \operatorname{Inn}(R) \times \operatorname{Inn}(S)$.

\begin{lem} \label{NonactingElementImpliesInnerProductive}

Suppose that $R$ is nonempty, and that there exists an $s \in S$ such that for every $x \in S$ we have $x^s = x$.
Then the group homomorphism $\operatorname{Inn}(R \times S) \to \operatorname{Inn}(R) \times \operatorname{Inn}(S)$ is an isomorphism.
    
\end{lem}

\begin{proof}

Denote by $H$ the image of our group homomorphism.
Our task is to show that $H = \operatorname{Inn}(R) \times \operatorname{Inn}(S)$.
Our assumption that $R$ is nonempty implies that the the restriction of the projection $\operatorname{Inn}(R) \times \operatorname{Inn}(S) \to \operatorname{Inn}(S)$ to $H$ is surjective.
To conclude the argument it is sufficient (and necessary) to show that $H$ contains the kernel $\operatorname{Inn}(R) \times \{\mathrm{id}_S\}$ of this projection.

Since $H$ is a subgroup of $\operatorname{Inn}(R) \times \operatorname{Inn}(S)$, it is sufficient (and necessary) to show that $H$ contains the generating set $\{(r,\mathrm{id}_S) \in \operatorname{Inn}(R) \times \{\mathrm{id}_S\} : r \in R\}$ of $\operatorname{Inn}(R) \times \{\mathrm{id}_S\}$.
Indeed, because $x^s = x$ for every $x \in S$, the image of $(r,s) \in \operatorname{Inn}(R \times S)$ under the group homomorpism $\operatorname{Inn}(R \times S) \to \operatorname{Inn}(R) \times \operatorname{Inn}(S)$ is $(r, \mathrm{id}_S)$ for every $r \in R$, and it lies in $H$ by definition.
\end{proof}

\begin{lem} \label{InnerProductiveImpliesSynchronized}

Suppose that the group homomorphism $\operatorname{Inn}(R \times S) \to \operatorname{Inn}(R) \times \operatorname{Inn}(S)$ is an isomorphism.
Then $R$ and $S$ are synchronized.
    
\end{lem}

\begin{proof}

For $r \in R$, $s \in S$, $g \in \operatorname{Inn}(R)$, $h \in \operatorname{Inn}(S)$, we should show that $(r,s)$ and $(r^g, s^h)$ lie in the same orbit in $R \times S$ under the action of $\operatorname{Inn}(R \times S)$.
Our assumption that the group homomorphism $\operatorname{Inn}(R \times S) \to \operatorname{Inn}(R) \times \operatorname{Inn}(S)$ is surjective provides us with $\tau \in \operatorname{Inn}(R \times S)$ that maps to $(g,h)$ in $\operatorname{Inn}(R) \times \operatorname{Inn}(S)$.
We conclude that $(r,s)^{\tau} = (r^g,s^h)$ as required.   
\end{proof}

\begin{cor} \label{GoursatRacks}

Let $R$ be a connected rack, and let $S$ be a rack for which the group $\operatorname{Inn}(S)$ is abelian.
Suppose that there exists $s \in S$ such that for every $x \in S$ we have $x^s = x$.
Let $X$ be a subset of $R \times S$ whose projection to $R$ generates $R$, and whose projection to $S$ generates $S$.
Then $X$ generates $R \times S$.
    
\end{cor}

\begin{proof}

Our assumption that $R$ is connected implies that it is nonempty, so we get from \cref{NonactingElementImpliesInnerProductive} that the group homomorphism $\operatorname{Inn}(R \times S) \to \operatorname{Inn}(R) \times \operatorname{Inn}(S)$ is an isomorphism.
In particular, we deduce from \cref{InnerProductiveImpliesSynchronized} that $R$ and $S$ are synchronized.

We claim that for every connected component $C$ of $R \times S$ we have $C \cap X \neq \emptyset$.
Our assumption that $R$ is connected and synchronized with $S$ implies that there exists a connected component $D$ of $S$ such that $C = R \times D$.
The assumption that the projection of $X$ to $S$ generates it implies, in view of \cref{GenerateRackCriterion}, that this projection is not disjoint from $D$. It follows that $X$ is not disjoint from $C$ - our claim is established.

Next we claim that the image of $X$ in $\operatorname{Inn}(R \times S)$ is a generating set.
To prove this, we first identify $\operatorname{Inn}(R \times S)$ with $\operatorname{Inn}(R) \times \operatorname{Inn}(S)$ using the isomorphism from the first paragraph.
Our initial assumptions on $X$ imply, in view of \cref{GenerateRackCriterion}, that the image $X$ in $\operatorname{Inn}(R) \times \operatorname{Inn}(S)$ projects to a generating set of each of the factors.
Therefore, it follows from Goursat's lemma that there exists a group $G$ and surjective group homomorphisms $\varphi \colon \operatorname{Inn}(R) \to G$, $\psi \colon \operatorname{Inn}(S) \to G$ such that the subgroup of $\operatorname{Inn}(R) \times \operatorname{Inn}(S)$ generated by the image of $X$ is
\[
\{(\alpha,\beta) \in \operatorname{Inn}(R) \times \operatorname{Inn}(S) : \varphi(\alpha) = \psi(\beta)\}.
\]

Our assumption that $\operatorname{Inn}(S)$ is abelian implies that $G$ is abelian as well.
Since $R$ is connected, its elements are conjugate in $\Gamma_R$ and thus also in $\operatorname{Inn}(R)$, so their images in $G$ are conjugate as well because $\varphi$ is surjective.
As $G$ is abelian, conjugate elements in $G$ are equal, so all elements of $R$ map to the same $g \in G$.
Since $R$ generates $\operatorname{Inn}(R)$, its image in $G$ generates it, so $g$ generates $G$.

Take an $r \in R$. It follows from the previous claim that $X$ contains an element $\pi$ from the connected component of $(r,s)$ in $R \times S$.
It follows from our initial assumption on $s$ that the projection to $\operatorname{Inn}(S)$ of the image of $\pi$ in $\operatorname{Inn}(R) \times \operatorname{Inn}(S)$ is the identity map on $S$.
We conclude that 
$
g = \psi(\mathrm{id}_S) = 1
$
and thus $G$ is trivial. 
This concludes the proof of our claim that the image of $X$ in $\operatorname{Inn}(R \times S)$ generates it.

The two claims we have proven, in conjunction with \cref{GenerateRackCriterion}, guarantee that $X$ generates $R \times S$ as required.
\end{proof}

\begin{lem} \label{IdealRackNormalSubgroup}

Let $R$ be a rack, and let $S$ be an ideal of $R$.
Then the image of the group homomoprhism $\Gamma_S \to \Gamma_R$ is a normal subgroup of $\Gamma_R$.
    
\end{lem}

\begin{proof}

It is sufficient (and necessary) to check that every $r \in R$ lies in the normalizer of the image of $\Gamma_S$ in $\Gamma_R$.
This follows at once from the fact that $s^r \in S$ fro every $s \in S$ which is a consequene of our assumption that $S$ is an ideal of $R$.
\end{proof}

\begin{lem}

Let $R$ be a rack, and let $S$ be an ideal of $R$.
Then there exists a unique rack structure on the set $R/S = R/\Gamma_S$ of orbits in $R$ under the action of $\Gamma_S$ for which the quotient map $R \to R/S$ is a homomorphism of racks.
    
\end{lem}

We say that the rack $R/S$ is the quotient of $R$ by $S$.

\begin{proof}

Uniqueness is clear from the surjectivity of the quotient map.
For existence we need to check two things.
First we need to check that for every $x,y \in R$ and $g \in \Gamma_S$ the elements $(x^g)^y$ and $x^y$ of $R$ lie in the same $\Gamma_S$-orbit. Since $S$ is an ideal of $R$, \cref{IdealRackNormalSubgroup} tells us that the image of $\Gamma_S$ in $\Gamma_R$ contains $y^{-1}gy$ so $x^y$ lies in the same $\Gamma_S$-orbit with
\[
(x^y)^{y^{-1}gy} = x^{yy^{-1}gy} = x^{gy} = (x^g)^y
\]
as required.

Second we need to check that for every $x,y \in R$ and $g \in \Gamma_S$ the elements $x^{y^g}$ and $x^y$ of $R$ lie in the same $\Gamma_S$-orbit. Since $S$ is an ideal of $R$, \cref{IdealRackNormalSubgroup} guarantees that the image of $\Gamma_S$ in $\Gamma_R$ contains $y^{-1}g^{-1}yg$ so $x^y$ lies in the same $\Gamma_S$-orbit with
\[
(x^y)^{y^{-1}g^{-1}yg} = x^{yy^{-1}g^{-1}yg} = x^{g^{-1}yg} = x^{y^g}
\]
as required.
We have thus checked for $x,y \in R$ that the $\Gamma_S$-orbit of $x^y$ depends only on the $\Gamma_S$-orbit of $x$ and on the $\Gamma_S$-orbit of $y$, endowing $R/S$ with the sought rack structure. 
\end{proof}

\begin{exmp}

For a rack $R$, the fibers of the map $R \to R/R$ are the connected components of $R$.
Following \cite{shus}*{Definition 4.9} we call $R/R$ the trivialization of $R$, and denote it by $R_{\mathrm{triv}}$.
The association to a rack $R$ of the set $R_{\mathrm{triv}}$ is a functor left adjoint to the functor that endows a set with the trivial rack structure.
    
\end{exmp}

\begin{exmp} \label{QuotientJoyceQuandle}

The quotient of $\mathcal S_\pm$ by the ideal $\{J_3\}$ is isomorphic to the trivial rack $\mathcal T_2$.
More generally, arguing as in the proof of \cref{SynchQuandle}, we see that for a connected quandle $R$ the quotient of the rack $R \times \mathcal S_{\pm}$ by the ideal $R \times \{J_3\}$ is isomorphic to the trivial rack $\mathcal T_2$.
    
\end{exmp}


    


    






    


    

\subsection{Addability} \label{Addability}

\begin{defn} \label{DefAddablePair}

Let $\mathcal C$ be an additive monoidal category, namely for every object $X$ of $\mathcal C$ the functors $X \otimes -$ and $- \otimes X$ are additive.
We say that a pair $(U,V)$ of braided objects in $\mathcal C$ is addable if it is equipped with isomorphisms $T_{U,V} \colon U \otimes V \to V \otimes U$ and $T_{V,U} \colon V \otimes U \to U \otimes V$ in $\mathcal C$ such that the diagrams
\[\begin{tikzcd}
	{U \otimes U \otimes V} && {U \otimes U \otimes V} && {U \otimes V \otimes U} \\
	&&&& {} \\
	{U \otimes V \otimes U} && {V \otimes U \otimes U} && {V \otimes U \otimes U} \\ \\
	{U \otimes U \otimes V} && {U \otimes U \otimes V} && {U \otimes V \otimes U}
	\arrow["{T_U \otimes \mathrm{id}_V}", from=1-1, to=1-3]
	\arrow["{\mathrm{id}_U \otimes T_{U,V}}", from=1-1, to=3-1]
	\arrow["{\mathrm{id}_U \otimes T_{U,V}}", from=1-3, to=1-5]
	\arrow["{T_{U,V} \otimes \mathrm{id}_U}", from=1-5, to=3-5]
	\arrow["{T_{U,V} \otimes \mathrm{id}_U}", from=3-1, to=3-3]
	\arrow["{\mathrm{id}_V \otimes T_{U}}", from=3-3, to=3-5]
    \arrow["{\mathrm{id}_U \otimes T_{V,U}}", from=3-1, to=5-1]
    \arrow["{T_{V,U} \otimes \mathrm{id}_U}", from=3-5, to=5-5]
    \arrow["{T_{U} \otimes \mathrm{id}_V}", from=5-1, to=5-3]
    \arrow["{\mathrm{id}_U \otimes T_{U,V}}", from=5-3, to=5-5]
\end{tikzcd}\]
and
\[\begin{tikzcd}
	{V \otimes U \otimes U} && {U \otimes V \otimes U} && {U \otimes U \otimes V} \\
	\\
	{V \otimes U \otimes U} && {U \otimes V \otimes U} && {U \otimes U \otimes V}
	\arrow["{T_{V,U} \otimes \mathrm{id}_U}", from=1-1, to=1-3]
	\arrow["{\mathrm{id}_V \otimes T_{U}}", from=1-1, to=3-1]
	\arrow["{\mathrm{id}_U \otimes T_{V,U}}", from=1-3, to=1-5]
	\arrow["{T_{U} \otimes \mathrm{id}_V}", from=1-5, to=3-5]
	\arrow["{T_{V,U} \otimes \mathrm{id}_U}", from=3-1, to=3-3]
	\arrow["{\mathrm{id}_U \otimes T_{V,U}}", from=3-3, to=3-5]
\end{tikzcd}\]
commute, and the diagrams obtained from these by interchanging $U$ with $V$ everywhere commute as well.

\end{defn}

Note that addability is not a property of a pair of braided objects, rather an extra structure (satisfying some properties).

\begin{defn} \label{MorphismAddablePairBVS}

A morphism from an addable pair $(U,V)$ of braided objects to an addable pair $(X,Y)$ of braided objects in $\mathcal C$ is a pair of morphisms $(f \colon U \to X, g \colon V \to Y)$ of braided objects in $\mathcal C$ such that the diagram
\[\begin{tikzcd}
	{U \otimes V} && {V \otimes U} && {U \otimes V} \\
	{X \otimes Y} && {Y \otimes X} && {X \otimes Y}
	\arrow["{T_{U,V}}"', from=1-1, to=1-3]
	\arrow["{f \otimes g}", from=1-1, to=2-1]
	\arrow["{T_{V,U}}"', from=1-3, to=1-5]
	\arrow["{g \otimes f}"', from=1-3, to=2-3]
	\arrow["{f \otimes g}"', from=1-5, to=2-5]
	\arrow["{T_{X,Y}}", from=2-1, to=2-3]
	\arrow["{T_{Y,X}}", from=2-3, to=2-5]
\end{tikzcd}\]
commutes.
    
\end{defn}

Addable pairs of braided objects in $\mathcal C$ form a category.

\begin{exmp} \label{AutoAddablePair}

Let $V$ be a braided object in $\mathcal C$. Then $(V,V)$ becomes an addable pair of braided objects if we set $T_{V,V} = T_V$.
    
\end{exmp}

\begin{defn} \label{DirectSumBVSDef}

To an addable pair $(U,V)$ of braided objects in $\mathcal C$ we functorially associate the object $U \oplus V$ in $\mathcal C$ with braiding defined as follows.
Identify $(U \oplus V) \otimes (U \oplus V)$ with 
\[
(U \otimes U) \oplus (V \otimes V) \oplus (U \otimes V) \oplus (V \otimes U)
\]
and let $T_{U \oplus V}$ be $T_U$ on the first summand, be $T_V$ on the second summand, be $T_{U,V}$ on the third summand, and be $T_{V,U}$ on the fourth summand.

\end{defn}

Note that the direct sum of an addable pair of braided objects depends not just on the braided objects but also on the extra structure carried by the addable pair.

On the category of addable pairs of braided objects in $\mathcal C$ we have an involution $(U,V) \mapsto (V,U)$ and $U \oplus V \cong V \oplus U$ as braided objects in $\mathcal C$.

The natural inclusion and projection maps
\begin{equation} \label{DirectSumBVSfunctoriality}
U \to U \oplus V, \qquad V \to U \oplus V, \qquad U \oplus V \to U, \qquad U \oplus V \to V,
\end{equation}
are morphisms of braided objects.


\begin{exmp} \label{PlainDirectSumEx}

Let $U$ and $V$ be braided objects in a symmetric monoidal additive category $\mathcal C$.
Then $(U,V)$ becomes an addable pair of braided objects if, in the notation of \cref{SwapSymmetricMonoidal}, we set $T_{U,V} = s_{U,V}$ and $T_{V,U} = s_{V,U}$. We then say that the braided object $U \oplus V$ is the plain direct sum of $V$ and $W$.
With the operation of plain direct sum, the category of braided objects in $\mathcal C$ is symmetric monoidal, the zero object of $\mathcal C$ being the unit for plain direct sums.

If $U$ and $V$ are permutational in the sense of \cref{PermBraidObj}, then so is their plain direct sum $U \oplus V$.    

\end{exmp}

\begin{exmp} \label{PlainDirectSumWedgeDec}

For a field $\kappa$ we have the plain direct sum decomposition of permutational braided vector spaces $\kappa_\wedge = \kappa_1 \oplus \kappa_{-1} = \kappa \oplus \kappa_{-1}$ introduced in \cref{ScalingBraidedVS} and \cref{IntroducingKwedge}.
    
\end{exmp}

\begin{exmp} \label{WeightedDirectSumExmp}

Let $\mathcal C$ be a symmetric monoidal additive category with unit $I$.
Let $\alpha, \beta \in \operatorname{Aut}(I)$ be commuting elements which we call weights. If $\mathcal C$ is the category of vector spaces over a field $\kappa$, then $\alpha,\beta \in \kappa^\times$.

Let $U$ and $V$ be braided objects in $\mathcal C$.
Then $(U,V)$ becomes an addable pair of braided objects if we let $T_{U,V}$ be the composition of $s_{U,V}$ with the identification $V \otimes U \to V \otimes U \otimes I$ and $\mathrm{id}_{V \otimes U} \otimes \alpha$, and similarly with  $U,V$ interchanged and $\alpha, \beta$ interchanged.
We denote the direct sum of this addable pair by $U \oplus_{\alpha, \beta} V$, and call it the weighted (by $\alpha, \beta$) direct sum of $U, V$.
The weighted direct sum is the plain one if and only if $\alpha = \beta  = \mathrm{id}_I$.

For instance, in the notation of \cref{LegendreBraidVS} we have
\[
\kappa \oplus_{1,-1} \kappa = \kappa \oplus_{\mathrm{id}_\kappa, - \mathrm{id}_\kappa} \kappa \cong \kappa_{\pm}
\]
as braided vector spaces over $\kappa$.
    
\end{exmp}

\begin{exmp} \label{AddableDecompositionRackCocyclePartition}

Let $S,T \subseteq R$ be a partition of a rack, let $\kappa$ be a field, and let $c \colon R \times R \to \kappa^\times$ be a $2$-cocycle. 
Then
\[
\kappa R(c) \cong \kappa S(c|_{S \times S}) \oplus \kappa T(c|_{T \times T})
\]
as braided vector spaces over $\kappa$, where we have taken the direct sum of the addable pair 
\[
(\kappa S(c|_{S \times S}), \kappa T(c|_{T \times T}))
\]
of braided vector spaces over $\kappa$ with respect to the maps
\[
\kappa S(c|_{S \times S}) \otimes \kappa T(c|_{T \times T}) \to \kappa T(c|_{T \times T}) \otimes \kappa S(c|_{S \times S}), \qquad s \otimes t \mapsto c(s,t) \cdot t \otimes s^t
\]
and
\[
\kappa T(c|_{T \times T}) \otimes \kappa S(c|_{S \times S}) \to \kappa S(c|_{S \times S}) \otimes \kappa T(c|_{T \times T}), \qquad t \otimes s \mapsto c(t,s) \cdot s \otimes t^s.
\]
In case $R$ is the disjoint union of $S$ and $T$ in the sense of \cref{DisjoinUnionRacks}, we get that $\kappa R$ is the plain direct sum of $\kappa S$ and $\kappa T$ in the sense of \cref{PlainDirectSumEx}.

\end{exmp}

\begin{exmp} \label{SawinRack}

With the notation of \cref{TrivialRack}, \cref{DisjoinUnionRacks}, and \cref{PermutationalRack}, we denote by $\mathcal S_{\wedge}$ the three-element rack $\mathcal T_1 \coprod (\Z/2\Z)$.
For a field $\kappa$ of characteristic different from $2$ we then have an injective morphism of braided vector spaces
\[
\kappa_{\wedge} = \kappa \oplus \kappa_{-1} \hookrightarrow \kappa \oplus \kappa[\Z/2\Z] = \kappa\mathcal T_1 \oplus \kappa[\Z/2\Z] = \kappa \left[\mathcal T_1 \coprod \Z/2\Z \right] = \kappa[\mathcal S_\wedge]
\]
arising from \cref{PlainDirectSumWedgeDec}, the direct sum of $\mathrm{id}_\kappa$ with the map from \cref{RackificationOneDimensional}, and the transformation of disjoint union into direct sum from \cref{AddableDecompositionRackCocyclePartition}.

It follows from \cref{TrivRackStructGroup}, \cref{RackificationOneDimensional}, and \cref{DirectProdStructGroupsRack} that the structure group of $\mathcal S_{\wedge}$ is abelian.
An argument similar to that in \cref{RackifyingCpm}, in conjunction with the fact that the image of a rack in its structure group is a quandle, can be used to show that the structure group of the rack $(\mathcal T_2)_{c_\wedge}$ from \cref{RackifyingCocycle} is abelian as well.

\end{exmp}

\begin{exmp} \label{PmPairsMor}

The morphism of addable pairs of braided vector spaces 
\[
(\operatorname{Span}_\kappa \{v\}, \operatorname{Span}_\kappa \{ \underline v\}) \to (\kappa[\{J_3\}], \kappa[\{J_1, J_2\}]), \qquad v \mapsto J_3, \quad \underline{v} \mapsto J_1 - J_2,
\]
over a field $\kappa$ described in terms of the notation of \cref{LegendreBraidVS} and \cref{AddableDecompositionRackCocyclePartition}, gives rise via the functoriality in \cref{MorphismAddablePairBVS} to the morphism of braided vector spaces in \cref{RackifyingCpm}.

\end{exmp}

\subsection{Tensor Products}

\begin{defn} \label{TensorProductBVSDef}

Let $\mathcal C$ be a symmetric monoidal category, and let $V,W$ be braided objects in $\mathcal C$. 
We define a braiding on $V \otimes W$ to be the morphism
\[
T_{V \otimes W} = (\mathrm{id}_V \otimes s_{V,W} \otimes \mathrm{id}_W) \circ (T_V \otimes T_W) \circ (\mathrm{id}_V \otimes s_{W,V} \otimes \mathrm{id}_W)
\]
from $V \otimes W \otimes V \otimes W$ to itself, using the notation from \cref{SwapSymmetricMonoidal}.
\end{defn}



The category of braided objects in $\mathcal C$ is thus a symmetric monoidal category with unit $I$ braided as in \cref{TrivialBraidingUnit}.

The construction in \cref{BraidedToNaturoid} preserves the monoidal structure, namely if $V$ and $W$ are braided objects in $\mathcal C$ then the homomorphism $B_n \to \operatorname{Aut}((V \otimes W)^{\otimes n})$ factors as
\[
B_n \to B_n \times B_n \to \operatorname{Aut}(V^{\otimes n}) \times \operatorname{Aut}(W^{\otimes n}) \to \operatorname{Aut}(V^{\otimes n} \otimes W^{\otimes n}) \cong \operatorname{Aut}((V \otimes W)^{\otimes n}) 
\]
where the first map is the diagonal.

For a field $\kappa$ the functor $(R,c) \mapsto \kappa R(c)$ from \cref{RacksToBraidedVS} is symmetric monoidal.

\begin{defn} \label{BVStwistedByZetaWedgePM}

For a braided vector space $W$ over a field $\kappa$ we introduce the braided vector spaces
\begin{equation*} 
W_{\zeta} = W \otimes \kappa_{\zeta} , \qquad W_\wedge = W \otimes \kappa_{\wedge}, \qquad W_{\pm} = W \otimes \kappa_{\pm}.
\end{equation*}

\end{defn}

If $R$ is a rack then by \cref{WedgePMCocycles} and the aformenioned monoidality we have
\begin{equation} \label{RwedgePMcocycles}
\begin{split}
&\kappa R_\wedge = \kappa R \otimes \kappa_{\wedge} = \kappa R \otimes \kappa \mathcal T_2(c_{\wedge}) = \kappa[R \times \mathcal T_2](c_\wedge), \\
&\kappa R_{\pm} = \kappa R \otimes \kappa \mathcal T_2(c_{\pm}) = \kappa [R \times \mathcal T_2](c_{\pm})
\end{split}
\end{equation}
as braided vector spaces over $\kappa$.
Similarly if $c$ is the $2$-cocycle on $\mathcal T_1$ with image $\zeta \in \kappa^\times$ then 
\begin{equation} \label{TwistingAsCocycle}
\kappa R_\zeta = \kappa R \otimes \kappa_\zeta = \kappa R \otimes \kappa \mathcal T_1(c) = \kappa R(c)
\end{equation}
as braided vectors spaces over $\kappa$, where (abusing notation) we view $c$ in the last expression as the $2$-cocycle on $R$ satisfying $c(x,y) = \zeta$ for all $x,y \in R$.

\begin{exmp} \label{WedgeTwistedRackExample}

Let $\kappa$ be a field, and $R$ be a rack. 
Following \cref{BVStwistedByZetaWedgePM} and \cref{IntroducingKwedge}, we have the vector space $\kappa R_\wedge$ over $\kappa$ with basis 
\[
\bigcup_{x \in R} \{x_1 = x \otimes v_1, \ x_{-1} = x \otimes v_{-1}\},
\]
braided by
\[
x_1 \otimes  y_1 \mapsto  y_1 \otimes x^y_1, \quad
x_1 \otimes y_{-1} \mapsto y_{-1} \otimes x^y_1, \quad
x_{-1} \otimes  y_1 \mapsto y_1 \otimes x^y_{-1}, \quad
x_{-1} \otimes y_{-1} \mapsto - y_{-1} \otimes x^y_{-1},
\]
for every $x,y \in R$.
    
\end{exmp}

\begin{exmp} \label{PMtwistBVS}

Let $\kappa$ be a field, and $R$ be a rack. 
Following \cref{BVStwistedByZetaWedgePM} and \cref{LegendreBraidVS}, we have the vector space $\kappa R_{\pm}$ over $\kappa$ with basis 
\[
\bigcup_{x \in R} \{x = x \otimes v, \ \underline x = x \otimes \underline v\},
\]
braided by
\[
x \otimes  y \mapsto  y \otimes x^y, \quad
x \otimes \underline y \mapsto \underline y \otimes x^y, \quad
\underline x \otimes  y \mapsto - y \otimes \underline x^y, \quad
\underline x \otimes \underline y \mapsto  \underline y \otimes \underline x^y,
\]
for every $x,y \in R$.
    
\end{exmp}





\begin{prop} \label{TensorAddablePair}

Let $(U,V)$ be an addable pair of braided objects in an additive symmetric monoidal category $\mathcal C$, and let $W$ be a braided object in $\mathcal C$.
Then $(U \otimes W, V \otimes W)$ functorially becomes an addable pair of braided objects in $\mathcal C$ if we set
\[
T_{U \otimes W, V \otimes W} = \mathrm{id}_V \otimes s_{U,W} \otimes \mathrm{id}_W  \circ T_{U,V} \otimes T_W \circ \ \mathrm{id}_U \otimes s_{W,V} \otimes \mathrm{id}_W
\]
and define $T_{V \otimes W, U \otimes W}$ by exchanging $U,V$ in the formula above. We then have
\[
(U \oplus V) \otimes W \cong (U \otimes W) \oplus (V \otimes W)
\]
as braided objects in $\mathcal C$, where the direct sums are formed with respect to the braided pairs.
    
\end{prop}

\begin{proof}

This is a straightforward check using \cref{DefAddablePair}.
\end{proof}

\begin{exmp} \label{TensoringWithKWedge}

For every braided vector space $W$ over a field $\kappa$, in view of \cref{BVStwistedByZetaWedgePM}, \cref{PlainDirectSumWedgeDec}, and \cref{TensorAddablePair},  we have
\[
W_{\wedge} = W \otimes \kappa_\wedge = W \otimes (\kappa \oplus \kappa_{-1}) = (W \otimes \kappa) \oplus (W \otimes \kappa_{-1}) = W \oplus W_{-1}
\]
as braided vector spaces over $\kappa$, where the direct sum of (a pair of) braided vector spaces in the third term is plain in the sense of \cref{PlainDirectSumEx}, and the direct sum in the (one to) last term corresponds to the addable pair $(W, W_{-1})$ defined as in \cref{TensorAddablePair}.  

In particular, this example illustrates the subtlety that, even when $U \oplus V$ is a plain direct sum, $(U \tensor W) \oplus (V \tensor W)$ need not be.  This is why we need to emphasize that, for braided vector spaces, addability is an extra structure -- if we simply declared direct sum to be `plain direct sum', it would not commute with (or transform in a manageable way under) tensor product.

In the special case $W = \kappa R$ for a rack $R$ we have
\[
T_{W \otimes W_{-1}}(x_1 \otimes y_{-1}) = y_{-1} \otimes x^y_1, \qquad
T_{W_{-1} \otimes W}(x_{-1} \otimes y_1) = y_1 \otimes x^y_{-1}, \qquad x,y \in R.
\]
This can also be seen from \cref{AddableDecompositionRackCocyclePartition} using \cref{RwedgePMcocycles}.

\end{exmp}

\begin{exmp} \label{TensoringWithLegendre}

For every braided vector space $W$ over a field $\kappa$, in view of \cref{BVStwistedByZetaWedgePM}, \cref{WeightedDirectSumExmp}, and \cref{TensorAddablePair}, we have
\[
W_{\pm} = W \otimes \kappa_{\pm} = W \otimes (\kappa \oplus_{1,-1} \kappa) = (W \otimes \kappa) \oplus (W \otimes \kappa) = W \oplus \underline W
\]
as braided vector spaces over $\kappa$, where $\underline W$ is a copy of $W$ and the direct sum appearing in the third of the four terms above corresponds to the addable pair $(W, \underline W)$ defined as in \cref{TensorAddablePair}.

In the special case $W = \kappa R$ for a rack $R$ we have
\[
T_{W \otimes \underline W}(x \otimes \underline y) = \underline y \otimes x^y, \qquad
T_{\underline W \otimes W}(\underline x \otimes y) = - y \otimes \underline x^y, \qquad x,y \in R.
\]
This can also be seen from \cref{AddableDecompositionRackCocyclePartition} using \cref{RwedgePMcocycles}.

\end{exmp}

\begin{exmp} \label{RackificationOneDimensionalTwisted}

Let $n$ be a positive integer, let $\kappa$ be a field of characteristic not dividing $n$, and let $\zeta \in \kappa$ be a primitive $n$th root of unity.
Using the tensor product of the identity map on $\kappa R$ with the injective morphism from \cref{RackificationOneDimensional}, and the monoidality of the functor from \cref{RacksToBraidedVS}, we get an injective morphism of braided vector spaces
\[
\kappa R_\zeta = \kappa R \otimes \kappa_\zeta \hookrightarrow \kappa R \otimes \kappa[\Z/n\Z] = \kappa[R \times \Z/n\Z]
\]
which coincides with the injective morphism obtained by identifying $\kappa R_{\zeta}$ with $\kappa R(c)$ from \cref{TwistingAsCocycle}, and applying \cref{RackifyingCocycle}.

\end{exmp}

\begin{exmp} \label{TwistedSawinRack}

Let $\kappa$ be a field of characteristic different from $2$.
Using the tensor product of the identity map on $\kappa R$ with the injective morphism from \cref{SawinRack}, and the monoidality of the functor from \cref{RacksToBraidedVS}, we get an injective morphism of braided vector spaces
\[
\kappa R_{\wedge} = \kappa R \otimes \kappa_\wedge \hookrightarrow \kappa R \otimes \kappa \mathcal S_{\wedge} = \kappa[R \times \mathcal S_{\wedge}]. 
\]
As a consequence of \cref{PartitionRackDef} and the remarks following it, we have a partition of $R \times \mathcal S_\wedge$ with one part being $R \times \mathcal T_1 \cong R$ and the other $R \times \Z/2\Z$.

    
\end{exmp}

\begin{exmp} \label{RackifyingCpmTwisted}

Let $\kappa$ be a field of characteristic different from $2$.
Using the tensor product of the identity map on $\kappa R$ with the injective morphism from \cref{RackifyingCpm}, and the monoidality of the functor from \cref{RacksToBraidedVS}, we get an injective morphism of braided vector spaces
\[
\kappa R_{\pm} = \kappa R \otimes \kappa_{\pm} \hookrightarrow \kappa R \otimes \kappa \mathcal S_{\pm} = \kappa[R \times \mathcal S_{\pm}]. 
\]
We have a partition of $R \times \mathcal S_\pm$ with one part being $R \times \{J_3\} \cong R$ and the other $R \times \{J_1, J_2\} \cong R \times \mathcal T_2$.
We thus get from \cref{AddableDecompositionRackCocyclePartition} that $\kappa[R \times \mathcal S_{\pm}] = \kappa[R \times \{J_3\}] \oplus \kappa[R \times \{J_1, J_2\}]$.

From (the functoriality in) \cref{TensorAddablePair} and \cref{PmPairsMor} we get a morphism of addable pairs of braided vector spaces 
\[
(\kappa R, \underline{\kappa R}) \to (\kappa [R \times \{J_3\}], \kappa[R \times \{J_1, J_2\}]), \qquad r \mapsto (r,J_3), \quad \underline{r} \mapsto (r,J_1) - (r,J_2),
\]
for $r \in R$, in the notation of \cref{TensoringWithKWedge}.
This morphism induces, via the functoriality in \cref{MorphismAddablePairBVS}, the above injective homomorphism of braided vector spaces $\kappa R_{\pm} \to \kappa[R \times \mathcal S_{\pm}]$.

    
\end{exmp}

\subsection{Convolution}

For objects $V,W$ of a symmetric monoidal additive category $\mathcal C$ we will make a tacit identification  
\begin{equation} \label{TacitBinomial}
(V \oplus W)^{\otimes n} \cong \bigoplus_{i=0}^n \bigoplus_{\sigma \in S_n/(S_i \times S_{n-i})} \bigotimes_{j=1}^n Z(i,j;\sigma) 
\end{equation}
where
\[
Z(i,j;\sigma) = \begin{cases}
V  & j \in \{\sigma(1), \dots, \sigma(i)\} \\
W  & j \in \{\sigma(i+1), \dots, \sigma(n)\}.
\end{cases}
\]

\begin{rem} \label{BraidActionColorAndInflation}

Let $(V,W)$ be an addable pair of braided objects in a symmetric monoidal additive category $\mathcal C$.
For every $0 \leq i \leq n$ the action of $B_n$ on $(V \oplus W)^{\otimes n}$ restricts to an action of $B_{i,n-i}$ on the subobject $V^{\otimes i} \otimes W^{\otimes (n-i)}$ obtained by taking the trivial coset $S_i \times S_{n-i}$ in \cref{TacitBinomial}.
In case of a plain direct sum the latter action is inflated from the external tensor product action of $B_i \times B_{n-i}$ on $V^{\otimes i} \otimes W^{\otimes (n-i)}$.

\begin{rem} \label{BinomialConnection}

Let $S,T \subseteq R$ be a partition of a rack, and let $\kappa$ be a field.
Setting $V = \kappa S$, $W = \kappa T$, we recall from \cref{AddableDecompositionRackCocyclePartition} that $\kappa R \cong \kappa V \oplus \kappa W$ as braided vector spaces.
The isomorphism $\kappa R^n \cong \kappa R^{\otimes n}$ from \cref{LinearizationNaturoid} interchanges the disjoint union decomposition in \cref{SetTheoreticTacitBinomial} for $k=2$, with the direct sum decomposition in \cref{TacitBinomial}, mapping the $i$th summand in \cref{TacitBinomial} to the span over $\kappa$ of $R(i, n-i)$.

\end{rem}


\end{rem}

\begin{prop} \label{InducedStructureTensorPowerDirectSum}

Let $(V,W)$ be an addable pair of braided vector spaces over a field $\kappa$. Then we have a natural isomorphism
\[
(V \oplus W)^{\otimes n} \cong \bigoplus_{i=0}^n \operatorname{Ind}^{B_n}_{B_{i,n-i}} V^{\otimes i} \otimes W^{\otimes (n-i)}
\]    
of representations of $B_n$ over $\kappa$.
\end{prop}

\begin{proof}

By a universal property of the direct sum, in order to construct a map from the right to the left hand side, it suffices to specify an element of 
\[
\operatorname{Hom}_{B_n} (\operatorname{Ind}^{B_n}_{B_{i,n-i}} V^{\otimes i} \otimes W^{\otimes (n-i)} , (V \oplus W)^{\otimes n}) = 
\operatorname{Hom}_{B_{i,n-i}}(V^{\otimes i} \otimes W^{\otimes (n-i)} , (V \oplus W)^{\otimes n})
\]
for every $0 \leq i \leq n$, because induction is left adjoint to restriction.
The requisite element is then provided by the inclusion of the trivial coset ($\sigma = \mathrm{id}$) summand in \cref{TacitBinomial}.

To construct a map in the other direction, it is enough to give an element of 
\[
\operatorname{Hom}_{B_n} ((V \oplus W)^{\otimes n}, \operatorname{Ind}^{B_n}_{B_{i,n-i}} V^{\otimes i} \otimes W^{\otimes (n-i)}) = 
\operatorname{Hom}_{B_{i,n-i}}((V \oplus W)^{\otimes n}, V^{\otimes i} \otimes W^{\otimes (n-i)})
\]
for every $0 \leq i \leq n$, because induction is right adjoint to restriction as $B_{i,n-i}$ is a finite index subgroup of $B_n$. 
The required element is given by the projection onto the trivial coset summand in \cref{TacitBinomial}.

We have thus obtained maps in opposite directions. To conclude, one checks that their compositions are identity morphisms.
\end{proof}

Next we explain the compatibility of the isomorphisms in \cref{InducedStructureTensorPowerDirectSum} as $n$ varies.
For that matter we put
\[
U_n = \bigoplus_{i=0}^n \operatorname{Ind}^{B_n}_{B_{i,n-i}} V^{\otimes i} \otimes W^{\otimes (n-i)}.
\]
We will now (intrinsically) define $\kappa$-linear maps (in fact, isomorphisms)
\[
\xi_{m,n} \colon U_m \otimes U_n \to U_{m+n}, \qquad m,n \geq 0,
\]
playing the role of the vertical maps in \cref{NaturoidCommDiag}.
That is, for nonnegative integers $m,n$, and $g \in B_m, \ h \in B_n$, we want to get a commutative diagram
\begin{equation*} 
\begin{tikzcd}
U_{m} \otimes U_{n} \arrow[rr,"g \otimes h"] \arrow [d, "\xi_{m,n}"] & & U_{m} \otimes U_{n} \arrow[d, "\xi_{m,n}"] \\
U_{m+n} \arrow[rr,"gh"] & &  U_{m+n}
\end{tikzcd}
\end{equation*}
of vector spaces over $\kappa$ where $gh$ stands for juxtaposition of braids.

To do this, it suffices to produce for every $0 \leq i \leq m$ and $0 \leq j \leq n$ a homomorphism of representations of $B_m \times B_n$ over $\kappa$ from 
\begin{equation} \label{TensorOfTwoInductions}
\operatorname{Ind}^{B_m}_{B_{i,m-i}} (V^{\otimes i} \otimes W^{\otimes (m-i)}) \otimes \operatorname{Ind}^{B_n}_{B_{j,n-j}} (V^{\otimes j} \otimes W^{\otimes (n-j)})
\end{equation}
to the restriction to $B_m \times B_n$ of the representation
\begin{equation} \label{InductionOfTensorProducts}
\operatorname{Ind}^{B_{m+n}}_{B_{i+j,m+n-(i+j)}} V^{\otimes i+j} \otimes W^{\otimes (m+n-(i+j))}
\end{equation}
of $B_{m+n}$.

The $\kappa$-linear representation of $B_m \times B_n$ in \cref{TensorOfTwoInductions} is isomorphic to
\[
\operatorname{Ind}^{B_m \times B_n}_{B_{i,m-i} \times B_{j,n-j}} (V^{\otimes (i+j)} \otimes W^{\otimes (m+n-(i+j))})
\]
and by the Mackey formula, the restriction to $B_m \times B_n$ of the representation in \cref{InductionOfTensorProducts} is isomorphic to
\[
\begin{split}
\bigoplus_{g \in (B_m \times B_n)\backslash B_{m+n} /B_{i+j,m+n-(i+j)}} \operatorname{Ind}^{B_{m} \times B_n}_{gB_{i+j,m+n-(i+j)}g^{-1} \cap (B_m \times B_n)} V^{\otimes i+j} \otimes W^{\otimes (m+n-(i+j))}.
\end{split}
\]
Since induction is left adjoint to restriction, we are looking for a homomorphism of representations of $B_{i, m-i} \times B_{j,n-j}$ over $\kappa$ from 
$V^{\otimes (i+j)} \otimes W^{\otimes(m+n-(i+j))}$ to the restriction to $B_{i,m-i} \times B_{j,n-j}$ of the representation
\begin{equation} \label{StrippedInduction}
\operatorname{Ind}^{B_{m} \times B_n}_{gB_{i+j,m+n-(i+j)}g^{-1} \cap (B_m \times B_n)} V^{\otimes i+j} \otimes W^{\otimes (m+n-(i+j))}
\end{equation}
of $B_m \times B_n$, in view of a univeral property of direct sums.

Since the index of $B_{i+j, m+n-(i+j)}$ in $B_{m+n}$ is finite, so is 
\[
[B_m \times B_n : gB_{i+j,m+n-(i+j)}g^{-1} \cap (B_m \times B_n)]
\]
because finiteness of the index is preserved by conjugation and by intersection with a subgroup.
Arguing similarly to the way we already did - applying the Mackey formula to the restriction of the representation in \cref{StrippedInduction} to $B_{i,m-i} \times B_{j,n-j}$, applying a universal property of direct sums, checking finiteness of index, and then using the ensuing right adjointness of induction to restriction, we eventually obtain the desired map $\xi_{m,n}$ from the identity morphism of the representation $V^{\otimes (i+j)} \otimes W^{\otimes(m+n-(i+j))}$ of a suitable subgroup of a braid group.


The aforementioned compatibility as $n$ varies is the resulting commutativity of the diagram
\[\begin{tikzcd}
	{(V \oplus W)^{\otimes m} \otimes (V \oplus W)^{\otimes n}} & {U_m \otimes U_n} \\
	{(V \oplus W)^{\otimes m + n}} & {U_{m+n}}
	\arrow[from=1-1, to=1-2]
	\arrow[from=1-1, to=2-1]
	\arrow[from=1-2, to=2-2, "\xi_{m,n}"]
	\arrow[from=2-1, to=2-2]
\end{tikzcd}\]
of vector spaces over $\kappa$ where the horizontal maps come from the isomorphisms given by \cref{InducedStructureTensorPowerDirectSum}, and the left vertical map is the isomorphism coming from the symmetric monoidal structure on vector spaces.




\begin{cor} \label{InducedStructureTensorPowerDirectSumPerm}

Let $V,W$ be permutational braided vector spaces over a field $\kappa$ in the sense of \cref{PermBraidObj}. Then
\[
(V \oplus W)^{\otimes n} \cong \bigoplus_{i=0}^n \operatorname{Ind}^{S_n}_{S_{i} \times S_{n-i}} V^{\otimes i} \boxtimes W^{\otimes (n-i)}
\]    
as representations of $S_n$ over $\kappa$, where the direct sum of permutational braided vector spaces on the left hand side is plain in the sense of \cref{PlainDirectSumEx}.

\end{cor}
    
\begin{proof}

Invoke \cref{InducedStructureTensorPowerDirectSum}, \cref{BraidActionColorAndInflation}, and the commutativity of induction and inflation.    
\end{proof}

\begin{cor} \label{BraidedSumCoinvs}

For an addable pair $(V,W)$ of braided vector spaces over a field $\kappa$ we have
\[
(V \oplus W)^{\otimes n}_{B_n} \cong \bigoplus_{i=0}^n (V^{\otimes i} \otimes W^{\otimes (n-i)})_{B_{i,n-i}}
\]
and in case the direct sum is plain, this is isomorphic to $\bigoplus_{i=0}^n V^{\otimes i}_{B_i} \otimes W^{\otimes (n-i)}_{B_{n-i}}$. 
As a result, the functor from \cref{DefAlgCoinvs} associating to a braided vector space over $\kappa$ its algebra of coinvariants is symmetric monoidal with respect to the plain direct sum monoidal structure from \cref{PlainDirectSumEx}.

\end{cor}

\begin{proof}

It follows from \cref{InducedStructureTensorPowerDirectSum}, additivity of coinvariants, and the Shapiro Lemma that
\[
\begin{split}
H_0(B_n, (V \oplus W)^{\otimes n}) &= H_0 \left(B_n, \bigoplus_{i=0}^n \operatorname{Ind}^{B_n}_{B_{i,n-i}} V^{\otimes i} \otimes W^{\otimes (n-i)} \right) \\ &= \bigoplus_{i=0}^n H_0(B_{i,n-i}, V^{\otimes i} \otimes W^{\otimes (n-i)}).
\end{split}
\]
In case the direct sum is plain, in view of \cref{BraidActionColorAndInflation} we have
\[
\begin{split}
\bigoplus_{i=0}^n H_0(B_{i,n-i}, V^{\otimes i} \otimes W^{\otimes (n-i)}) &=
\bigoplus_{i=0}^n H_0(B_{i} \times B_{n-i}, V^{\otimes i} \boxtimes W^{\otimes (n-i)})
\\
&\cong 
\bigoplus_{i=0}^n H_0(B_{i}, V^{\otimes i}) \otimes H_0(B_{n-i}, W^{\otimes (n-i)}).
\end{split}
\]


For the (symmetric) monoidality statement we mainly need to check that 
\[
C(V \oplus W) \cong C(V) \otimes C(W)
\]
as graded $\kappa$-algebras, where the direct sum of braided vector spaces on the left hand side is plain.
Indeed it follows from the above, and the discussion after the proof of \cref{InducedStructureTensorPowerDirectSum}, that we have an isomorphism of graded $\kappa$-algebras
\[
\begin{split}
C(V \oplus W) = \bigoplus_{n = 0}^\infty (V \oplus W)^{\otimes n}_{B_n} &\cong 
\bigoplus_{n = 0}^\infty \bigoplus_{i=0}^n V^{\otimes i}_{B_i} \otimes W^{\otimes (n-i)}_{B_{n-i}} \\ &\cong \left( \bigoplus_{i=0}^{\infty} V^{\otimes i}_{B_i} \right) \otimes \left( \bigoplus_{j = 0}^{\infty} W^{\otimes j}_{B_j} \right)= C(V) \otimes C(W)
\end{split}
\]
as required.
\end{proof}

For an addable pair $(V,W)$ of braided vector spaces over a field $\kappa$, \cref{DirectSumBVSfunctoriality} and the functoriality of the algebra of coinvariants provide us with an inclusion of graded $\kappa$-algebras 
\begin{equation} \label{InclusionGradedAlgebrasFromDirSumBVS}
C(V) \to C(V \oplus W),
\end{equation}
and a surjection of graded $\kappa$-algebras
\begin{equation} \label{SurjectionGradedAlgebrasFromDirSumBVS}
C(V \oplus W) \to C(V).
\end{equation}
This inclusion (respectively, surjection) arises also from \cref{BraidedSumCoinvs} via the inclusion of (respectively, projection to) the $i=n$ term.

We denote by $\operatorname{triv}_n$ the trivial one-dimensional representation of $S_n$ over a field $\kappa$, by $\operatorname{sign}_n$ the sign representation of $S_n$ over $\kappa$, by $\operatorname{Perm}_n = \operatorname{Ind}_{S_{n-1}}^{S_n} \kappa$ the permutation representation of $S_n$, and by $\operatorname{std}_n$ the standard representation of $S_n$ on $\{v \in \operatorname{Perm}_n : v_1 + \dots + v_n = 0\}$. If the characteristic of $\kappa$ does not divide $n$ we have
$
\operatorname{Perm}_n = \operatorname{triv}_n \oplus \operatorname{std}_n.
$

\begin{cor} \label{co:wedgedecomp}

We have an isomorphism of $S_n$-representations
\[
\kappa_\wedge^{\otimes n} \cong \bigoplus_{i=0}^n \operatorname{Ind}^{S_n}_{S_{i} \times S_{n-i}} \mathrm{sign}_{n-i} \cong \bigoplus_{i=0}^n \wedge^i \operatorname{Perm}_n  
\]
over $\kappa$ and in case the characteristic of $\kappa$ does not divide $n$ these representations of $S_n$ are also isomorphic to
\[
\bigoplus_{i=0}^n \wedge^i \operatorname{std}_n \oplus \wedge^{i-1} \operatorname{std}_n  \cong \bigoplus_{i=0}^{n-1} (\wedge^i \operatorname{std}_n)^{\oplus 2}.
\]

\end{cor}

\begin{proof}

\cref{PlainDirectSumWedgeDec} and \cref{InducedStructureTensorPowerDirectSumPerm} tell us that
\[
\kappa_\wedge^{\otimes n} = (\kappa \oplus \kappa_{-1})^{\otimes n} = 
\bigoplus_{i=0}^n \operatorname{Ind}^{S_n}_{S_{i} \times S_{n-i}} \kappa^{\otimes i} \otimes \kappa_{-1}^{\otimes (n-i)} = \bigoplus_{i=0}^n \operatorname{Ind}^{S_n}_{S_{i} \times S_{n-i}} \mathrm{sign}_{n-i}.
\]

We claim that for every $0 \leq i \leq n$ we have
\begin{equation} \label{ClaimIsomRepsSn}
\operatorname{Ind}^{S_n}_{S_{i} \times S_{n-i}}  \operatorname{sign}_{n-i} \cong \wedge^{n-i} \operatorname{Perm}_n
\end{equation}
as representations of $S_n$ over $\kappa$. 
Since induction is left adjoint to restriction we have
\[
\begin{split}
\Hom_{S_n}(\operatorname{Ind}^{S_n}_{S_{i} \times S_{n-i}}  \operatorname{sign}_{n-i}, \wedge^{n-i} \operatorname{Perm}_n) &\cong
\Hom_{S_i \times S_{n-i}}(\operatorname{sign}_{n-i}, \wedge^{n-i} \operatorname{Perm}_n) \\
&\cong \Hom_{S_i \times S_{n-i}}(\wedge^{n-i} \operatorname{Perm}_{n-i}, \wedge^{n-i} \operatorname{Perm}_n)
\end{split}
\]
as vector spaces over $\kappa$.
The inclusion of $\{i+1, \dots, n\}$ into $\{1, \dots, n\}$ gives us a homomorphism $\operatorname{Perm}_{n-i} \to \operatorname{Perm}_n$ 
of representations of $S_i \times S_{n-i}$, so taking its $(n-i)$th wedge power gives us a homomorphism of representations of $S_n$ over $\kappa$ from the left to the right hand side of \cref{ClaimIsomRepsSn}.

Since $\operatorname{Perm}_n$ is a self-dual representations of ($S_n$ and thus also of) $S_i \times S_{n-i}$, and duality commutes with forming wedge powers, we get a morphism $\wedge^{n-i}\operatorname{Perm}_n \to \wedge^{n-i}\operatorname{Perm}_{n-i}$ of $S_i \times S_{n-i}$-representations over $\kappa$. Since the group $S_n$ is finite, we have
\[
\begin{split}
\Hom_{S_n}(\wedge^{n-i}\operatorname{Perm}_n, \operatorname{Ind}^{S_n}_{S_{i} \times S_{n-i}}  \operatorname{sign}_{n-i}) &\cong
\Hom_{S_i \times S_{n-i}}(\wedge^{n-i}\operatorname{Perm}_n, \operatorname{sign}_{n-i}) \\
& \cong \Hom_{S_i \times S_{n-i}}(\wedge^{n-i}\operatorname{Perm}_n, \wedge^{n-i}\operatorname{Perm}_{n-i})
\end{split}
\]
as vector spaces over $\kappa$, which gives us a homomorphism of representations of $S_n$ over $\kappa$ from the right to the left hand side of \cref{ClaimIsomRepsSn}.
We have thus constructed morphisms between the $S_n$-representations in \cref{ClaimIsomRepsSn}, and one can check that these morphisms are mutually inverse.

If the characteristic of $\kappa$ does not divide $n$ we have
\begin{equation} \label{PermRepInTermsOfStdWedge}
\wedge^i \operatorname{Perm}_n \cong \wedge^i(\operatorname{std}_n \oplus \operatorname{triv}_n) = \wedge^i \operatorname{std}_n \oplus \wedge^{i-1}\operatorname{std}_n
\end{equation}
as representations of $S_n$ over $\kappa$, so we obtain the required isomorphisms by noticing that $i = 0$ and $i = n$ contribute only one (nonzero) term each.
\end{proof}



For every $0 \leq i \leq n$, as in \cref{ProductFiniteEtaleMapConf} and \cref{ConfOpenImmersoionProduct}, we have the morphisms 
\[
\tau^{(i)} \colon \operatorname{Conf}^{i,n-i} \to \operatorname{Conf}^n, \qquad c_i \colon \operatorname{Conf}^{i,n-i}  \to \operatorname{Conf}^i \times \operatorname{Conf}^{n-i}.
\]

\begin{cor} \label{ConvolutionCorrDirectSum}

Let $\kappa$ be a finite field, and let $V,W$ be finite-dimensional braided vector spaces over $\kappa$.
Let $S$ be an open subscheme of the spectrum of the ring of integers of a number field $K$, and for every nonnegative integer $n$ let $\mathcal V_n$ (respectively, $\mathcal W_n$) be a constructible locally constant \'etale sheaf of vector spaces over $\kappa$ on $\operatorname{Conf}^n \times \ S$ such that the analytification of the base change of $\mathcal V_n$ (respectively, $\mathcal W_n$) to $\operatorname{Conf}^n \times \ \C$ corresponds to the representation $V^{\otimes n}$ (respectively, $W^{ \otimes n}$) of $B_n$ over $\kappa$.
Then the \'etale sheaf of vector spaces 
\[
\mathcal U_n =  \bigoplus_{i=0}^n \tau^{(i)}_* c_i^{-1} \left(\mathcal V_i \boxtimes \mathcal W_{n-i}\right)
\]
over $\kappa$ is locally constant constructible, the analytification of its base change to $\operatorname{Conf}^n \times \ \C$ corresponds to the representation $(V \oplus W)^{\otimes n}$ of $B_n$ over $\kappa$ where the direct sum of braided vector spaces is plain, and the trace function of $\mathcal U$ is the Dirichlet convolution of the trace functions of $\mathcal V$ with the trace function of $\mathcal W$, namely for every $\mathfrak p \in S$, a finite extension $\F_q$ of $\mathcal O_K/\mathfrak p$, and every monic squarefree $f \in \F_q[t]$ we have
\[
\operatorname{tr}(\operatorname{Frob}_f, \mathcal U_f) = \sum_{gh = f} \operatorname{tr}( \operatorname{Frob}_g, \mathcal V_{g}) \cdot \operatorname{tr}(\operatorname{Frob}_h, \mathcal W_h)
\]
where $g,h \in \F_q[t]$ range over monic (coprime, squarefree) polynomials, and an expression such as $\mathcal U_f$ is a shorthand for the \'etale stalk of $\mathcal U_{\deg f}$ at a geometric point over $f \in \operatorname{Conf}^{\deg f}(\F_q)$.

\end{cor}

\begin{proof}

The property of being locally constant constructible is preserved under tensor products, pullbacks, direct sums, and pushforwards by proper smooth maps.
The morphisms $\tau^{(i)}$ are \'etale hence smooth, and are finite hence proper, so the local constancy and constructibility of the sheaves $\mathcal U_n$ follows.

Base change to $\operatorname{Conf}^n \times \ \C$ is an additive functor, namely it commutes with direct sums.
By the proper base change theorem, base change to $\C$ commutes with the $\tau^{(i)}$.
Pullback by $c_i$ also satisfies a suitable commutativity with base change to $\C$ because the latter is itself given by pullback, and because pullback reverses the order of composition.
Similarly, pullback to $\C$ commutes with the external tensor product.

Analytification is an additive functor that commutes with pushforwards, pullbacks, and external tensor products.

The analytification of the base change of $\tau^{(i)}$ (respectively, $c_i$) to $\C$ induces the inclusion of $B_{i, n-i}$ into $B_n$, (respectively, the projection of $B_{i,n-i}$ onto $B_i \times B_{n-i}$).

The equivalence between locally constant sheaves and representations of the fundamental group is additive, transforms $\tau_*^{(i)}$ to the induction of representations from $B_{i,n-i}$ to $B_n$, transforms $c_i^{-1}$ to the inflation of representations of $B_{i} \times B_{n-i}$ to $B_{i,n-i}$.
The asserted correspondence to $(V \oplus W)^{\otimes n}$ is now a consequence of \cref{BraidActionColorAndInflation} and \cref{InducedStructureTensorPowerDirectSum}.

According to the function-sheaf dictionary direct sums correspond to summations of functions, pushforward by the finite map $\tau^{(i)}$ corresponds to summation over all the factorizations over $\F_q$ as a product of a degree $i$ polynomial times a degree $n-i$ polynomial, pullback corresponds to restriction of functions, and tensor products correspond to products of functions.
The statement on convolution thus follows.
\end{proof}


\begin{cor} \label{WedgeArithmetized}

For every nonnegative integer $n$ there exist a unique lisse sheaf $\mathcal U_n$ on $\operatorname{Conf}^n$ of vector spaces over $\kappa$ satisfying the following properties.

\begin{itemize}

    \item The arithmetic monodromy group of $\mathcal U_n$ coincides with its geometric monodromy group.
    
    \item The analytification of the base change of $\mathcal U_n$ to $\operatorname{Conf}^n \times \ \C$ corresponds to the representation $\kappa_{\wedge}^{\otimes n}$ of $B_n$.

    \item For every monic squarefree polynomial $f \in \F_q[t]$ we have
\begin{equation} \label{WeirdArithFunction}
\operatorname{tr}(\operatorname{Frob}_f, \mathcal U_f) = 
\begin{cases}
0 & f \text{ has an irreducible factor over } \F_q \text{ of even degree} \\
d_2(f) &\text{ otherwise}.
\end{cases}
\end{equation}

\end{itemize}
    
\end{cor}

\begin{proof}

Uniqueness is a general fact that follows from a version of Chebotarev's density theorem.

We construct $\mathcal U_n$ by invoking \cref{ConvolutionCorrDirectSum} with $K = \Q$, $S = \operatorname{Spec} \Z$, $\mathcal V_n$ the constant sheaf of rank $1$, and $\mathcal W_n$ the sheaf corresponding to the sign representation of $S_n$ as in \cref{MobiusFunctionRepThry}.
In this case $V$ is the trivial one-dimensional braided vector space, and $W = \kappa_{-1}$.
The trace function of $\mathcal V_n$ is $1$ by \cref{TrivialRepSnTraceFunction}, and the trace function of $\mathcal W_n$ is $(-1)^n$ times the M\"obius function by \cref{MobiusFunctionRepThry}, in particular both these functions are multiplicative.
Because $\kappa_\wedge = \kappa \oplus \kappa_{-1}$ by \cref{PlainDirectSumWedgeDec}, it remains to check that the arithmetic function on the right hand side of \cref{WeirdArithFunction} agrees with the convolution of the constant function $1$ and the function $f \mapsto (-1)^{\deg f} \cdot \mu(f)$.
Since the convolution of multiplicative functions is multiplicative, the function on the right hand side of \cref{WeirdArithFunction} is multiplicative as the divisor function is, and we only deal with squarefree (monic) polynomials in $\F_q[t]$, so it suffices to check the aforementioned agreement on irreducible polynomials, for which both sides of \cref{WeirdArithFunction} equal to $2$ if the degree of the irreducible polynomial is odd and to $0$ if that degree is even.   
\end{proof}

\begin{cor} \label{ArithmetizingPMbvs}

Let $\chi \colon \F_q^\times \to \{\pm 1\}$ be the nontrivial character, and put $\kappa = \overline{\Q_\ell}$.
Then the analytification of the base change to $\C$ of $\bigoplus_{i=0}^n \tau_*^{(i)} \operatorname{Res}^{-1}\mathcal L_\chi$ from \cref{TraceFunctionCharacterSheafResultantSumTwo} corresponds to the representation $\kappa_{\pm}^{\otimes n}$ of $B_n$.

\end{cor}

\begin{proof}

The analytification of the base change to $\C$ of $\tau_*^{(i)} \operatorname{Res}^{-1}\mathcal L_\chi$ corresponds to the induction from $B_{i,n-i}$ to $B_n$ of the one-dimensional representation with $\sigma_i^2$ acting by negation, and $\sigma_1, \dots, \sigma_{i-1}, \sigma_{i+1}, \dots, \sigma_{n-1}$ act as the identity.

It follows from \cref{WeightedDirectSumExmp} and \cref{InducedStructureTensorPowerDirectSum} that 
\[
\kappa_{\pm}^{\otimes n} \cong \bigoplus_{i=0}^n \operatorname{Ind}^{B_n}_{B_{i,n-i}} \operatorname{Span}_\kappa v^{\otimes i} \otimes \underline{v}^{\otimes (n-i)}
\]
as representations of $B_n$ in the notation of \cref{LegendreBraidVS}.
To conclude the argument, one checks that the action of the generators $\{\sigma_1, \dots, \sigma_{i-1}, \sigma_i^2, \sigma_{i+1}, \dots, \sigma_{n-1} \}$ of $B_{i,n-i}$ on $\operatorname{Span}_\kappa v^{\otimes i} \otimes \underline{v}^{\otimes (n-i)}$ is the one from the previous paragraph.
\end{proof}

\begin{exmp} \label{IndResChiIntoRack}

It follows from \cref{PmPairsMor}, \cref{BinomialConnection}, and \cref{InducedStructureTensorPowerDirectSum} that for a field $\kappa$ we have an injective homomorphism of representations
\begin{equation*}
\bigoplus_{i=0}^n \operatorname{Ind}^{B_n}_{B_{i,n-i}} \operatorname{Span}_\kappa v^{\otimes i} \otimes \underline{v}^{\otimes (n-i)} \to \bigoplus_{i=0}^n \kappa \mathcal S_{\pm}(n-i,i)
\end{equation*}
of $B_n$ over $\kappa$, mapping the $i$th summand on the left into the $i$th summand on the right.

\end{exmp}

\subsection{Algebra of Coinvariants} \label{CoinvsAlg}

\begin{defn} 

Let $R$ be a rack. 
For $x,y \in R$ consider the recursively defined sequence
\[
x_0 = x, \qquad x_{i+1} = x_i^y, \qquad i \geq 0.
\]
Using the (right) action of $\Gamma_R$ on $R$ we can also define this sequence by
\[
x_i = x_0^{y^i}, \qquad i \geq 0.
\]

\end{defn}

\begin{prop} \label{CyclotomicCentralElementsPowers}

Let $R$ be a finite rack, let $\kappa$ be a field, let $c \colon R \times R \to \kappa^\times$ be a cyclotomic $2$-cocycle, and let $x,y \in R$.
Then there exists a (minimal) positive integer $\mathfrak q(x,y)$ for which $x_{\mathfrak q(x,y)} = x$, and a minimal positive integer $\mathfrak p(x,y)$ for which
\[
x_{\mathfrak p(x,y)} = x, \qquad \prod_{i=0}^{\mathfrak p(x,y) - 1} c(x_i,y) = 1,
\]
so in the graded $\kappa$-algebra of coinvariants of the braided vector space $\kappa R(c)$ from \cref{RacksToBraidedVS}
the degree-one element $x$ commutes with the homogeneous element $y^{\mathfrak p(x,y)}$ 
and
\[
y^{P_y} \in Z(C(\kappa R(c))), \qquad P_y = \operatorname{lcm}_{x \in R} \mathfrak p(x,y),
\]
namely a suitable power of $y$ lies in the center of the algebra $C(\kappa R(c))$ of coinvariants of the braided vector space $\kappa R(c)$.

\end{prop}

\begin{proof}

By \cref{RackDef} the function $t \mapsto t^y$ is a bijection on $R$ so our assumption that $R$ is finite implies that there exists a (minimal) positive integer $m$ for which $x_{\mathfrak q(x,y)} = x$. For every positive integer $l$ we thus have
\[
\prod_{i=0}^{l \cdot \mathfrak q(x,y) - 1} c(x_i,y) = \prod_{j=0}^{\ell-1} \ \prod_{i=j \cdot \mathfrak q(x,y)}^{(j+1) \cdot \mathfrak q(x,y) -1} c(x_i,y) = \prod_{j=0}^{\ell-1} \prod_{i=0}^{\mathfrak q(x,y)-1} c(x_i,y) = \left(\prod_{i=0}^{\mathfrak q(x,y)-1} c(x_i,y)\right)^l
\]
which is equal to $1$ for some (minimal) positive integer $l = l_0$ by our assumption that $c$ is cyclotomic.
We put $\mathfrak{p}(x,y) = l \cdot \mathfrak q(x,y)$.

Next we claim that for every positive integer $j$ we have
\begin{equation} \label{HowToCommuteWithPower}
x y^j = \prod_{i=0}^{j-1} c(x_i,y) \cdot y^jx_j
\end{equation}
in $C(\kappa R(c))$, or equivalently in the notation of \cref{BnActionBVsCocycleNabla}, that
\[
\nabla(x_1, y, \dots, y; \sigma_1 \cdots \sigma_{j}) = \prod_{i=0}^{j-1} c(x_i,y)
\]
where on the left hand side above $y$ appears (implicitly) $j$ times.
We prove \cref{HowToCommuteWithPower} the same way we proved \cref{BnActionBVsCocycleNabla} - by induction on $j$ with the base case $j=1$ being identical to \cref{CocycleToBVs}.
Assuming that $j > 1$ we get from the induction hypothesis that
\[
xy^j = xy^{j-1}y = \prod_{i=0}^{j-2} c(x_i,y) \cdot y^{j-1}x_{j-1}y = 
\prod_{i=0}^{j-2} c(x_i,y) \cdot c(x_{j-1},y) \cdot  y^{j-1} y x_{j-1}^y = \prod_{i=0}^{j-1} c(x_i,y) \cdot y^j x_j
\]
completing the induction and thus the proof of the claim.

Specializing the claim above to $j = \mathfrak{p}(x,y)$ we see that $x$ and $y^{\mathfrak{p}(x,y)}$ indeed commute.
Since $P_y$ is by our definition a multiple of $\mathfrak{p}(x,y)$, we get that $y^{P_y}$ (being a power of $y^{\mathfrak p(x,y)}$) also commutes with $x$.
Because $C(\kappa R(c))$ is generated in degree $1$, it follows that $y^{P_y}$ lies in its center as it commutes with all homogeneous elements of degree $1$.
\end{proof}

\begin{rem} \label{DependenceOnConnectedComponent}

We have $\mathfrak p(x^z,y^z) = \mathfrak p(x,y)$ for $x,y,z \in R$, so in case $c$ is a cocycle as in \cref{CocycleWedgePMexmp} (arising from the partition of $R$ into its connected components) the integer $P_y$ depends on $y$ only via the connected component of $R$ in which $y$ lies.

\end{rem}

\begin{lem} \label{GraphLemma}

Let $V$ be a vector space over a field $\kappa$ with basis $B$, let $\Gamma$ be a directed graph with $B$ as its set of vertices, and a set of edges $E$ labeled by elements of $\kappa^\times$. 
Denote by $s,t \colon E \to B$ the source and target maps,
and by $\lambda \colon E \to \kappa^\times$ the labeling.

Suppose that for every $e \in E$ there exists an $\bar e \in E$ with 
\[
s(\bar e) = t(e), \qquad  t(\bar e) = s(e), \qquad \lambda(\bar e) = \lambda(e)^{-1}.
\]

Denote by $W$ the subspace of $V$ spanned by $\{s(e) - \lambda(e) \cdot t(e) : e \in E\}$ over $\kappa$.
Let $\mathcal B$ be a subset of $B$ containing exactly one vertex from each connected component of $\Gamma$. 
Then the restriction of the quotient map $V \to V/W$ to the subset
\[
\mathcal B^1 = \{v \in \mathcal B : \textup{ the product in } \kappa^\times \text{ of the labels of every cycle in } \Gamma \textup{ visiting } v  \text{ is } 1 \}
\]
of $V$ is injective, and its image is a basis of $V/W$ over $\kappa$.
In particular, a vector $v \in B$ lies in $W$ if and only if there exists a cycle in $\Gamma$ visiting $v$ the product of whose labels is different from $1$.
    
\end{lem}

\begin{rem}

The graph $\Gamma$ is allowed to have an $e \in E$ with $s(e) = t(e)$, and the set of edges of $\Gamma$ with a given source and target may contain more than one element.

\end{rem}

\begin{proof}

We start by reducing to the case $\Gamma$ is connected.
Let $\Gamma_i$ be the connected components of $\Gamma$ for $i \in I$ with sets of vertices $B_i$, sets of edges $E_i$, and maps $s|_{E_i}, t|_{E_i} \colon E_i \to B_i$. Denote by $V_i$ the span of $B_i$ over $\kappa$, and by $W_i$ the subspace of $V_i$ spanned by $\{s(e) - \lambda(e) \cdot t(e) : e \in E_i\}$ over $\kappa$.
Put $\mathcal B_i = \mathcal B \cap B_i$ and $\mathcal B_i^1 = \mathcal B^1 \cap \mathcal B_i$.
We then have
\[
B = \bigcup_{i \in I} B_i, \qquad E = \bigcup_{i \in I} E_i, \qquad V = \bigoplus_{i \in I} V_i, \qquad W = \bigoplus_{i \in I} W_i, \qquad V/W = \bigoplus_{i \in I} V_i/W_i,
\]
the quotient map $V \to V/W$ is the direct sum of the quotient maps $V_i \to V_i/W_i$, 
the set $\mathcal B_i$ is a singleton, and 
\[
\mathcal B^1_i = \{v \in \mathcal B_i : \textup{ the product in } \kappa^\times \text{ of the labels of every cycle in } \Gamma_i \textup{ visiting } v  \text{ is } 1 \}.
\]
These observations are sufficient to complete the reduction to the case $\Gamma$ is connected.

Suppose now that $\Gamma$ is connected, so that $\mathcal B = \{b\}$ for some $b \in B$.
The required injectivity follows from the fact that any function on a set of cardinality at most $1$ is injective.
We distinguish between two cases.

The first case is that there exists a cycle in $\Gamma$ visiting $b$ whose labels multiply up to $\alpha \in \kappa$ with $\alpha \neq 1$, namely $\mathcal B^1$ is empty.
It then follows from the definition of $W$, and an induction on the length of the cycle, that $(\alpha - 1)b \in W$, and thus $b \in W$.
From the connectedness of $\Gamma$ we then deduce that the basis $B$ of $V$ is contained in $W$, hence $W = V$ or equivalently $V/W = \{0\}$.
Therefor the image of $\mathcal B^1$ in $V/W$ (an empty set) is indeed a basis for $V/W$.
    
The second case is that the product of the labels of every cycle in $\Gamma$ visiting $b$ is $1$, namely $\mathcal B^1 = \{b\}$.
It then follows from the connectedness of $\Gamma$ that the image of $\mathcal B^1$ in $V/W$ spans $V/W$ over $\kappa$, so it remains to show that $V/W$ is nonzero.
We need to show that there exists a nonzero $\kappa$-linear functional on $V/W$, or equivalently a nonzero $\kappa$-linear functional $\varphi \colon V \to \kappa$ that vanishes on $W$.
We define $\varphi$ by specifying its value on every $v \in B$ to be the product in $\kappa^\times$ of the labels of a path in $\Gamma$ from $b$ to $v$.  
Our assumption that the product of the labels of every cycle in $\Gamma$ visiting $b$ is $1$ guarantees that $\varphi$ is well-defined. 
As $\varphi(b) = 1$, this is indeed a nonzero functional.
Since $W$ is contained in the kernel of $\varphi$, our treatment of the second case, and thus the whole proof, is complete.
\end{proof}

\begin{cor} \label{BasisVanishingPureTensorsRackCoinvsCocycle}

Let $R$ be a rack, let $\kappa$ be a field, and let $c \colon R \times R \to \kappa^\times$ be a $2$-cocycle.
Let $S \subseteq R^n$ be a set of representatives for the orbits of the action of $B_n$ on $R^n$.
Then as $(x_1, \dots, x_n)$ ranges over the elements of $S$ whose stabilizer in $B_n$ only contains braids $g$ for which $\nabla(x_1, \dots, x_n;g) = 1$, the resulting products $x_1 \cdots x_n$ in $C(\kappa R(c))$ are pairwise distinct and form a basis for $\kappa R(c)^{\otimes n}_{B_n}$.
In particular, the stabilizer in $B_n$ of a tuple $(x_1, \dots, x_n) \in R^n$ contains an element $g$ for which $\nabla(x_1, \dots, x_n;g) \neq 1$ if and only if $x_1 \cdots x_n = 0$ in $C(\kappa R(c))$.


\end{cor}

\begin{proof}

Apply \cref{GraphLemma} to $V = \kappa R(c)^{\otimes n} = (\kappa R)^{\otimes n} = \kappa (R^{n})$, $B = R^n$, the graph $\Gamma$ being the Cayley graph of the action of $B_n$ on $R^n$ with respect to the symmetric generating set from \cref{SymmGenSetForBn}, and $\lambda(e) = c(s(e)_i, s(e)_{i+1})$ (respectively, $\lambda(e) = c(t(e)_i, t(e)_{i+1})$) if $e$ corresponds to $\sigma_i$ (respectively, $\sigma_{i}^{-1}$) for some $1 \leq i \leq n-1$.
Note that $V/W = H_0(B_n, \kappa R(c)^{\otimes n})$, let $\mathcal B = S$ so that $\mathcal B^1$ is the set over which $(x_1, \dots, x_n)$ ranges, and observe that the product of the labels of a cycle starting (and ending) at $(x_1, \dots, x_n) \in R^n$ is $\nabla(x_1, \dots, x_n;g)$ where $g$ is the element of the stabilizer of $(x_1, \dots, x_n)$ in $B_n$ corresponding to our cycle.
\end{proof}

\begin{prop} \label{SuffiecientConditionVanishingCoinvariantsPure}

Let $\kappa$ be a field, let $V$ be a braided vector space over $\kappa$, let $R$ be rack, and let $c \colon R \times R \to \kappa^\times$ be a $2$-cocycle.
Suppose that $(V,\kappa R(c))$ carries an addable pair structure, let $0 \leq i \leq n$, let $w \in V^{\otimes i}$, and let $(x_1, \dots, x_{n-i}) \in R^{n-i}$.
Suppose that there exists $g \in B_{n-i}$ stabilizing $(x_1, \dots, x_{n-i})$ for which $\nabla(x_1, \dots, x_{n-i};g) \neq 1$.
Then the vector $w \otimes x_1 \otimes \dots \otimes x_{n-i} \in V^{\otimes i} \otimes \kappa R(c)^{\otimes (n-i)}$ maps to $0$ in $H_0(B_{i,n-i}, V^{\otimes i} \otimes \kappa R(c)^{\otimes (n-i)})$, the action of $B_{i,n-i}$ being the one from \cref{BraidActionColorAndInflation}.

\end{prop}

\begin{rem}

In view of \cref{BraidedSumCoinvs} this is equivalent to
\[
wx_1 \cdots x_{n-i} = 0
\]
in the algebra of coinvariants of the braided vector space $V \oplus \kappa R(c)$.

\end{rem}

\begin{proof}

Since the inclusion of $B_{n-i}$ into $B_n$ factors via the inclusion of $B_{i,n-i}$ into $B_n$, we can view $g$ as an element of $B_{i,n-i}$ for which
\[
(w \otimes x_1 \otimes \dots \otimes x_{n-i})^g = \nabla(x_1, \dots, x_{n-i};g) \cdot w \otimes x_1 \otimes \dots \otimes x_{n-i}
\]
because of the stabilization assumption, so in the $B_{i,n-i}$-coinvarinats we get that
\[
w \otimes x_1 \otimes \dots \otimes x_{n-i} = \nabla(x_1, \dots, x_{n-i};g) \cdot w \otimes x_1 \otimes \dots \otimes x_{n-i}
\]
hence this vector is zero because $\nabla(x_1, \dots, x_{n-i};g) \neq 1$ by assumption.    
\end{proof}

\begin{cor} \label{QuandleConditionVanishingCoinvariants}

Let $R$ be a rack, let $\kappa$ be a field, let $c \colon R \times R \to \kappa^\times$ be a $2$-cocycle, and let $r \in R$ for which $r^r = r$ and $c(r,r) \neq 1$.
Let $0 \leq i \leq n$, let $(x_1, \dots, x_{n-i}) \in R^{n-i}$ such that $x_l = x_j = r$ for some $1 \leq l < j \leq n-i$.
Let $V$ be a braided vector space over $\kappa$ for which $(V,\kappa R(c))$ is an addable pair.
Then for every $w \in V^{\otimes i}$ the element $w \otimes x_1 \otimes \dots \otimes x_{n-i} \in V^{\otimes i} \otimes \kappa R(c)^{\otimes (n-i)}$ maps to $0$ in $H_0(B_{i,n-i}, V^{\otimes i} \otimes \kappa R(c)^{\otimes n-i})$.

\end{cor}

\begin{proof}

Our assumption that $r^r=r$ implies that the braid 
\[
g = \sigma_{j-1} \sigma_{j-2} \cdots \sigma_{l+1} \sigma_{l} \sigma_{l+1}^{-1} \cdots \sigma_{j-2}^{-1} \sigma_{j-1}^{-1} \in B_{n-i}
\]
fixes $(x_1, \dots, x_{n-i}) \in R^{n-i}$, and 
$
\nabla(x_1, \dots, x_{n-i};g) = c(r,r) \neq 1
$
by our assumptions on $r$, so the statement follows from \cref{SuffiecientConditionVanishingCoinvariantsPure}.
\end{proof}

\begin{cor} \label{PigeonholePrincVanishingCoinvs}

Let $R$ be a rack, let $\kappa$ be a field, let $c \colon R \times R \to \kappa^\times$ be a $2$-cocycle, and set 
\[
Q = \{r \in R : r^r = r, \text{ and } c(r,r) \neq 1\}.
\]
Let $0 \leq i \leq n$, let $(x_1, \dots, x_{n-i}) \in R^{n-i}$ for which
$
\# \{1 \leq m \leq n-i : x_m \in Q\} > |Q|.
$
Let $V$ be a braided vector space over $\kappa$ for which $(V, \kappa R(c))$ is an addable pair.
Then for every $w \in V^{\otimes i}$ the element $w \otimes x_1 \otimes \dots \otimes x_{n-i} \in V^{\otimes i} \otimes \kappa R(c)^{\otimes (n-i)}$ maps to $0$ in $H_0(B_{i,n-i}, V^{\otimes i} \otimes \kappa R(c)^{\otimes (n-i
)})$.

\end{cor}

\begin{proof}

By the pigeonhole principle there exist $1  \leq l < j \leq n-i$ for which 
$
x_l = x_j \in Q
$
so our statement follows by invoking \cref{QuandleConditionVanishingCoinvariants} with $r = x_l = x_j$.    
\end{proof}

\begin{cor} \label{QuandleVanishingCoinvariantsCor}

Let $R$ be a finite quandle, let $\kappa$ be a field, and let $c \colon R \times R \to \kappa^\times$ be a $2$-cocycle such that $c(x,x) \neq 1$ for every $x \in R$.
Let $V$ be a braided vector space over $\kappa$ for which $(V, \kappa R(c))$ is an addable pair.
Then $H_0(B_{i,n-i}, V^{\otimes i} \otimes \kappa R(c)^{\otimes (n-i)}) = 0$ for every $0 \leq i < n-|R|$, hence the inclusion $C(V) \to C(V \oplus \kappa R(c))$ of $\kappa$-algebras from \cref{InclusionGradedAlgebrasFromDirSumBVS} makes $C(V \oplus \kappa R(c))$ into a finitely generated left (respectively, right) $C(V)$-module.
    
\end{cor}

\begin{proof}

We invoke \cref{PigeonholePrincVanishingCoinvs} noting that $Q = R$ in view of our assumptions that $R$ is a quandle and $c(x,x) \neq 1$ for every $x \in R$, obtaining the required vanishing.
Therefore, \cref{DefAlgCoinvs} and \cref{BraidedSumCoinvs} tell us that
\[
\begin{split}
C(V \oplus \kappa R(c)) = \bigoplus_{n = 0}^{\infty} H_0(B_n, (V \oplus \kappa R(c))^{\otimes n}) &=
\bigoplus_{n = 0}^{\infty} \bigoplus_{i=0}^n H_0 (B_{i,n-i}, V^{\otimes i} \otimes \kappa R(c)^{\otimes (n-i)}) \\
&= \bigoplus_{n = 0}^{\infty} \bigoplus_{i=n-|R|}^n H_0 (B_{i,n-i}, V^{\otimes i} \otimes \kappa R(c)^{\otimes (n-i)})
\end{split}
\]
so we see that the finite set 
\[
\{x_1 \cdots x_j \in C(\kappa R(c)): 0 \leq j \leq |R|, \text{ and } (x_1, \dots, x_j) \in R^j\}
\]
generates $C(V \oplus \kappa R(c))$ as a left (respectively, right) $C(V)$-module.
\end{proof}

\begin{cor} \label{pr:0controlled}

Let $R$ be a finite quandle, let $\kappa$ be a field, and let $\zeta \neq 1$ be a root of unity in $\kappa^\times$. 
Then in the notation of \cref{BVStwistedByZetaWedgePM} we have $H_0(B_n, \kappa R_{\zeta}^{\otimes n}) = 0$ for all $n > |R|$.  
\end{cor}

    

\begin{proof}

Invoke \cref{QuandleVanishingCoinvariantsCor} with the constant cocycle whose value is $\zeta$, $V = 0$ and $i=0$, recalling \cref{TwistingAsCocycle}.
\end{proof}

For specific choices of $R$ we can give much sharper bounds than $|R|$ for the degree of the graded ring $C(\kappa R_\zeta)$.

\begin{prop} \label{pr:dexamples}

If $G$ is a finite group with a subgroup $H$ of index $2$, $R$ is a nonempty set of involutions in $G \setminus H$, and $\zeta \in \kappa^\times$ is a root of unity whose order does not divide the exponent of $H$, then 
$
\deg C(\kappa R_\zeta) = 1.
$

If $G$ is the symmetric group $S_N$, $R$ is the conjugacy class of transpositions, and $\zeta \in \kappa$ is a root of unity whose order does not divide $3$, then  
$
\deg C(\kappa R_\zeta) \leq \lfloor N/2 \rfloor.
$
\end{prop}


\begin{proof}
    For the first case, let $g_1,g_2 \in R$; we will show that $g_1 \cdot g_2 = 0$ in $C(\kappa R_\zeta)$. Our assumptions that $[G : H] = 2$ and $R \cap H = \emptyset$ imply that $g_2 = g_1 h$ for some $h \in H$. We claim that for every nonnegative integer $m$ we have in $C(\kappa R_\zeta)$ the equality
    $$
g_1 \cdot g_2 = \zeta^m g_1 h^{m} \cdot g_1 h^{m+1},
$$
where the elements $g_1h^m, g_1h^{m+1}$ of $G$ in fact lie in $R$.    
We prove the claim by induction on $m$, with the base case $m = 0$ being a consequence of our choice of $h$.     
Suppose now that $m$ is positive, use the induction hypothesis, and the braiding to get that 
\[
g_1 \cdot g_2 = \zeta^m g_1 h^{m} \cdot g_1 h^{m+1} = \zeta^{m+1} g_1 h^{m+1} \cdot (g_1 h^{m+1})^{-1}(g_1 h^{m})g_1 h^{m+1} = 
\zeta^{m+1} g_1 h^{m+1} \cdot h^{-1}g_1 h^{m+1}
\]
so to conclude the proof of the claim it suffices to check that $h^{-1}g_1 = g_1h$. Indeed our assumption that the elements of $R$ are of order $2$ gives   
$$
g_1 h = g_2 = g_2^{-1} = h^{-1} g_1^{-1} = h^{-1} g_1.
$$


Taking $m$ in the claim to be the order of $h$, we find that $g_1 \cdot g_2 = \zeta^m g_1 \cdot g_2$, and $\zeta^m \neq 1$ by the hypothesis on the exponent of $H$, so $g_1 \cdot g_2 = 0$ in $C(\kappa R_\zeta)$ as required.
We have thus shown that $\deg C(\kappa R_\zeta) \leq 1$. 
Because the group $B_1$ is trivial and $\kappa R_{\zeta}^{\otimes 1} = \kappa R_\zeta$ is nonzero as $R$ is nonempty by assumption, we get that $\deg C(\kappa R_\zeta) \geq 1$.

We now consider the second case - $R$ is the set of transpositions in $S_N$. 
We take $n > \lfloor N/2 \rfloor$, and $(\tau_1, \dots, \tau_n) \in R^n$ tasking ourselves with showing that $\tau_1 \cdots \tau_n = 0$ in $C(\kappa R_\zeta)$.
Our assumption that $n > \lfloor N/2 \rfloor$ implies that 
\[
2n \geq 2 (\lfloor N/2 \rfloor + 1) \geq 2(N/2 - 1/2 + 1) = N + 1 > N
\]
so there exist $1 \leq i < j \leq n$ and $1 \leq l \leq n$ for which $\tau_i(l) \neq l$ and $\tau_j(l) \neq l$.
Applying \cref{CoinvariantsUsingNabla} with $g = \sigma_{j-1} \cdots \sigma_{i+1} \in B_n$ we get that there exists $\nabla \in \kappa^\times$ for which
\[
\tau_1 \cdots \tau_n = 
\nabla \tau_1 \cdots \tau_i \cdot \tau_j \cdot \tau_{i+1}^{\tau_j} \cdots \tau_{j-1}^{\tau_j} \cdot \tau_{j+1} \cdots \tau_n 
\]
in $C(\kappa R_\zeta)$.
Applying \cref{CoinvariantsUsingNabla} to the right hand side above with $g = \sigma_i^3 \in B_n$ we get that
\[
\tau_1 \cdots \tau_n = \zeta^3 \tau_1 \cdots \tau_n
\]
in $C(\kappa R_\zeta)$ because our choice of $i$ and $j$ was such that the transpositions $\tau_i$ and $\tau_j$ generate a subgroup of (a copy of) $S_3$.
Therefore, our assumption that the order of $\zeta$ does not divide $3$ implies that $\tau_1 \cdots \tau_n = 0$ in $C(\kappa R_\zeta)$ as required. 
\end{proof}


Let $R$ be a rack, and let $\kappa$ be a field. Then as in \cref{TensoringWithKWedge} we have $\kappa R_{\wedge} = \kappa R \oplus \kappa R_{-1}$ as braided vector spaces over $\kappa$, so \cref{InclusionGradedAlgebrasFromDirSumBVS} provides us with an inclusion of graded $\kappa$-algebras $C(\kappa R) \to C(\kappa R_{\wedge})$.

\begin{cor}

Suppose that $R$ is a finite quandle, and that the characteristic of $\kappa$ is different from $2$.
Then the inclusion $C(\kappa R) \to C(\kappa R_{\wedge})$ of graded $\kappa$-algebras makes $C(\kappa R_{\wedge})$ into a finitely generated left (respectively, right) $C(\kappa R)$-module.
    
\end{cor}

\begin{proof}

Our assumptions that $R$ is a finite quandle, and the characteristic of $\kappa$ is different from $2$, allow us to apply \cref{QuandleVanishingCoinvariantsCor} with $V = \kappa R$, and $c$ the constant function with value $-1 \neq 1$, recalling from \cref{TwistingAsCocycle} that $\kappa R_{-1} = \kappa R(c)$.
\end{proof}


\begin{lem} \label{CokernelFiniteImpliesKernelIs}

Let $\kappa$ be a field, let $A = \bigoplus_{n = 0}^\infty A_n$ be a (not necessarily commutative, associative, or unital) graded $\kappa$-algebra, and let $h \in A$ be a homogeneous element.
Let $N = \bigoplus_{n = 0}^\infty N_n$ be a graded left $A$-module with $\dim_\kappa N_n < \infty$ for every $n \geq 0$.
Suppose that the cokernel of the $\kappa$-linear map on $N$ of multiplication by $h$ is finite-dimensional over $\kappa$.
Then the kernel of this map is also finite-dimensional over $\kappa$.
    
\end{lem}
    
\begin{proof}

Put $d = \deg h$, so that $h \in A_d$, and denote by
\[
M^h \colon N \to N, \qquad M_n^h \colon N_n \to N_{n + d}, \qquad M^h = \bigoplus_{n = 0}^\infty M^h_n, \qquad M^h_n(v) = hv, \qquad v \in N_n,
\]
the map of multiplication by $h$. 
Our assumption that $h$ is homogeneous implies that
\[
\operatorname{Ker}(M^h) = \bigoplus_{n=0}^\infty \operatorname{Ker}(M^h_n), \qquad 
\operatorname{Coker}(M^h) = \bigoplus_{n=0}^\infty \operatorname{Coker}(M^h_{n-d})
\]
where $N_n$ and $M^h_n$ are understood to be zero for negative $n$.
Therefore, our assumption that $\dim_\kappa \operatorname{Coker}(M^h) < \infty$ implies that $M^h_n$ is surjective for $n$ large enough.
Our finite-dimensionality assumption then gives us the inequality of nonnegative integers $\dim_\kappa N_{n+d} \leq \dim_\kappa N_n$.
Since nonincreasing sequences of nonnegative integers are eventually constant, and there are only finitely many residue classes modulo $d$, we conclude that $\dim_\kappa N_{n+d} = \dim_\kappa N_n$ for suitably large $n$.
Because a surjective linear map between vector spaces of the same dimension is an isomorphism, we conclude that $M^h_n$ is an isomorphism for suitably large $n$.
For these $n$ it then follows that $\operatorname{Ker}(M^h_n) = 0$ so $\operatorname{Ker}(M^h)$ is finite-dimensional over $\kappa$ as required.
\end{proof}

\begin{lem} \label{UpgradeSurjectivityByFiniteBit}

Let $\kappa$ be a field, let $A$ be a (not necessarily commutative or unital) $\kappa$-algebra, let $M$ be a right $A$-module that is finite-dimensional over $\kappa$, and let $N$ be a finitely generated left $A$-module. 
Then $M \otimes_A N$ is finite-dimensional over $\kappa$.
    
\end{lem}

\begin{proof}

The assumed finite generation of $N$ means that there exists an nonnegative integer $n$ and a surjection $A^{\oplus n} \to N$ of left $A$-modules.
This gives us a surjection $M \otimes_A A^{\oplus n} \to M \otimes_A N$ of vector spaces over $\kappa$.
We have $M \otimes_A A^{\oplus n} \cong M^{\oplus n}$ as vector spaces over $\kappa$, so our assumption that $\dim_\kappa M < \infty$ implies that 
\[
\dim_\kappa M \otimes_A N \leq \dim_\kappa M \otimes_A A^{\oplus n} = \dim_\kappa M^{\oplus n} = n \cdot \dim_\kappa M < \infty
\]
as required.
\end{proof}

\begin{cor} \label{IsomPowersAgree}

Let $R$ be a finite connected rack, let $\kappa$ be a field, and let $c \colon R \times R \to \kappa^\times$ be a $2$-cocycle. For a nonnegative integer $n$, using the notation of \cref{GeneratingNtuplesRack}, define 
\[
H_0(B_n, \kappa R(c)^{\otimes n})^\times = \operatorname{Span}_\kappa \{ x_1 \cdots x_n \in C(\kappa R(c)) : (x_1, \dots, x_n) \in R^n_\times\}.
\]
Then for every nonnegative integer $m$ and $(y_1, \dots, y_m) \in R^m$ the $\kappa$-linear map 
\[
M_{y_1 \cdots y_m} \colon H_0(B_n, \kappa R(c)^{\otimes n})^\times \to H_0(B_{n+m}, \kappa R(c)^{\otimes (n+m)})^\times, \qquad M_{y_1 \cdots y_m}(v) = y_1 \cdots y_m v, 
\]
of multiplication by $y_1 \cdots y_m$ from the left in $C(\kappa R(c))$ is an isomorphism for $n$ large enough.

Suppose moreover that the values of $c$ are $d$th roots of unity, 
and that $m$ is divisible by $d$, and by $P_w$ from \cref{CyclotomicCentralElementsPowers} for every $w \in R$, so that $w^m$
lies in the center of $C(\kappa R(c))$.
Then $M_{y^m} = M_{w^m}$ for every $y,w \in R$ and $n$ large enough.
    
\end{cor}

\begin{proof}

We start by showing that $M_{y_1 \cdots y_m}$ is surjective for $n$ large enough.
We need to show that for $n$ large enough, an $x_1 \cdots x_{n+m} \in H_0(B_{n+m}, \kappa R(c)^{\otimes (n+m)})^\times$ with $x_1, \dots, x_{n+m}$ being a generating set of $R$, lies in the image of $M_{y_1 \cdots y_m}$.
It follows from \cite{shus}*{Proposition 4.22} that for $n$ large enough there exists $(z_1, \dots, z_n) \in R^n_\times$ with and $g \in B_{n+m}$ such that 
\[
(y_1, \dots, y_m, z_1, \dots, z_n)^g = (x_1, \dots, x_{n+m}).
\]
\cref{CoinvariantsUsingNabla} then tells us that in $C(\kappa R(c))$ we have
\[
y_1 \cdots y_m \cdot z_1 \cdots z_n = \nabla(y_1, \dots, y_m, z_1, \dots, z_n;g) \cdot x_1 \cdots x_{n+m}.
\]
Since $z_1, \dots, z_n$ generate $R$ we have
\[
\nabla(y_1, \dots, y_m, z_1, \dots, z_n;g)^{-1} \cdot z_1 \cdots z_n \in H_0(B_n, \kappa R(c)^{\otimes n})^\times
\]
so $x_1 \cdots x_{n+m}$ does indeed lie in the image of $M_{y_1 \cdots y_m}$, and the required surjectivity is thus established.

To get that $M_{y_1 \cdots y_m}$ is an isomorphism for $n$ large enough we use its surjectivity and invoke \cref{CokernelFiniteImpliesKernelIs} with $A = C(\kappa R(c))$, $h = y_1 \cdots y_m$, and the homogeneous ideal 
\[
N = \bigoplus_{n=0}^\infty H_0(B_n, \kappa R(c)^{\otimes n})^\times
\]
of $A$.

At last for $n$ large enough and an $x_1 \cdots x_n \in H_0(B_n, \kappa R(c)^{\otimes n})^\times$ we need to show that in $C(\kappa R(c))$ we have
\[
w^m \cdot x_1 \cdots x_n = y^m \cdot x_1 \cdots x_n.
\]
Our assumptions that $R$ is finite, connected, and generated by $x_1, \dots, x_n$ imply that there exists a nonnegative integer $l$ and (not necessarily distinct) integers $1 \leq i_0, \dots, i_l \leq n$ such that the sequence in $R$ defined by
\begin{equation} \label{wSequenceForInduction}
w_0 = w, \qquad w_{j+1} = w_j^{x_{i_j}}, \qquad 0 \leq j \leq l, 
\end{equation}
satisfies $w_{l+1} = y$.

We claim that for every $0\leq j \leq l+1$ we have 
\[
w^m \cdot x_1 \cdots x_n = w_j^m \cdot x_1 \cdots x_n
\]
in $C(\kappa R(c))$.
We prove this by induction on $j$ with the base case $j = 0$ being tautological in view of \cref{wSequenceForInduction}.
Assuming now that $j$ is positive, we use the induction hypothesis, write an expanded version of our expression, apply \cref{CoinvariantsUsingNabla}, use our assumption that $d$ divides $m$, use the centrality of $m$th powers, and finally recall \cref{wSequenceForInduction} to get that
\[
\begin{split}
w^m \cdot x_1 \cdots x_n &= w_j^m \cdot x_1 \cdots x_n = w_j^m \cdot x_1 \cdots \cdot x_{i_{j}-1} \cdot x_{i_j} \cdot x_{i_{j+1}}\cdots x_n  \\
&= \nabla(x_1, \dots, x_{i_j}; \sigma_{i_j-1} \cdots \sigma_1) \cdot w_j^m x_{i_j} \cdot x_1^{x_{i_j}} \cdots \cdot x_{i_{j}-1}^{x_{i_j}} \cdot x_{i_{j+1}}\cdots x_n \\
&= \nabla(x_1, \dots, x_{i_j}; \sigma_{i_j-1} \cdots \sigma_1) \cdot x_{i_j} \cdot (w_j^{x_{i_j}})^m \cdot x_1^{x_{i_j}} \cdots \cdot x_{i_{j}-1}^{x_{i_j}} \cdot x_{i_{j+1}}\cdots x_n \\
&= \nabla(x_1, \dots, x_{i_j}; \sigma_{i_j-1} \cdots \sigma_1) \cdot (w_j^{x_{i_j}})^m \cdot x_{i_j} \cdot x_1^{x_{i_j}} \cdots \cdot x_{i_{j}-1}^{x_{i_j}} \cdot x_{i_{j+1}}\cdots x_n \\ 
&= w_{j+1}^m \cdot x_1 \cdots \cdot x_{i_{j}-1} \cdot x_{i_j} \cdot x_{i_{j+1}}\cdots x_n 
\end{split}
\]
and complete the induction, and thus also the proof of the claim.
The case $j = l+1$ of our claim is what we wanted to prove.
\end{proof}

\begin{defn} \label{DefDeredConnectedRack}

A rack $R$ is said to be hereditarily connected if every nonempty subrack of $R$ is connected.
    
\end{defn}

The following is identical to \cite{evw}*{Definition 3.1} except that we do not assume finiteness.

\begin{defn} \label{NonsplittingEVWinfinite}

Let $G$ be a group, and let $R$ be a conjugacy-closed generating set of $G$.
We say that $(G,R)$ is nonsplitting if for every subgroup $H$ of $G$ the intersection of $H$ with $R$ is either empty or a (single) conjugacy class of $H$.

\end{defn}

\begin{prop} \label{eg:hcnonsplitting}

A conjugacy-closed generating set $R$ of a group $G$ is hereditarily connected (as a rack) if and only if $(G,R)$ is nonsplitting. 

\end{prop}

\begin{proof}

Suppose first that $R$ is hereditarily connected, and let $H$ be a subgroup of $G$ with $R \cap H \neq \emptyset$.
Our task is to show that $R \cap H$ is a (single) conjugacy class of $H$.
Our assumption that $R$ is a conjugacy-closed generating set of $G$ implies that, when we view $G$ as a rack, $R$ is an ideal of $G$.
Therefore $R \cap H$ is an ideal (in particular, a subrack) of $H$, namely $R \cap H$ is a conjugacy-closed subset of $H$ in view of \cref{IdealsInGroups}.
Our assumptions that $R$ is hereditarily connected, and $R \cap H$ being nonempty, imply that $R \cap H$ is a connected rack.
In view of \cref{SingleConjugacyClassConnected}, this means that $R \cap H$ is a (single) conjugacy class of $H$.

Suppose now that $(G, R)$ is nonsplitting, and let $S$ be a nonempty subrack of $R$.
Our goal is to show that $S$ is connected.
Denote by $H$ the subgroup of $G$ generated by $S$, and note that $S \subseteq H \cap R$, in particular $H \cap R$ is nonempty, so $H \cap R$ is a (single) conjugacy class of $H$ by our assumption that $(G,R)$ is nonsplitting.
In other words, the action of $H$ by conjugation on $H \cap R$ is transitive.
\cref{StructureGroupRackInGroup} tells us that the action of the structure group $\Gamma_{H \cap R}$ on $H \cap R$ factors via the group homomorphism $\Gamma_{H \cap R} \to H$.
Since $S$ generates $H$ the composition of the group homomorphisms $\Gamma_S \to \Gamma_{H \cap R} \to H$ is surjective, so the action of $\Gamma_S$ on $H \cap R$ is transitive. 
It thus follows from \cref{SequalsRtransitive} that $S = H \cap R$, hence $S$ is a connected rack in light of \cref{SingleConjugacyClassConnected} because $H \cap R$ generates $H$ since it contains $S$ (which generates $H$).
The proof of both implications is thus complete.
\end{proof}

\begin{defn}[\cite{el}*{Definition 4.1.1}] \label{def:controlled}

We say that a braided vector space $V$ over a field $\kappa$ is $1$-controlled if there exists a homogeneous central element $U \in C(V)$ such that the $\kappa$-linear map of multiplication by $U$ on $C(V)$ has kernel and cokernel supported in finitely many degrees. We say that a coefficient system $(V, W)$ for $B_{g,f}^n$ is $1$-controlled if $V$ is.

\end{defn}

Homological stability for $1$-controlled coefficient systems is obtained in \cite{el}*{Theorem 4.2.6}.  We note that the stability slope obtained in that result can be made explicit, at least in principle.

Next we extend \cite{evw}*{Lemma 3.5} and \cite{el}*{Proposition A.31}.

\begin{lem} \label{evw3Point5}

Let $R$ be a hereditarily connected finite rack, let $\kappa$ be a field whose characteristic does not divide $|S|$ for any nonempty subrack $S$ of $R$, and let $c \colon R \times R \to \kappa^\times$ be a $2$-cocycle valued in $d$th roots of unity. Then for a positive integer $m$ divisible by $d$ and by $P_w$ from \cref{CyclotomicCentralElementsPowers} for every $w \in R$, the kernel (respectively, cokernel) of multiplication in $C(\kappa R(c))$ by the homogeneous central element
\[
h = \sum_{r \in R} r^m \in Z(C(\kappa R(c)))
\]
is finite-dimensional over $\kappa$.
Hence $\kappa R(c)$ is $1$-controlled. 
    
\end{lem}

\begin{proof}

For a positive integer $\nu$ put
\[
H_0(B_n, \kappa R(c)^{\otimes n})_\nu = \operatorname{Span}\{x_1 \cdots x_n \in C(\kappa R(c)) : (x_1, \dots, x_n) \in R^n, \quad \# \langle x_1, \dots, x_n \rangle \geq \nu\}.
\]
We claim that for $n$ large enough the $\kappa$-linear map
\[
M_{h, \nu} \colon H_0(B_n, \kappa R(c)^{\otimes n})_\nu \to H_0(B_{n+m}, \kappa R(c)^{\otimes (n+m)})_\nu, \qquad M_{h,\nu}(u) = hu 
\]
of multiplication by $h$ in $C(\kappa R(c))$ is surjective.

We prove our claim by descending induction on $\nu$, with the base case $\nu = |R| + 1$ being true because of our assumption that $R$ is finite, and the surjectivity of any map to the $0$ vector space.
Suppose now that $\nu \leq |R|$, that $n$ is large enough, and let $y_1, \dots, y_{n+m}$ be generators of a subrack $S$ of $R$ with $|S| \geq 
\nu$. Our task is to show that $y_1 \cdots y_{n+m}$ lies in the image of $M_{h,\nu}$.
Our assumption that $R$ is hereditarily connected implies that $S$ is connected, so \cref{IsomPowersAgree} applied to $S$ and $c|_{S \times S}$ provides us with generators $x_1, \dots, x_n$ of $S$ such that
\begin{equation*} \label{PowersOfsSameSurj}
s^m \cdot x_1 \cdots x_n = y_1 \cdots y_{n+m}
\end{equation*}
in $C(\kappa S(c|_{S \times S}))$ for every $s \in S$.
We thus have
\[
\begin{split}
h \cdot x_1 \cdots x_n = \sum_{r \in R} r^m \cdot x_1 \cdots x_n &= \sum_{s \in S} s^m \cdot x_1 \cdots x_n + \sum_{r \in  R \setminus S} r^m \cdot x_1 \cdots x_n \\ &= |S| \cdot y_1 \cdots y_{n+m} + \sum_{r \in  R \setminus S} r^m \cdot x_1 \cdots x_n.
\end{split}
\]

For each $r \in R \setminus S$ the subrack of $R$ generated by $r,x_1, \dots, x_n$ properly contains $S$, so its number of elements exceeds $\nu$, hence the induction hypothesis tells us that $r^m \cdot x_1 \cdots x_n$ lies in the image of $M_{h,\nu +1}$, and thus also in the image of $M_{h, \nu}$.
Clearly, $h \cdot x_1 \cdots x_n$ lies in the image of $M_{h, \nu}$ as well, so we conclude that $|S| \cdot y_1 \cdots y_{n+m}$ is in the image of $M_{h, \nu}$ too.
Our assumption on the characteristic of $\kappa$ allows us to conclude that $y_1 \cdots y_{n+m}$ lies in the image of $M_{h,\nu}$, completing the induction and thus the proof of the claim.

The special case $\nu = 1$ of our claim guarantees that the cokernel of multiplication by $h$ in $A = C(\kappa R(c))$ is finite-dimensional over $\kappa$.
Invoking \cref{CokernelFiniteImpliesKernelIs} with $N = A$ we get that the kernel of this map is finite-dimensional as well.
\end{proof}

\begin{cor} \label{OneControlledGeneralizedWedge}

Let $R$ be a finite rack, let $\kappa$ be a field, and let $c \colon R \times R \to \kappa^\times$ be a $2$-cocycle. 
Let $S,T \subseteq R$ be a partition such that $S$ is hereditarily connected, $T$ is a quandle, the characteristic of $\kappa$ does not divide the order of any nonempty subrack of $S$, the values of $c|_{S \times S}$ are $d$th roots of unity, and $c(t,t) \neq 1$ for every $t \in T$. 
Then for a positive integer $m$ divisible by $d$ and by $P_s = \operatorname{lcm}_{r \in R} \mathfrak p(r,s)$ from \cref{CyclotomicCentralElementsPowers} for every $s \in S$, the homogeneous element
\begin{equation} \label{InclusionToExplain}
h = \sum_{s \in S} s^m \in C(\kappa S(c|_{S \times S})) \hookrightarrow C(\kappa R(c))
\end{equation}
lies in the center of $C(\kappa R(c))$, and the kernel (respectively, cokernel) of the map of left multiplication by $h$ on $C(\kappa R(c))$ is finite-dimensional over $\kappa$. Hence $\kappa R(c)$ is $1$-controlled.
    
\end{cor}

\begin{proof}

As in \cref{AddableDecompositionRackCocyclePartition} we have an isomorphism of braided vector spaces
\[
\kappa R(c) \cong \kappa S(c|_{S \times S}) \oplus \kappa T(c|_{T \times T})
\]
so \cref{InclusionGradedAlgebrasFromDirSumBVS} gives us the inclusion of graded $\kappa$-algebras appearing in \cref{InclusionToExplain}.

\cref{CyclotomicCentralElementsPowers} guarantees that for every $s \in S$ the degree $m$ element $s^m$ lies in the center of $C(\kappa R(c))$ in view of our assumption that $P_s$ divides $m$, so the degree $m$ homogeneous element $h$ lies in this center as well.

Using our assumption that $S$ is finite and hereditarily connected we get from \cref{evw3Point5} that the kernel (respectively, cokernel) of multiplication by $h$ on $C(\kappa S(c|_{S \times S}))$ is finite-dimensional over $\kappa$.
In particular for the associative $\kappa$-algebra $A = C(\kappa S(c|_{S \times S}))$, the quotient $A$-module $M = A/(h)$ is finite-dimensional over $\kappa$.
We infer from \cref{QuandleVanishingCoinvariantsCor} that $N = C(\kappa R(c))$ is finitely generated as a left $A$-module.
\cref{UpgradeSurjectivityByFiniteBit} then tells us that the vector space
\[
M \otimes_A N = A/(h) \otimes_A N = N/hN
\]
is finite-dimensional over $\kappa$, or equivalently, the cokernel of multiplication by $h$ on $C(\kappa R(c))$ is finite-dimensional over $\kappa$. At last, from \cref{CokernelFiniteImpliesKernelIs} applied to the graded $\kappa$-algebra $C(\kappa R(c))$ as a module over itself, we get that the kernel of multiplication by $h$ on $C(\kappa R(c))$ is finite-dimensional over $\kappa$.
\end{proof}

For the following recall \cref{WedgeTwistedRackExample} and its notation.

\begin{cor} \label{co:nonsplitting1controlled}

Let $R$ be a finite hereditarily connected quandle, and let $\kappa$ be a field whose characteristic is not $2$ and does not divide the order of any nonempty subrack of $R$.
Then for a positive integer $m$ divisible by $\mathfrak q(x,y)$ from \cref{CyclotomicCentralElementsPowers} for every $x,y \in R$, the kernel (respectively, cokernel) of multiplication in $C(\kappa R_{\wedge})$ by
\[
h = \sum_{r \in R} r_1^m + r_{-1}^m \in Z(C(\kappa R_{\wedge}))
\]
is finite-dimensional over $\kappa$.
Hence $\kappa R_\wedge$ is $1$-controlled.

\end{cor}

\begin{proof}

By \cref{RwedgePMcocycles} we have $\kappa R_{\wedge} = \kappa [R \times \mathcal T_2](c_{\wedge})$ in the notation of \cref{CocycleWedgePMexmp}, so we apply \cref{OneControlledGeneralizedWedge} with the partition $R_\varphi, R_\psi \subseteq R \times \mathcal T_2$ from \cref{ProductRackPartition}, and $d = 1$, observing that $c_\psi$ is the constant function $-1 \neq 1$,
and that an integer is divisible by $\mathfrak q(x,y)$ for every $x,y \in R$ if and only if it is divisible by $\operatorname{lcm}_{w \in R \times \mathcal T_2} \mathfrak p(w,s)$ for every $s \in R_{\varphi}$.
\end{proof}


Next we adopt the setup and notation of \cref{ProductRackPartition}.
In particular $R$ is a rack, $\kappa$ a field, and $c$ a $\kappa^\times$-valued $2$-cocycle on $R \times \mathcal T_2$.
By \cref{AddableDecompositionRackCocyclePartition} we then have
\[
\kappa [R \times \mathcal T_2](c) \cong \kappa R_{\varphi}(c_\varphi) \oplus \kappa R_{\psi}(c_\psi)
\]
as braided vector spaces.
\cref{InclusionGradedAlgebrasFromDirSumBVS} then gives us the inclusions of graded $\kappa$-algebras
\begin{equation} \label{InclusionT2ProdCoc}
C(\kappa R_{\varphi}(c_\varphi)) \to C(\kappa [R \times \mathcal T_2](c)), \quad
C(\kappa R_{\psi}(c_\psi)) \to C(\kappa [R \times \mathcal T_2](c)), 
\end{equation}
and \cref{SurjectionGradedAlgebrasFromDirSumBVS} gives us the surjections of graded $\kappa$-algebras 
\begin{equation} \label{ProjectionT2ProdCoc}
C(\kappa[R \times \mathcal T_2](c)) \to C(\kappa R_{\varphi}(c_\varphi)), \quad
C(\kappa[R \times \mathcal T_2](c)) \to C(\kappa R_{\psi}(c_\psi)).
\end{equation}

\begin{thm} \label{OneControlledCocycleT2}

Suppose that $R$ is a finite hereditarily connected quandle, that $c$ is valued in the $d$th roots of unity in $\kappa$, and that for every $r \in R$ we have 
\begin{equation} \label{ValuesOfCarentInverses}
c((r, \varphi),(r, \psi)) \cdot c((r, \psi),(r, \varphi)) \neq 1.
\end{equation}
Then 
for some positive integer $N$ we have an isomorphism of $\kappa$-algebras
\[
\bigoplus_{n \geq N} \kappa [R \times \mathcal T_2](c)^{\otimes n}_{B_n} \cong
\left( \bigoplus_{n \geq N} \kappa R_\varphi(c_\varphi)^{\otimes n}_{B_n} \right) \times
\left( \bigoplus_{n \geq N} \kappa R_\psi(c_\psi)^{\otimes n}_{B_n} \right)
\]
induced by the inclusion and projection maps from \cref{InclusionT2ProdCoc} and \cref{ProjectionT2ProdCoc}.

Suppose moreover that the characteristic of $\kappa$ does not divide the order of any nonempty subrack of $R$.
Then for a positive integer $m$ divisible by $d$ and by $P_w$ from \cref{CyclotomicCentralElementsPowers} for every $w \in R \times \mathcal T_2$, the kernel (respectively, cokernel) of multiplication in $C(\kappa[R \times \mathcal T_2](c))$ by
\[
h = \sum_{\pi \in R \times \mathcal T_2} \pi^m = \sum_{r \in R} (r,\varphi)^m + (r, \psi)^m \in Z(C(\kappa[R \times \mathcal T_2](c)))
\]
is finite-dimensional over $\kappa$.
Hence $\kappa [R \times \mathcal T_2](c)$ is $1$-controlled.

\end{thm}

\begin{proof}


We first claim that $(r, \varphi)(r,\psi) = 0$ in $C(\kappa[R \times \mathcal T_2](c))$ for every $r \in R$.
Indeed \cref{CoinvariantsUsingNabla} with $g = \sigma_1^2 \in B_{2}$, our assumption that $R$ is a quandle, the fact that $\sigma_1^2$ fixes $((r, \varphi),(r, \psi)) \in (R \times \mathcal T_2)^2$, and \cref{TwoExamplesForNabla} tell us that
\[
(r, \varphi)(r,\psi) = c((r, \varphi),(r,\psi)) \cdot c((r, \psi), (r,\varphi)) \cdot (r, \varphi)(r, \psi)
\]
in $C(\kappa[R \times \mathcal T_2](c))$, so from our assumption in \cref{ValuesOfCarentInverses} we get that $(r, \varphi)(r,\psi) = 0$.

Next we take $(x_1, \dots, x_{n}) \in R^{n}$ for $n$ large enough, and task ourselves with showing that the element
\[
e = (x_1, \varphi) \cdots (x_{n-1}, \varphi) \cdot (x_{n}, \psi)
\]
in $C(\kappa[R \times \mathcal T_2](c))$ is $0$.
Since $n$ is large enough and $R$ is finite, it follows from the pigeonhole principle that there exist $1 \leq i < j \leq n$ for which $x_i = x_j$. 
Applying \cref{CoinvariantsUsingNabla} with $g = \sigma_{j-1} \cdot \ldots \cdot \sigma_{i+1} \in B_{n}$ we get that there exists $\nabla_1 \in \kappa^\times$ for which
\[
e = \nabla_1 \cdot 
(x_1, \varphi) \cdots (x_i, \varphi) \cdot (x_i, \chi) \cdot (x_{i+1}^{x_i}, \varphi) \cdots (x_{j-1}^{x_i}, \varphi) \cdot (x_{j+1}, \varphi) \cdots (x_{n-1}, \varphi) \cdot (x_n, \psi)
\]
in $C(\kappa[R \times \mathcal T_2](c))$ where $\chi = \varphi$ if $j < n$ and $\chi = \psi$ if $j = n$.
In the latter case, it follows from the claim in the beginning of the proof that $(x_i, \varphi)(x_i, \chi) = 0$ hence $e = 0$, so we assume from now on that we are in the former case, namely $j < n$ and $\chi = \varphi$.






Let $Q$ be the subquandle of $R$ generated by $x_1, \dots, x_{n}$.
Our assumption that $R$ is hereditarily connected implies that $Q$ is connected.
Since $n$ is large enough, we get from \cite{shus}*{Theorem 2.4} that the restriction of the group homomorphism $B_n \to S_n$ to the stabilizer of 
\[
(x_1, \dots x_i, x_i, x_{i+1}^{x_i}, \dots x_{j-1}^{x_i}, x_{j+1}, \dots x_{n-1}, x_n) \in Q^{n}
\]
is surjective.
In particular, there exists a braid $b$ in the stabilizer that maps to the transposition $(i+1 \ \   n)$, so applying \cref{CoinvariantsUsingNabla} with $b$ we get that there exists $\nabla_2 \in \kappa^\times$ such that
\[
e = \nabla_1 \cdot \nabla_2 \cdot 
(x_1, \varphi) \cdots (x_i, \varphi) \cdot (x_i, \psi) \cdot (x_{i+1}^{x_i}, \varphi) \cdots (x_{j-1}^{x_i}, \varphi) \cdot (x_{j+1}, \varphi) \cdots (x_{n-1}, \varphi) \cdot (x_n, \varphi)
\]
in $C(\kappa[R \times \mathcal T_2](c))$.
As before, we get from the claim in the beginning of the proof that $e = 0$.

Our argument remains valid when $\varphi$ and $\psi$ are interchanged, so the addition map
\[
\left( \bigoplus_{n \geq N} \kappa R_\varphi(c_\varphi)^{\otimes n}_{B_n} \right) \times
\left( \bigoplus_{n \geq N} \kappa R_\psi(c_\psi)^{\otimes n}_{B_n} \right) \to \bigoplus_{n \geq N} \kappa[R \times \mathcal T_2](c)^{\otimes n}_{B_n}
\]
is a homomorphism of $\kappa$-algebras for some positive integer $N$.
Because for every $0 \leq i \leq n$ and $n$ sufficiently large, either $i$ or $n-i$ is sufficiently large, this addition map is also surjective, and thus an isomorphism as required.



It then follows from \cref{evw3Point5} applied to $(R_\varphi, c_\varphi)$ and to $(R_\psi, c_\psi)$ that the kernel (respectively, cokernel) of multiplication by $h$ in $C(\kappa[R \times \mathcal T_2](c))$ is finite-dimensional over $\kappa$.
\end{proof}

For the following recall \cref{PMtwistBVS} and its notation.

\begin{cor} \label{co:pm1controlled}

Let $R$ be a finite hereditarily connected quandle, and let $\kappa$ be a field of characteristic different from $2$ and not dividing the order of any nonempty subrack of $R$.
Then for an even positive integer $m$ divisible by $\mathfrak q(x,y)$ from \cref{CyclotomicCentralElementsPowers} for every $x,y \in R$, the kernel (respectively, cokernel) of multiplication in $C(\kappa R_{\pm})$ by
\[
h = \sum_{r \in R} r^m + \underline{r}^m \in Z(C(\kappa R_{\pm}))
\]
is finite-dimensional over $\kappa$.
Hence $\kappa R_{\pm}$ is $1$-controlled.

\end{cor}

\begin{proof}

By \cref{RwedgePMcocycles} we have $\kappa R_{\pm} = \kappa [R \times \mathcal T_2](c_{\pm})$ in the notation of \cref{CocycleWedgePMexmp}, so we apply \cref{OneControlledCocycleT2} with $d = 2$ observing that \cref{ValuesOfCarentInverses} is indeed satisfied,
and that an even integer is divisible by $\mathfrak q(x,y)$ for every $x,y \in R$ if and only if it is divisible by $P_w$ for every $w \in R \times \mathcal T_2$.
\end{proof}

\begin{rem}  

It follows from \cref{co:pm1controlled} and \cite{el}*{Theorem 4.2.6} that for every nonnegative integer $i$ the $C(\kappa R_\pm)$-module 
\[
M_i = \bigoplus_{n=0}^\infty H_i(B_n, \kappa R_\pm^{\tensor n})
\]
is finitely generated, and gives bounds on the degrees of generators. But in the special case where $R$ is trivial, Hoang's analysis in \cite{anh} does more; it shows that $M_i$ is a torsion module for $C(\kappa_\pm) = \kappa[x,y]/(xy)$, and indeed has finite total dimension over $\kappa$.  Is that true for more general $R$?  One might ask whether an even stronger statement is the case, in the spirit of Landesman and Levy's results in \cite{ll}:  is the natural map from $H_i(B_n, \kappa R_\pm^{\tensor n})$ to $H_i(B_n, \kappa_\pm^{\tensor n})$ an isomorphism when $i$ is small relative to $n$? 

\end{rem}






\subsection{Duality}

Recall that a symmetric monoidal category $\mathcal C$ is called a closed compact category if every object of $\mathcal C$ admits a dual object.

\begin{exmp}

The category of 
(graded) finite-dimensional vector spaces over a field is compact closed.

\end{exmp}

\begin{defn}

Let $\mathcal C$ be a closed compact category, and let $W$ be a braided object in $\mathcal C$ with braiding $T_W \colon W \otimes W \to W \otimes W$. Endow $W^\vee$ with the braiding $T_W^\vee \colon W^{\vee}  \otimes W^{\vee}  \to W^{\vee} \otimes W^{\vee}$ obtained by identifying $W^\vee \otimes W^\vee$ with $(W \otimes W)^\vee$.

\end{defn}

The category of braided objects in $\mathcal C$ is thus a compact closed category. In particular, for braided objects $V$ and $W$ of $\mathcal C$ we have 
\begin{equation*} 
(V \otimes W)^\vee \cong V^\vee \otimes W^\vee
\end{equation*}
as braided objects in $\mathcal C$.

For a braided object $V$ in $\mathcal C$, the action of $B_n$ on $(V^{\vee})^{\otimes n}$, introduced in \cref{BraidedToNaturoid}, is dual to the action of $B_n$ on $V^{\otimes n}$.

\begin{exmp} \label{DualityBVSexmp}

Let $\kappa$ be a field, let $R$ be a rack, and let $c \colon R \times R \to \kappa^\times$ be a $2$-cocycle. Then $\kappa R(c)^\vee \cong \kappa R(\operatorname{inv} \circ c)$ as braided vector spaces, where $\operatorname{inv} \colon \kappa^\times \to \kappa^\times$ is the inversion map.
In particular, we have $\kappa R^\vee \cong \kappa R$, for $\zeta \in \kappa^\times$ we have $\kappa R_\zeta^\vee = \kappa R_{\zeta^{-1}}$ in view of \cref{TwistingAsCocycle}, and $\kappa R_{\wedge}^\vee \cong \kappa R_{\wedge}$, $\kappa R_{\pm}^{\vee} \cong \kappa R_\pm$ in view of \cref{RwedgePMcocycles}.

\end{exmp}

Suppose that $\mathcal C$ is moreover additive (monoidal).
Let $(U,V)$ be an addable pair of braided objects in $\mathcal C$ with respect to $T_{U,V}$ and $T_{V,U}$. Then $(U^\vee, V^\vee)$ becomes an addable pair of braided objects in $\mathcal C$ once we set $T_{U^\vee, V^\vee} = T_{V,U}^\vee$ and $T_{V^\vee, U^\vee} = T_{U,V}^\vee$. 
We then have 
$
(U \oplus V)^\vee \cong U^\vee \oplus V^\vee
$
as braided objects in $\mathcal C$, where the direct sum is taken with respect to the above addable pairs.
If the sum on the left hand side is plain, then so is the one on the right hand side.


\section{Multiplicative Characters of Discriminants} \label{SecMulCharDiscGenZero}

We state and prove a version of \cref{CharDiscCancThm} with explicit savings, and an explicit dependence of $q$ on $R$, using material from \cref{CharSheavesSection}.
Let $\zeta \in \overline{\F_\ell}$ be a root of unity of order $\mathfrak o$, let $\kappa = \F_\ell[\zeta]$, and put $d = \deg C(\kappa R_{\zeta})$, which is finite by \cref{pr:0controlled}. Then as $n = n_1 + \dots + n_k \to \infty$ we have 
\begin{equation} \label{pr:discdisc}
\left| \sum_{L \in \mathcal E_q^R(G;n_1, \dots, n_k)} \chi(\disc(f_L)) \right| \leq 2^{n-1} |R|^n q^{n - \frac{n-d}{2d+4}}
\end{equation}
so we get a power saving if $q > (2|R|)^{2d+4}$.

\begin{rem}

The natural action of $\Gm$ on $\A^1$ induces an action of the group $\Gm(\F_q) = \F_q^\times$ on the set $\mathcal E_q^R(G;n_1, \dots, n_k)$; if $\theta \in \Gm(\F_q)$ and $L \in \mathcal E_q^R(G;n_1, \dots, n_k)$, then 
\[
\disc({f_{\theta L}}(t)) = \disc(\theta^{-n}f_L(\theta t)) = \theta^{n(1-n)} \disc f_L(t).
\]
The action of $\F_q^\times$ on $\mathcal E_q^R(G;n_1, \dots, n_k)$ is free because for every monic squarefree $f \in \F_q[t]$ of degree at least $2$ and every $\theta \neq 1$, we have $f(\theta t) \neq \theta^n f(t)$ as polynomials. 
It follows that $\chi(\disc(f_L))$ is exactly equidistributed on $\mathcal E_q^R(G;n_1, \dots, n_k)$, and thus the sum in \cref{pr:discdisc} is $0$, unless $\mathfrak o$ is a divisor of $n(n-1)$.

\end{rem}

\begin{proof}


In view of \cref{TraceFunctionCharacterSheafDiscriminant} and \cref{PushforwardFromHurwitzSpaceTrace} the sheaf $\tau_* \pi_* \overline{\mathbb Q_\ell} \otimes \delta^{-1} \mathcal L_\chi$ is lisse, punctually pure of weight $0$, 
and its trace function on a degree $n$ monic squarefree $g \in \F_q[t]$ is given by
\[
\operatorname{tr}(\operatorname{Frob}_{\bar g}, \tau_* \pi_* \overline{\mathbb Q_\ell} \otimes \delta^{-1} \mathcal L_\chi) = \#\{L \in \mathcal E_q^R(G;n_1, \dots, n_k) : f_L = g\} \cdot \chi(\disc(g)).
\]
We thus have
\[
\begin{split}
\sum_{L \in \mathcal E_q^R(G;n_1, \dots, n_k)} \chi(\disc(f_L)) &= \sum_{g \in \operatorname{Conf}^n(\F_q)} \#\{L \in \mathcal E_q^R(G;n_1, \dots, n_k) : f_L = g\} \cdot \chi(\disc(g)) \\
&= \sum_{g \in \operatorname{Conf}^n(\F_q)} \operatorname{Tr}(\operatorname{Frob}_{\bar g}, \tau_* \pi_* \overline{\mathbb Q_\ell} \otimes \delta^{-1} \mathcal L_\chi). 
\end{split}
\]
The Grothendieck--Lefschetz trace formula gives
\[
\sum_{g \in \operatorname{Conf}^n(\F_q)} \operatorname{tr}(\operatorname{Frob}_{\bar g}, \tau_* \pi_* \overline{\mathbb Q_\ell} \otimes \mathcal \delta^{-1} \mathcal L_\chi) = 
\sum_{i=0}^{2n} (-1)^i \operatorname{tr} (\Frob_q, H^i_c(\Conf^n \A^1_{\overline{\F_q}}, \tau_* \pi_* \overline{\mathbb Q_\ell} \otimes \delta^{-1} \mathcal L_\chi)).
\]
By the triangle inequality, the absolute value of the right hand side is bounded from above by
\[
\sum_{i=0}^{2n} \left| \operatorname{tr} (\Frob_q, H^i_c(\Conf^n \A^1_{\overline{\F_q}}, \tau_* \pi_* \overline{\mathbb Q_\ell} \otimes \delta^{-1} \mathcal L_\chi))\right|.
\]
Since $\tau_* \pi_* \overline{\mathbb Q_\ell} \otimes \delta^{-1} \mathcal L_\chi$ is punctually pure of weight $0$, \cite{Deligne} bounds from above the absolute value of each eigenvalue of $\operatorname{Frob}_q$ on $H^i_c(\Conf^n \A^1_{\overline{\F_q}}, \tau_* \pi_* \overline{\mathbb Q_\ell} \otimes \delta^{-1} \mathcal L_\chi)$ by $q^{i/2}$, so the sum above is bounded by
\[
\sum_{i=0}^{2n} q^{i/2} \cdot \dim_{\overline{\Q_\ell}} H^i_c(\Conf^n \A^1_{\overline{\F_q}}, \tau_* \pi_* \overline{\mathbb Q_\ell} \otimes \delta^{-1} \mathcal L_\chi) \leq 
\sum_{i=0}^{2n} q^{i/2} \cdot \dim_\kappa H^i_c(\Conf^n \A^1_{\overline{\F_q}}, \tau_* \pi_* \kappa \otimes \delta^{-1} \mathcal L_\chi).
\]

In view of \cref{SignInflation}, the analytification of the base change of $\tau_* \pi_* \kappa \otimes \delta^{-1} \mathcal L_\chi$ to $\C$ corresponds to a direct summand of the representation $\kappa R^{\otimes n} \otimes \kappa_{\zeta}^{\otimes n} \cong \kappa R_{\zeta}^{\otimes n}$ of $B_n$. It then follows from \cref{ComparisonCohomCor}, \cref{DualityBVSexmp}, and the sentence preceding it, that the right hand side is bounded from above by
\[
\sum_{i=0}^{2n} q^{i/2} \cdot \dim_\kappa H_{2n-i}(B_n, (\kappa R_\zeta^{\otimes n})^{\vee}) = 
\sum_{i=0}^{2n} q^{i/2} \cdot \dim_\kappa H_{2n-i}(B_n, \kappa R_{\zeta^{-1}}^{\otimes n}).
\]



By \cref{th:0controlledvanishing} we have 
\[ 
\sum_{i=0}^{2n} q^{i/2} \cdot \dim_\kappa H_{2n-i}(B_n, \kappa R_{\zeta^{-1}}^{\otimes n}) = 
\sum_{\frac{n-d}{d+2} \leq j \leq 2n} q^{n - \frac{j}{2}} \cdot \dim_\kappa H_{j}(B_n, \kappa R_{\zeta^{-1}}^{\otimes n}).
\]
It follows from \cref{Salvetti} that
\[
\sum_{\frac{n-d}{d+2} \leq j \leq 2n} q^{n - \frac{j}{2}} \cdot \dim_\kappa H_{j}(B_n, \kappa R_{\zeta^{-1}}^{\otimes n}) \leq q^{n - \frac{n-d}{2d+4}} \dim_\kappa \kappa R_{\zeta^{-1}}^{\otimes n} \sum_{j=0}^{n-1} \binom{n-1}{j} = 2^{n-1} |R|^n q^{n - \frac{n-d}{2d+4}}.
\]

    
\end{proof}

\begin{rem}
    It's not clear how sharp one can expect \cref{pr:discdisc} to be.  However, the case $G = \Z/2\Z$, $R = \{1\}$ is instructive.  In this case, $d = 1$, so \cref{th:0controlledvanishing} shows that $H_i(B_n, \kappa_\zeta^{\tensor n})$ vanishes for $i < (n-1)/3$. The cohomology of $B_n$ with the given coefficients was computed by Frenkel~\cite{frenkelcommutator} and, for $\zeta$ a nontrivial cube root of unity, there is indeed a nontrivial class in $H_{n/3}(B_n, \kappa_\zeta^{\tensor n})$ when $3|n$, suggesting that any substantially better range will require an argument that takes into account the multiplicative order of $\zeta$.  Given the recent result of Landesman--Levy that the stable homology of (connected components of) Hurwitz spaces maps down isomorphically to the the stable homology of configuration space, it seems very reasonable to ask whether the natural map
$$
H_i(B_n, \kappa R_\zeta^{\tensor n}) \ra H_i(B_n, \kappa _\zeta^{\tensor n})
$$
is an isomorphism in a range $i < \beta n$ (with $\beta$ possibly depending on $R$ and $\zeta$); in that case, the homology group on the left would be well-controlled by Frenkel's computation of the homology group on the right.
\end{rem}





\section{M\"obius Function in Higher Genus} \label{HigherGenus}

Here we state and prove an analog of the result in the previous section for the M\"obius function on curves of higher genus. Since this is similar, our treatment will be more abbreviated.

It is easy to see that the cohomology of $H_n(\C)$ with coefficients in this local system is exactly $H^*(B_{g,f}^n, (W \tensor (k[c] \tensor V_{-1})^{\tensor n})$.  These are cohomology groups we have already shown to be zero in some range in \cref{th:0controlledvanishing}.  That theorem allows us to prove an upper bound for the sum of M\"{o}bius of discriminants of $G$-extensions.

Recall that for a squarefree divisor $Z = \sum_{Y \in S} Y$ on a variety $X$ over a field $\F$, where $S$ is a finite set of (distinct) irreducible subvarieties of $X$ over $\F$, the M\"obius function is given by $(-1)^{\# S}$.
In case $\F$ is finite, the M\"obius function of $Z$ equals $(-1)^{\deg Z}$ times the sign of the (conjugacy class of the) permutation induced by $\operatorname{Frob}_{\#\F}$ on the support of $Z$ over $\overline{\F}$.

\begin{prop} \label{pr:mobdisc}

Let $G$ be a nontrivial finite group, let $q$ be a prime power large enough in terms of $|G|$ and coprime to it, let $K$ be the function field of a smooth projective curve $X$ of genus $g$ over $\F_q$, let $\infty$ be an $\F_q$-point of $K$, and let $D$ be a squarefree divisor on $X$ of degree $f$ which is disjoint from $\infty$. Let $R \subseteq G$ be a generating set which is closed under conjugation and raising to $q$th powers.  
Let $\mathcal E_n$ be the set of Galois regular extensions $L/K$ inside a fixed separable closure $K^{\textup{sep}}$ of $K$ satisfying the following conditions.

\begin{itemize}

\item The extension $L/K$ is split completely at $\infty$.

\item The ramification of $L/K$ is supported at $D$ and a squarefree divisor $Z(L)$ of degree $n$ disjoint from $D$.

\item The extension $L/K$ is equipped with a group isomorphism $\varphi \colon \operatorname{Gal}(L/K) \to G$ with respect to which the ramification type over every point of $Z(L)$ lies in $R$.

\end{itemize}


Then we have cancellation in
$
\sum_{L \in \mathcal E_n} \mu(Z(L))
$
with the same power saving as in \cref{pr:discdisc}.


\end{prop}

\begin{proof}

We have
\[
\left| \sum_{L \in \mathcal E_n} \mu(Z(L)) \right| = \left| \sum_{L \in \mathcal E_n} (-1)^{\deg Z(L)} \operatorname{sign}(\operatorname{Frob}_q, \operatorname{supp }Z(L)) \right| = 
\left| \sum_{L \in \mathcal E_n} \operatorname{sign}(\operatorname{Frob}_q, \operatorname{supp }Z(L)) \right|.
\]

Let $U$ be the complement of $D \infty$ in $X$, or (by an abuse of notation) a lift of it to characteristic $0$.
It follows from \cite{el}*{Definitions 2.4.2, 2.4.5} that there exists a Hurwitz scheme $H_n$ with $H_n(\F_q) = \mathcal E_n$.
The scheme $H_n$ admits a Galois $S_n$-cover obtained by labeling the $n$ branch points, and this $S_n$-cover has a subcover $\tilde{H}_n \ra H_n$ corresponding to the index-$2$ subgroup $A_n$ of $S_n$. The \'etale double cover $\tilde{H}_n \ra H_n$ corresponds to a continuous homomorphism from $\pi_1^{\textup{\'et}}(H_n)$ to $\{\pm 1\} \subseteq \Q_\ell^\times$, which we view as a rank-$1$ local system over $\kappa = \Q_\ell$.

It follows from \cite{el}*{3.1.9, 3.1.10} that there exist a finite-dimensional vector space $W$ over $\kappa$, a (self-dual) coefficient system $(\kappa R_{-1}, W)$ for $B^n_{g,f}$, and lisse sheaves $\mathcal F_n$ punctually pure of weight zero on $\Conf^n U$ such that the analytification of the base change of $\mathcal F_n$ to $\C$ is the representation $W \otimes \kappa R_{-1}^{\otimes n}$ of $B^n_{g,f}$, the action of $B_n$ (viewed as a subgroup of $B^n_{g,f}$) on $\kappa R_{-1}^{\otimes n}$ is the one from \cref{BraidedToNaturoid} while its action on $W$ is trivial, and
\[
\sum_{L \in \mathcal E_n} \operatorname{sign}(\operatorname{Frob}_q, \operatorname{supp }Z(L)) = \sum_{x \in \Conf^n U (\F_q)} \operatorname{tr}(\operatorname{Frob}_q,\mathcal F_{\bar x}).
\]


From the Grothendieck--Lefschetz trace fomula, the triangle inequality, Deligne's Riemann Hypothesis, \cref{ComparisonCohomCor} and replacing $\kappa$ with its residue field, we get that
\[
\begin{split}
\left| \sum_{x \in \Conf^n U (\F_q)} \operatorname{tr}(\operatorname{Frob}_q,\mathcal F_{\bar x}) \right| &\leq \sum_{i=0}^{2n} |\operatorname{tr}(\Frob_q, H^i_c(\Conf^n U_{\bar{\F}_q}, \mathcal F_n))| \\
&\leq \sum_{i=0}^{2n} q^{i/2} \cdot \dim H^i_c(\Conf^n U_{\bar{\F}_q}, \mathcal F_n) \\
&\leq \sum_{i=0}^{2n} q^{i/2} \cdot \dim H_{2n-i}(B^n_{g,f}, W \otimes \kappa R_{-1}^{\otimes n}).
\end{split}
\]
The result then follows by invoking \cref{pr:0controlled}, \cref{th:0controlledvanishing}, and \cref{AaronJordanExplicated} as we did in the proof of \cref{pr:discdisc}.




\end{proof}

\begin{rem} The similarity in the statements (and proofs) of \cref{pr:discdisc} and \cref{pr:mobdisc} may lead one to wonder why a statement generalizing both is not given.  
The reason is that we do not know of coefficient systems $(\kappa R_\zeta, W)$ in the higher-genus setting with $\zeta$ a root of unity of order not dividing $2$.

\end{rem}

\begin{rem}
    \cref{pr:mobdisc} is likely not sharp.  For instance, in case $G = S_3$ and $R$ being the conjugacy class of transpositions, \cref{pr:dexamples} shows that the degree of $C(\kappa R_{-1})$ is $1$; so \cref{pr:mobdisc} says, for instance, that the sum of $\mu(\operatorname{rad} \disc L)$ as $L$ ranges over $S_3$-extensions of $\F_q(t)$ with radical discriminant $q^n$ grows with order at most $q^{(5/6)n}$ as $n \to \infty$ for sufficiently large (but fixed) $q$.  Casual computational experiments suggest that the sum of $\mu(\disc L)$ over cubic extensions of $\Q$ with discriminant at most $X$ behaves more like a random walk, and is thus typically of size $X^{1/2}$, much smaller than $X^{5/6}$. 
\end{rem}

\section{Primality of Conductors}

\begin{prop} \label{pr:wedgevanishing}

Let $R$ be a finite quandle, and let $\kappa$ be a field of characteristic other than $2$.  Write $d$ for the degree of $C(\kappa R_{-1})$, which is finite by \cref{pr:0controlled}.  Suppose $(\kappa R,W)$ is a coefficient system on a genus $g$ surface with $f$ punctures.
Then
$$
H_p(B_{g,f}^n, W \tensor \kappa R^{\tensor n} \tensor \wedge^m \operatorname{std}_n) = 0
$$
whenever $p < (m-d)/(d+2)$.

\end{prop}

\begin{rem}
    When $R$ is a conjugacy class in a finite group $G$ satisfying the nonsplitting condition in Definition~\ref{NonsplittingEVWinfinite}, $g=0$, and $\kappa = \mathbb Q$, this vanishing result with a smaller linear stable range in $m$ -- and indeed a vanishing result for representations of $S_n$ with $m$ or fewer boxes below the top row, not just $\wedge^m \operatorname{std}_n$ -- follows from the main results of \cite{hmw} on the representation stability of the homology of ordered Hurwitz spaces.  
\end{rem}

\begin{proof}

The braided vector space $\kappa R_\wedge$ carries a nontrivial grading; namely, if we write $\kappa R_\wedge = \kappa R \oplus \kappa R_{-1}$ as in \cref{TensoringWithKWedge}, we may take the grade-$0$ piece to be $\kappa R$ and the grade $1$ piece to be $\kappa R_{-1}$.  With this grading, $\deg \kappa R_\wedge = 1$.

Moreover, it follows from \cref{co:wedgedecomp} that the graded pieces of $\kappa R_\wedge^{\tensor n}$ are
$$
(\kappa R_\wedge^{\tensor n})_m
=
\kappa R^{\tensor n} \tensor \wedge^m \operatorname{Perm}_n
=
(\kappa R^{\tensor n} \tensor \wedge^m \operatorname{std}_n) \oplus (\kappa R^{\tensor n} \tensor \wedge^{m-1} \operatorname{std}_n).
$$


Then, by \cref{BraidedSumCoinvs}, the $m$th graded piece of the coinvariants ring $\oplus_n H_0(B_n, \kappa R_\wedge^{\tensor n})$ can be written as
$$
\bigoplus_{n=0}^\infty H_0(B_n, \kappa R_\wedge^{\tensor n})_m =
\bigoplus_{n=m}^\infty H_0(B_{n-m,m}, \kappa R^{\tensor n-m} \tensor \kappa R_{-1}^{\tensor m}).
$$
The latter coinvariants module is a quotient of the coinvariants by the subgroup $B_m \leq B_{n-m,m}$, and the coinvariants under $B_m$ can be written as
$$
\bigoplus_{n=m}^{\infty} H_0(B_m, \kappa R^{\tensor n-m} \tensor \kappa R_{-1}^{\tensor m})
=
\bigoplus_{n=m}^\infty \kappa R^{\tensor n-m} \tensor H_0(B_m, \kappa R_{-1}^{\tensor m}).
$$
Our choice of $d$ is such that $H_0(B_m, \kappa R_{-1}^{\tensor m})$ vanishes for all $m > d$, so the same holds for $H_0(B_n, \kappa R^{\tensor n}_{\wedge})_m$, namely $\deg C(\kappa R_\wedge) \leq d$.
As $H_p(B_{g,f}^n, W \tensor \kappa R^{\tensor n} \tensor \wedge^m \operatorname{std})$ is a direct summand of
$$
H_p(B_{g,f}^n, W \tensor \kappa R^{\tensor n} \tensor \wedge^m \operatorname{std}) \oplus H_p(B_{g,f}^n, W \tensor \kappa R^{\tensor n} \tensor \wedge^{m-1} \operatorname{std}) = H_p(B_{g,f}^n, W \tensor \kappa R_\wedge^{\tensor n})_m
$$
it follows from \cref{th:0controlledvanishing} that
$$
H_p(B_{g,f}^n, W \tensor \kappa R^{\tensor n} \tensor \wedge^m \operatorname{std}_n) = 0
$$
whenever $p < (m-d)/(d+2)$.
\end{proof}

It is likely that the methods of the last section of this paper building on the works of Landesman--Levy in conjunction with the inclusion of braided vector spaces from \cref{RackificationOneDimensional} and our study of inner automorphism groups of (direct products of) racks can be used to obtain an alternative proof of a slightly weakened form of \cref{pr:wedgevanishing} guaranteeing less vanishing.

\begin{lem} \label{StdWedgeYoungDiag}
    
For integers $0 \leq i \leq n-1$, the Young diagram corresponding to the irreducible representation $\wedge^i \operatorname{std}_n$ of $S_n$ is $(n-i,1,1, \dots ,1)$.  

\end{lem}

\begin{proof}

Follows by induction from Pieri's (branching) rule.
\end{proof}

\begin{prop} \label{LLpowerSavingSingleConj}

Let $G$ be a nontrivial finite group, and let $R$ be a conjugacy class that generates $G$.
Let $q$ be a prime power that is large enough in terms of $|G|$ and coprime to it.
Suppose that for every $r \in R$ we have $r^q \in R$.
For each $n$ in a growing sequence of positive integers choose a connected component $Z_n$ of $\mathsf{Hur}^n_{G,R} \times \overline{\F_q}$ that is defined over $\F_q$.
Then as $n \to \infty$ along this sequence of positive integers we have
\[
\# Z_n(\F_q) \sim \left(1 - \frac{1}{q} \right) \cdot q^n
\]
with a power saving error term.

\end{prop}

\begin{proof}

Abusing notation, we will at times view $Z_n$ as lifted to characteristic $0$.
We compare the Grothendieck--Lefschetz trace formula of $Z_n$ with that of $\Conf^n$.
The latter formula, in conjunction with \cite{Rosen}*{Proposition 2.3}, says that
\[
\left( 1 - \frac{1}{q}\right) \cdot q^n = \# \Conf^n(\F_q) = \sum_{i=0}^{2n} (-1)^i \operatorname{tr}(\Frob_q, H^i_c(\Conf^n \times \overline{\F_q}, \Q_\ell))
\]
and the former says that
\[
\#Z_n(\F_q) = \sum_{i=0}^{2n} (-1)^i \operatorname{tr}(\Frob_q, H^i_c(Z_n \times \overline{\F_q}, \Q_\ell)).
\]
As a result, using the triangle inequality we get
\[
\left|\#Z_n(\F_q) - \left(1 - \frac{1}{q}\right)q^n \right| \leq 
\sum_{i=0}^{2n} \left| \operatorname{tr}(\Frob_q, H^i_c(Z_n \times \overline{\F_q}, \Q_\ell)) - \operatorname{tr}(\Frob_q, H^i_c(\Conf^n \times \overline{\F_q}, \Q_\ell)) \right|.
\]

It follows from \cite{ll}*{Theorem 1.4.2}, and \cref{ComparisonCohomCor} that there exists $1 < \lambda < 2$ for which the morphism of schemes $\pi \colon Z_n \to \Conf^n$ induces, for all $n$ large enough, a $\Frob_q$-equivariant isomorphism of vector spaces
\[
H^i_c(Z_n \times \overline{\F_q}, \Q_\ell) \cong H^i_c(\Conf^n \times \overline{\F_q}, \Q_\ell)
\]
over $\Q_\ell$ whenever $i > \lambda n$.
In particular, the trace of $\Frob_q$ on these vector spaces is the same, so another application of the triangle inequality gives
\[
\left|\#Z_n(\F_q) - \left(1 - \frac{1}{q}\right)q^n \right| \leq 
\sum_{0 \leq i \leq \lambda n} \left| \operatorname{tr}(\Frob_q, H^i_c(Z_n \times \overline{\F_q}, \Q_\ell))\right| + \left| \operatorname{tr}(\Frob_q, H^i_c(\Conf^n \times \overline{\F_q}, \Q_\ell)) \right|.
\]

It follows from Deligne's Riemann Hypothesis that the above is at most
\[
\sum_{0 \leq i \leq \lambda n} q^{i/2} \cdot (\dim H^i_c(Z_n \times \overline{\F_q}, \Q_\ell) + \dim H^i_c(\Conf^n \times \overline{\F_q}, \Q_\ell)).
\]
Setting $\kappa = \F_\ell$ and using \cref{ComparisonCohomCor} we can bound this by 
\[
\sum_{0 \leq i \leq \lambda n} q^{i/2} \cdot (\dim_\kappa H_{2n-i}(B_n, \kappa R^n) + \dim_\kappa H_{2n-i}(B_n, \kappa)).
\]
Applying \cref{Salvetti} we get an upper bound of
\[
\sum_{0 \leq i \leq \lambda n} q^{i/2} \binom{n-1}{2n-i}(|R|^n + 1) \leq q^{\frac{\lambda}{2}n} \cdot (2|R|)^n 
\]
which gives us the desired power savings once $q$ is large enough because $\frac{\lambda}{2} < 1$.
\end{proof}

\begin{proof}[Proof of \cref{MainArithApp}]

As in \cref{HurwitzPointsGCovers} we identify $\mathcal E_q^R(G;n)$ with $\mathsf{Hur}_{G,R}^{n}(\F_q)$.
Let $Z_1, \dots, Z_m$  be the connected components of $\mathsf{Hur}_{G,R}^{n} \times \overline{\F_q}$ that are defined over $\F_q$.
We then have
\[
\mathcal E_q^R(G;n) = \mathsf{Hur}_{G,R}^{n}(\F_q) = \bigcup_{j=1}^m Z_j(\F_q).
\]
As a result we have
\[
\# \mathcal E_q^R(G;n) = \sum_{j=1}^m \# Z_j(\F_q)
\]
and
\[
\#\{L \in \mathcal E_q^R(G;n) : f_L \text{ is irreducible}\} = \sum_{j=1}^m \#\{L \in Z_j(\F_q) : f_L \text{ is irreducible}\}.
\]

It follows from \cite{lwzb}*{Corollary 12.9} that (the number of components) $m$, for large enough $n$, depends on $n$ only via the congruence class of $n$ (modulo $|G|^2$). 
It is thus sufficient to show for every $1 \leq j \leq m$ that as $n \to \infty$ we have
\begin{equation*}
\#\{L \in Z_j(\F_q) : f_L \text{ is irreducible}\} \sim \frac{q}{(q-1)n} \# Z_j(\F_q) 
\end{equation*}
with a power saving error term.
\cref{LLpowerSavingSingleConj} tells us that as $n \to \infty$ we have
\[
\frac{q}{(q-1)n}\#Z_j(\F_q) \sim \frac{q^n}{n}
\]
for every $1 \leq j \leq m$ with a power saving error term.
It would therefore be enough to show that
\begin{equation*}
\#\{L \in Z_j(\F_q) : f_L \text{ is irreducible}\} \sim \frac{q^n}{n}
\end{equation*}
as $n \to \infty$ with a power saving error term for every $1 \leq j \leq m$.

It follows from \cite{Rosen}*{Theorem 2.2}, \cref{VonMangoldtUsingWedges}, and the Grothendieck--Lefschetz trace formula that as $n \to \infty$ we have the asymptotic with power saving error term

\begin{equation*}
\begin{split}
\frac{q^n}{n} \sim \sum_{g \in \Conf^n(\F_q)} \mathbf{1}_{\textup{irr}}(g)  &= \frac{1}{n} \sum_{k=0}^{n-1} (-1)^k \ \sum_{g \in \Conf^n(\F_q)} \operatorname{tr}(\operatorname{Frob}_{\bar g}, \wedge^k \operatorname{std}_n) \\
&= \frac{1}{n} \sum_{k=0}^{n-1} \ \sum_{i=0}^{2n} (-1)^{i+k} \operatorname{tr}(\Frob_q, H^i_c(\Conf^n \times \overline{\F_q}, \wedge^k \operatorname{std}_n)).
\end{split}
\end{equation*}


\cref{VonMangoldtUsingWedges} tells us also that
\[
\#\{L \in Z_j(\F_q) : f_L \text{ is irreducible}\} = 
\sum_{L \in Z_j(\F_q)} \mathbf{1}_{\textup{irr}}(f_L) = \frac{1}{n} \sum_{k=0}^{n-1} (-1)^k \sum_{L \in Z_j(\F_q)} \operatorname{tr}(\Frob_{\overline{f_L}}, \wedge^k \operatorname{std}_n) 
\]
so using the morphism $\pi \colon Z_j \to \Conf^n$ and the Grothendieck--Lefschetz trace formula, we get that this equals
\[
\frac{1}{n} \sum_{k=0}^{n-1} (-1)^k \sum_{L \in Z_j(\F_q)} \operatorname{tr}(\Frob_{\bar{L}}, \pi^{-1} \wedge^k \operatorname{std}_n) = 
\frac{1}{n} \sum_{k=0}^{n-1} \sum_{i=0}^{2n} (-1)^{i+k} \operatorname{tr}(\Frob_q, H^i_c(Z_j \times \overline{\F_q}, \pi^{-1} \wedge^k \operatorname{std}_n)).
\]
In view of the triangle inequality, it would thus be sufficient to obtain a power saving bound on
\begin{equation} \label{ComparisonTraceFormulas}
\sum_{k=0}^{n-1} \sum_{i=0}^{2n} 
\left| \operatorname{tr}(\Frob_q, H^i_c(Z_j \times \overline{\F_q}, \pi^{-1} \wedge^k \operatorname{std}_n)) - 
\operatorname{tr}(\Frob_q, H^i_c(\Conf^n \times \overline{\F_q}, \wedge^k \operatorname{std}_n))
\right|.
\end{equation}

It follows from \cite{llNew}*{Theorem 1.3.5}, \cref{ComparisonCohomCor}, \cref{SingleConjugacyClassConnected}, and \cref{StdWedgeYoungDiag} that for every $0 < \alpha < 1$ there exists $1 < \lambda < 2$ for which the morphism of schemes $\pi \colon Z_n \to \Conf^n$ induces, for all $n$ large enough, a $\Frob_q$-equivariant isomorphism of vector spaces
\[
H^i_c(Z_j \times \overline{\F_q}, \pi^{-1} \wedge^k \operatorname{std}_n) \cong H^i_c(\Conf^n \times \overline{\F_q}, \wedge^k \operatorname{std}_n)
\]
over $\Q_\ell$ whenever $i > \lambda n$ and $k \leq \alpha n$.
Consequently, this range of $i$ and $k$ makes no contribution to \cref{ComparisonTraceFormulas}.

Put $\kappa = \F_\ell$.
By the projection formula, \cref{ComparisonCohomCor}, and the fact that the lisse sheaves $\pi_* \overline{\mathbb Q_\ell} \otimes \wedge^k \operatorname{std}_n$ are self-dual, we have
\begin{equation*}
\dim H^i_c(Z_j \times \overline{\F_q}, \pi^{-1} \wedge^k \operatorname{std}_n) = \dim H^i_c(\Conf^n \times \overline{\F_q}, \pi_* \overline{\mathbb Q_\ell} \otimes \wedge^k \operatorname{std}_n) \leq \dim H_{2n-i}(B_n, \kappa R^n \otimes \wedge^k \operatorname{std}_n)
\end{equation*}
and
\begin{equation*}
\dim H^i_c(\Conf^n \times \overline{\F_q}, \wedge^k \operatorname{std}_n) \leq \dim H_{2n-i}(B_n, \kappa \mathcal T_1^n \otimes \wedge^k \operatorname{std}_n)
\end{equation*}
which both vanish by \cref{pr:wedgevanishing} and \cref{ForgetFulGroupsRacks} whenever $i > 2n - (k-d)/(d+2)$.
In particular, for $k > \alpha n$ we get vanishing of cohomology in degrees $i$ exceeding
\[
2n - (k-d)/(d+2) < \left(2- \frac{\alpha}{d+2} \right)n + 1. 
\]
Therefore, after making $\lambda$ closer to $2$ if necessary, we see that there is no contribution to \cref{ComparisonTraceFormulas} from the range $i > \lambda n$ and $k > \alpha n$.
Combining this conclusion with the one from the previous paragraph, we infer that the cohomological degrees $i > \lambda n$ do not contribute to \cref{ComparisonTraceFormulas}.

Hence, in view of the triangle inequality, it remains to provide an upper bound for
\[
\sum_{k=0}^{n-1} \ \sum_{0 \leq i \leq \lambda n} 
\left| \operatorname{tr}(\Frob_q, H^i_c(Z_j \times \overline{\F_q}, \pi^{-1} \wedge^k \operatorname{std}_n))\right| + \left|
\operatorname{tr}(\Frob_q, H^i_c(\Conf^n \times \overline{\F_q}, \wedge^k \operatorname{std}_n))
\right|.
\]
The sheaves in question are punctually pure of weight $0$, so Deligne's Riemann Hypothesis furnishes us with an upper bound of
\[
\sum_{k=0}^{n-1} \ \sum_{0 \leq i \leq \lambda n} q^{i/2} \cdot \left( \dim H^i_c(Z_j \times \overline{\F_q}, \pi^{-1} \wedge^k \operatorname{std}_n) +
\dim H^i_c(\Conf^n \times \overline{\F_q}, \wedge^k \operatorname{std}_n)
\right).
\]
By \cref{Salvetti} this is bounded from above by
\[
\sum_{k=0}^{n-1} \ \sum_{0 \leq i \leq \lambda n} q^{i/2} \cdot \binom{n-1}{2n-i} \binom{n-1}{k} (|R|^n+1) \leq (4|R|)^n q^{\frac{\lambda}{2}n}.
\]
Since $\frac{\lambda}{2} < 1$ we obtain the required power saving once $q$ is large enough.
\end{proof}







\section{Sum of Legendre Symbols}

\begin{proof}[Proof of \cref{LegendreThmConductors}]

The contribution of $g=1$, $h = f_L$, and that of $h=1$, $g = f_L$, produce the main term on the right hand side.
Our task is therefore to obtain, as $n \to \infty$, a power saving in
\begin{equation*}
\sum_{L \in \mathcal E_q^R(G;n)} \sum_{i = 1}^{n-1} \ \sum_{\substack{(g,h) \in \Conf^{i,n-i}(\F_q) \\ gh = f_L}} \left( \frac{g}{h} \right).    
\end{equation*}
Taking an absolute value, exchanging the order of summation, and summing trivially over $i$, we are faced with obtaining a power saving in
\begin{equation*}
\sum_{i = 1}^{n-1} \left| \sum_{L \in \mathcal E_q^R(G;n)} \ \sum_{\substack{(g,h) \in \Conf^{i,n-i}(\F_q) \\ gh = f_L}} \left( \frac{g}{h} \right) \right|.    
\end{equation*}
It follows from the law of quadratic reciprocity \cite{Rosen}*{Theorem 3.5} that the contribution of any $i$ equals that of $n-i$, so it suffices to establish, for every $0 < i < n/2$, power saving cancellation for the sum in the absolute value above.

Let $\chi \colon \F_q^\times \to \C^\times$ be the nontrivial quadratic character.
It follows from \cref{ChiResultantLegendreSymbol}, \cref{HurwitzPointsGCovers}, \cref{PushforwardFromHurwitzSpaceTrace}, and \cref{TraceFunctionCharacterSheafResultantSumOne} that
\begin{equation*}
\begin{split}
&\sum_{L \in \mathcal E_q^R(G;n)}  \ \sum_{\substack{(g,h) \in \Conf^{i,n-i}(\F_q) \\ gh = f_L}} \left( \frac{g}{h} \right) = \\
&\sum_{f \in \Conf^n(\F_q)} \# \{L \in \mathcal E_q^R(G;n) : f_L = f \} \ \sum_{\substack{(g,h) \in \Conf^{i,n-i}(\F_q) \\ gh = f}}  \chi(\operatorname{Res}(h,g)) = \\
&\sum_{f \in \Conf^n(\F_q)} \operatorname{tr}(\operatorname{Frob}_{\bar f}, \pi_* \mathbb Q_\ell)
\cdot \operatorname{tr}(\operatorname{Frob}_{\bar f}, \tau_*^{(n-i)} \operatorname{Res}^{-1}\mathcal L_\chi) = \\
&\sum_{f \in \Conf^n(\F_q)} \operatorname{tr}(\operatorname{Frob}_{\bar f}, \pi_* \mathbb Q_\ell \otimes \tau_*^{(n-i)} \operatorname{Res}^{-1}\mathcal L_\chi).
\end{split}
\end{equation*}

Put $\kappa = \F_\ell$.
By the Grothendieck--Lefschetz trace formula, Deligne's Riemann Hypothesis, \cref{ComparisonCohomCor}, the notation of \cref{GeneratingNtuplesRack} and \cref{LegendreBraidVS}, and \cref{ArithmetizingPMbvs}, the absolute value of the above is at most
\[
\begin{split}
&\sum_{j=0}^{2n} \left| \operatorname{tr}(\Frob_q, H_c^j(\Conf^n \times \overline{\F_q}, \pi_* \mathbb Q_\ell \otimes \tau_*^{(n-i)} \operatorname{Res}^{-1}\mathcal L_\chi)) \right| \leq \\
&\sum_{j=0}^{2n} q^{j/2} \cdot \dim H_c^j(\Conf^n \times \overline{\F_q}, \pi_* \mathbb Q_\ell \otimes \tau_*^{(n-i)} \operatorname{Res}^{-1}\mathcal L_\chi) \leq \\
&\sum_{j=0}^{2n} q^{j/2} \cdot \dim H_{2n-j}(B_n, \kappa R^n_{\times} \otimes \operatorname{Ind}^{B_n}_{B_{n-i,i}} \operatorname{Span}_{\kappa} v^{\otimes (n-i)} \otimes \underline{v}^{\otimes i}).
\end{split}
\]

It suffices to show that there exists $\lambda > 0$ for which
\begin{equation} \label{SufficientVanishingLegendreCancellation}
H_{j}(B_n, \kappa R^n_{\times} \otimes \operatorname{Ind}^{B_n}_{B_{n-i,i}} \operatorname{Span}_\kappa v^{\otimes (n-i)} \otimes \underline{v}^{\otimes i}) = 0, \qquad 0 \leq j \leq \lambda n.     
\end{equation}
Indeed, then from \cref{Salvetti} we get the bound
\[
\begin{split}
&\sum_{0 \leq j \leq (2-\lambda)n} q^{j/2} \cdot
\binom{n-1}{2n-j} \cdot \dim (\kappa R^n_{\times} \otimes \operatorname{Ind}^{B_n}_{B_{n-i,i}} \operatorname{Span}_\kappa v^{\otimes (n-i)} \otimes \underline{v}^{\otimes i}) \leq \\
&\sum_{0 \leq j \leq (2-\lambda)n} q^{(1 - \lambda/2)n} \cdot \binom{n-1}{2n-j} \cdot \dim \kappa R^n_{\times} \cdot \dim(\operatorname{Ind}^{B_n}_{B_{n-i,i}} \operatorname{Span}_\kappa v^{\otimes (n-i)} \otimes \underline{v}^{\otimes i}) \leq \\
&q^{(1 - \lambda/2)n} \cdot 2^{n-1} \cdot |R|^n \cdot \binom{n}{i} \leq (4 |R| q^{1 - \lambda/2})^{n}
\end{split}
\]
which gives a power saving once $q$ is large enough.

We now turn to proving the vanishing in \cref{SufficientVanishingLegendreCancellation}.
Tensoring the inclusion from \cref{IndResChiIntoRack} with $\kappa R^n_\times$, and using the notation of \cref{MultiPartitionExmpAction}, we get an injective homomorphism of representations
\[
\kappa R^n_\times \otimes \operatorname{Ind}^{B_n}_{B_{n-i,i}} \operatorname{Span}_\kappa v^{\otimes (n-i)} \otimes \underline{v}^{\otimes i} \to \kappa R^n_\times \otimes \kappa \mathcal S_{\pm}(i,n-i) = \kappa [R^n_\times \times  \mathcal S_{\pm}(i,n-i)]
\]
of $B_n$ over $\kappa$.
We note that $R^n_\times \times  \mathcal S_{\pm}(i,n-i)$ is a subset of $R^n \times \mathcal S_{\pm}^n = (R \times \mathcal S_{\pm})^n$.
It follows from our assumption that $0< i < n$, \cref{SingleConjugacyClassConnected}, \cref{InnerAutJoyceQuandle}, and \cref{GoursatRacks},  that $R^n_\times \times  \mathcal S_{\pm}(i,n-i)$ is in fact a subset of $(R \times \mathcal S_{\pm})_\times^n$.

Therefore, it follows from \cite{ll}*{Theorem 1.4.2}, \cite{llNew}*{Example 1.4.7}, \cref{SynchQuandle} (or \cref{NonactingElementImpliesInnerProductive} and \cref{InnerProductiveImpliesSynchronized}) and \cref{QuotientJoyceQuandle}, that there exists $\lambda > 0$ such that the only contribution to the homology of $B_n$ in degrees $i \leq \lambda n$ with coefficients in $\kappa[R^n_\times \times \mathcal S_\pm(i,n-i)]$ comes from the (irreducible) constituents of the representation $\kappa[R^n_\times \times \mathcal S_\pm(i,n-i)]$ of $B_n$ appearing in the representation $\operatorname{Ind}_{B_{i,n-i}}^{B_n} \kappa$ of $B_n$.
Our task is thus to show that
\[
\operatorname{Hom}_{B_n}(\operatorname{Ind}^{B_n}_{B_{n-i,i}} \operatorname{Span}_\kappa v^{\otimes (n-i)} \otimes \underline{v}^{\otimes i}, \operatorname{Ind}_{B_{i,n-i}}^{B_n} \kappa) = 0,
\]
or equivalently, by the right adjointness of induction from a finite index subgroup, that 
\[
H_0(B_{i,n-i}, \operatorname{Ind}^{B_n}_{B_{n-i,i}} \operatorname{Span}_\kappa v^{\otimes (n-i)} \otimes \underline{v}^{\otimes i}) = 0.
\]
This follows from Mackey theory and the fact that this representation of the braid group does not factor via the symmetric group.
\end{proof}

\section{Acknowledgments}

We thank Will Sawin for suggesting \cref{SawinRack}, and for multiple helpful discussions.  We are also grateful to Aaron Landesman and Craig Westerland for useful discussions while the paper was underway.

Jordan Ellenberg's reserach is supported by National Science Foundation grant DMS-2301386 and by a Simons Fellowship.
Mark Shusterman's research is co-funded by the European Union (ERC, Function Fields, 101161909). Views and opinions expressed are however those of the author only and do not necessarily reflect those of the European Union or the European Research Council. Neither the European Union nor the granting authority can be held responsible for them.

\bibliographystyle{plain}
\bibliography{refs}

\end{document}